\documentclass[11pt,reqno]{amsart}
\usepackage{amsmath,amsfonts,amsthm,amssymb,color}
\usepackage[T1]{fontenc}
\usepackage{lmodern}
\usepackage{enumerate}
\usepackage{mathrsfs}
\usepackage{xcolor}
\usepackage{graphicx}
\usepackage{pdfsync}

\usepackage[left=1.1in, right=1.1in,  top=1.1in,bottom=1.1in]{geometry}
\usepackage{hyperref}   
\hypersetup{   
     colorlinks   = true, 
     citecolor    = blue, 
     linkcolor    = blue
}

\newcommand{\by}{\bar{Y}}

\newcommand{\ccf}{\textsc{cf}}

\newcommand{\dcf}{d_{\textsc{cf},p}}

\newcommand{\hte}{\hat{\vartheta}}

\newcommand{\iot}{\int_{0}^{t}}

\newcommand{\ist}{\int_{s}^{t}}

\newcommand{\ott}{[0,T]}

\newcommand{\1}{{\bf 1}}

\newcommand{\tv}{\textsc{tv}}

\newcommand{\xrn}{\xrightarrow}

\newcommand{\HZERO}{\mathbf{(H_0)}}
\newcommand{\HUNW}{\mathbf{(C_w)}}
\newcommand{\HUNS}{\mathbf{(C_s)}}
\newcommand{\HDEUX}{\mathbf{(I_w)}}
\newcommand{\HTROIS}{\mathbf{(I_s)}}
\newcommand{\distdomi}{{\cal D}_p}

\newcommand{\R}{\mathbb R}
\newcommand{\N}{\mathbb N}

\newcommand{\ES}{\mathbf{E}}
\newcommand{\ER}{\mathbb{R}}

\newcommand{\be}{\mathbf{E}}

\newcommand{\bp}{\mathbf{P}}
\newcommand{\PE}{\mathbf{P}}

\newcommand{\cac}{\mathcal C}

\newcommand{\cd}{\mathcal D}

\newcommand{\cf}{\mathcal F}

\newcommand{\cl}{\mathcal L}
\newcommand{\cm}{\mathcal M}
\newcommand{\cn}{\mathcal N}

\newcommand{\cw}{\mathcal W}

\newcommand{\al}{\alpha}

\newcommand{\ep}{\varepsilon}

\newcommand{\ga}{\gamma}

\newcommand{\ka}{\kappa}

\newcommand{\oom}{\Omega}

\newcommand{\si}{\sigma}
\newcommand{\te}{\vartheta}
\newcommand{\tte}{\Theta}
\newcommand{\vp}{\varphi}

\newcommand{\lp}{\left(}
\newcommand{\rp}{\right)}
\newcommand{\lc}{\left[}
\newcommand{\rc}{\right]}
\newcommand{\lcl}{\left\{}
\newcommand{\rcl}{\right\}}

\newcommand{\lla}{\left\langle}
\newcommand{\rra}{\right\rangle}

\newcommand{\cal}{\mathcal}

\newtheorem{theorem}{Theorem}[section]

\newtheorem{lemma}[theorem]{Lemma}
\newtheorem{notation}[theorem]{Notation}

\newtheorem{proposition}[theorem]{Proposition}

\theoremstyle{remark}
\newtheorem{remark}[theorem]{Remark}
\theoremstyle{remark}

\numberwithin{equation}{section}

\newcommand{\bean}{\begin{eqnarray*}}
\newcommand{\eean}{\end{eqnarray*}}
\newcommand{\ben}{\begin{enumerate}}
\newcommand{\een}{\end{enumerate}}
\newcommand{\beq}{\begin{equation}}
\newcommand{\eeq}{\end{equation}}

\begin{document}

\title[Drift estimation for SDEs with additive fractional noise]{A general drift estimation procedure for stochastic differential equations with additive fractional noise}

\date{\today}

\author[F. Panloup \and S. Tindel \and M. Varvenne]
{Fabien Panloup \and Samy Tindel \and Maylis Varvenne}

\address{Fabien Panloup: Larema, Universit\'e d'Angers, 2 Bd Lavoisier, 49045 Angers Cedex 01, France}	
								
\address{Samy Tindel: Department of Mathematics, Purdue University, 150 N. University Street, West Lafayette, IN 47907, United States}

\address{Maylis Varvenne : Institut de Math\'ematiques de Toulouse, Universit\'e Paul Sabatier, 118 Route de Narbonne, 31062 Toulouse Cedex 9, France}

\date{\today}

\begin{abstract}
In this paper we consider the drift estimation problem for a general differential equation driven by an additive multidimensional fractional Brownian motion, under ergodic assumptions on the drift coefficient. Our estimation procedure is based on the identification of the invariant measure, and we provide consistency results as well as some information about the convergence rate. We also give some examples of coefficients for which the identifiability assumption for the invariant measure is satisfied.
\end{abstract}

\thanks{S. Tindel is supported by the NSF grant  DMS-1613163}

\date{\today}

\maketitle

{
\hypersetup{linkcolor=black}
 \tableofcontents 
}

\section{Introduction}
\label{sec:intro}

Let $B$ be a $d$-dimensional fractional Brownian motion with Hurst parameter $H\in(0,1)$ defined on a complete probability space $(\oom,\cf,\bp)$. We recall  that $B$ is a centered Gaussian process. Its law is thus characterized by its covariance function, which is defined by
$$
\be \big[  B_t^i \, B_s^j \big] = \frac 12 \lp t^{2H} + s^{2H} - |t-s|^{2H}  \rp \, \1_{ \{ 0\}}(i-j),
\qquad s,t\in\R.
$$
The variance of the   increments of $B$ is  then given by
\begin{equation}\label{eq:var-increments-fbm}
\be \lc | B_t^i - B_s^i |^2 \rc = |t-s|^{2H}, \qquad s,t\in\R, \quad i=1,\ldots, m,
\end{equation}
 and this implies that  almost surely the fBm paths are $\gamma$-H\"{o}lder
  continuous for any $\gamma<H$. 
  
\smallskip

In this article, we will consider the following  $\R^d$-valued stochastic differential equation driven by $B$:
\begin{equation}\label{eq:sde}
Y_t=y_0+\iot  b_{\te_0}(Y_s) \, ds +  {\sigma B_t}
, \qquad t\in\ott.
\end{equation}
Here $y_0 \in \mathbb{R}^d$ is a given initial condition, $B=(B^{1}, \ldots, B^{d})$ is the aforementioned fractional Brownian motion (fBm), the unknown parameter 
$\te_0$ lies in a certain set $\tte$ which will be specified later on, $\{b_{\te}(\cdot), \, \te\in\tte \}$ is a known family of  drift coefficients with $b_{\te}(\cdot):\R^d\to\R^d $, and $\si$ is a {$d\times d$}-matrix which is supposed to be known. More precisely, we do not discuss here the problem of estimation of the diffusion parameter $\si$ and of the Hurst index $H$ (on this topic, see $e.g.$ \cite{coeur},  \cite{istas} or \cite{VT}). For the sake of simplicity, we also assume throughout the paper that $\si$ is invertible (on this topic, see Remark~\ref{rem:invertsig}). Our aim is to get an accurate estimation of $\te_{0}$ according to some discrete observations of $Y$.

\smallskip
{When $H=\frac{1}{2}$, $i.e.$ when $(Y_t)_{t\ge0}$ is a diffusion process, the drift estimation for solutions of stochastic differential equations is a widely investigated topic. Among many references devoted to parametric estimation procedures for discretely observed elliptic diffusions, let us quote the early contribution \cite{kessler97}. For adaptive and non parametric methods, one may refer to  \cite{comte-genon, dalalyan, hof, kutoyants,papavasiliou,strauch} and to the references therein. }
In the fractional setting, the estimation problem for the drift term in equation \eqref{eq:sde} has {also} received a lot of attention in the recent past (see e.g \cite{BEO,HN-foup,KL,LB,Rao,TV}). However, the following restrictions hold in all those contributions: 

\noindent
$\bullet$
The coefficient $b_{\te_0}(Y_s)$ is of the form $\te_{0} b(Y_s)$ or even $\te_{0} \,Y_s$ when Ornstein-Uhlenbeck processes are involved.

\noindent
$\bullet$
The observation is either in continuous time or a discretized version of continuous observations.

\noindent
$\bullet$
Rates of convergence of the estimators are not computed, a noticeable exception being the central limit theorems obtained in \cite{HN-foup}.

\noindent
Let us also mention the nonparametric method put forward in the interesting recent paper~\cite{CM}. The context in \cite{CM} is much more general than in the aforementioned references, but the estimation procedure is based on the observation of several paths of \eqref{eq:sde}. This makes its practical implementation delicate.

With these considerations in mind, let us recall that the article \cite{NT} proposed an estimator valid for a wide class of functions $b$ in \eqref{eq:sde}, directly based on discrete observations of the process. This estimator is obtained through a least square procedure which is easily implemented. It is fairly general, but still exhibits some gaps that 
we aim at filling in this paper.
{Indeed, on the one hand \cite{NT} only focuses on the case $H>1/2$.
This obviously facilitates many of the stochastic calculus manipulations invoked in order to analyze the convergence of estimators. On the other hand \cite{NT} crucially assumes that the drift $b_{\te}(\cdot)$ has a gradient shape. This is essential in order to identify the long-time behavior of certain stochastic integrals with respect to the underlying fBm (see Lemma 3.2 of \cite{NT} for details). 
Eventually \cite{NT} also required the following \textit{identifiability} assumption:} 
\begin{equation}\label{eq:cdt-identifiability-strong}
\mathbf{E} \lc |  b_{\te_0}( \bar{Y}_0)|^2 \rc = \mathbf{E} \lc|  b_{\te}( \bar{Y}_0) |^{2} \rc
\quad\mbox{iff}\quad
\te=\te_0.\\
\end{equation}
{Even though \eqref{eq:cdt-identifiability-strong} can be considered as a quite standard hypothesis in an ergodic framework (see e.g \cite{kessler97} for a similar requirement), asking the class of models to be identifiable in this sense is mostly adapted to the drifts of the form $b_\te(\cdot)=\te h(\cdot)$. It may be restrictive or at least, difficult to check in practice.}\\

{In this paper, we thus aim at filling some of the aforementioned gaps. Our initial objective is to alleviate the identifiability condition. More precisely, instead of assuming that the model can be identified through the integral of one very specific function with respect to the invariant distribution, we will now just suppose that the model is identifiable by its invariant distribution, $i.e.$ that two models are equal if and only if they have the same stationary distribution. We also get rid of the restriction on $H$ and on the gradient shape of $b_\te$ by considering an approach which essentially avoids to use stochastic calculus. In other words, our consistence and rate of convergence results will be mainly based on the study of the long-time behavior of distances between paths built with the same noise. In the additive noise case we are handling, those distances have the advantage to be regular. Eventually, even if our ergodic type assumptions are similar to those in \cite{NT} (and called $\HUNS$ in the current paper),  we also obtain a (weak) consistence result under a weaker assumption (called $\HUNW$ below), where contractivity of the drift is only assumed out of a compact set. 
}
%
%
%
%
%
%

\section{Main results}

In this section we will first give some general notation which will be used throughout the paper. Then we will specify our assumptions on the coefficients of \eqref{eq:sde} and describe the estimator we are considering. Eventually we give our almost sure convergence result as well as the convergence rate we have been able to obtain.

\subsection{Notation}\label{subsection:notation}
We consider the set of parameters $\tte$ as a subset of $\R^{q}$ for $q\ge 1$.
Let $f:\R^d\times\tte\to\R$ be a $\cac^{p_1,p_2}$ function, where $p_1,p_2$ are two integers greater than 1. Then for any $p\le p_{1}$ and any tuple $(i_1,\ldots i_{p})\in\{1,\ldots,d\}^{p}$, we set $\partial^{i_1\ldots i_{p}}_x f$ for $\frac{\partial^{p} f}{\partial x_{i_1}\ldots \partial x_{i_{p}}}$. 
 Analogously, for $p\le p_{2}$ we use  
the notation $\partial^{i_1\ldots i_{p}}_{\vartheta} f$ for $\frac{\partial^{p} f}{\partial \vartheta_{i_1}\ldots \partial \vartheta_{i_{p}}}$
whenever $(i_1,\ldots i_{p})\in\{1,\ldots,q\}^{p}$. Moreover, we will write 
$\partial_x  f$ (resp. $\partial_{\vartheta} f$) for the Jacobi-matrices
$(\partial_{x_1} f, \ldots, \partial_{x_d}f)$  
(resp. $(\partial_{\te_1} f, \ldots, \partial_{\te_q}f) $).

Let ${\cal M}_1(\R^d)$ denotes the set of probability measures on $\R^d$. 
We say that  $d$ is a distance on  ${\cal M}_1(\R^d)$ if it metrizes its usual topology, namely the weak convergence topology. Among those distances we will consider the $p$-Wasserstein distance, which is defined as follows: for every $\nu,\mu\in{\cal M}_1(\R^d)$, we introduce the set $\cac(\nu,\mu)$ of couplings between $\nu$ and $\mu$, that is
\begin{equation}\label{eq:def-couplings}
\cac(\nu,\mu)
=
\lcl  (X,Y); \,  {\cal L}(X)=\nu, \, {\cal L}(Y)=\mu \rcl.
\end{equation}
Then the $p$-Wasserstein distance is written as
\begin{equation}\label{eq:defwp}
{\cal W}_p(\nu,\mu)=\inf\lcl \ES[|X-Y|^p]^{\frac{1}{p}} ; \,  (X,Y)\in\cac(\nu,\mu) \rcl .
\end{equation}
\begin{remark} The distance ${\cal W}_1$ can also be represented as
\begin{equation}\label{eq:W1_Lip}
{\cal W}_1(\nu,\mu)=\sup\{|\nu(h)-\mu(h)|;~\|h\|_{\rm Lip}\leq1\}.
\end{equation}
\end{remark}

We will denote by $\distdomi$ the set of distances on ${\cal M}_1(\R^d)$ dominated by the $p$-Wasserstein distance for a given $p>0$. Namely, we set
\beq\label{set_distances}
\distdomi
:=
\{\textnormal{distances $d$ on ${\cal M}_1(\R^d)$} ; \,  \exists \;c>0 \text{ such that } \forall \nu,\mu\in {\cal M}_1(\R^d), d(\nu,\mu)\le c {\cal W}_p(\nu,\mu)\}
\eeq
In particular, a distance $d$ that belongs to $\distdomi$ induces a weaker topology than the  $p$-Wasserstein distance.
When necessary in some of the next results (or in the numerical experimentations), we will introduce specific distances which belong to $\distdomi$.

\begin{remark}
The family of $p$-Wasserstein distances obviously provides examples of distances in $\cd_{p}$. The Fortet-Mourier distance (see e.g \cite[Chapter 6]{villani2008optimal}), defined by 
\begin{equation*}
d_{\textsc{FM}}(\nu,\mu)
=
\sup\lcl  |\be[h(X)]-\be[h(Y)] | ; \, \text{ where } X\sim\mu, \, Y\sim\nu, \,  
\text{ and }\|h\|_{{\rm Lip}}\le1, \,  \|h\|_{\infty}\le1 \rcl,
\end{equation*}
is also easily seen to be an element of $\cd_{1}$ thanks to \eqref{eq:W1_Lip}. In this article we shall work with the distances $d_{\ccf,p}$ and $d_{s}$ introduced below in Section \ref{sec:rateofconvergence}, which are trivially proved to sit in $\cd_{1}$ (due to relation \eqref{eq:W1_Lip}).
\end{remark}

\subsection{Assumptions}
Before we proceed to a specific statement of our estimator, let us describe the assumptions under which we shall work. We start by  a  standard hypothesis on the parameter set $\tte$, which is supposed to be a compact set.

\noindent $\HZERO:$ The set $\tte$ is compactly embedded in $\R^q$ for a given $q\ge 1$.

Next we recall that our drift estimators rely on the invariant measure for the solution of equation \eqref{eq:sde}. The existence and uniqueness of this invariant measure is usually obtained under some coercivity assumptions on the drift $b$. In the current paper we will distinguish between two notions of coercivity, respectively named weak and strong and denoted by $\mathbf{(C_w)}$ and $\mathbf{(C_s)}$. The weak assumption can be summarized as follows. 

\noindent $\HUNW:$
 We have $b \in \mathcal{C}^{1,1} (\mathbb{R}^d \times \tte; \mathbb{R}^d)$ and 
there exist constants $\alpha, \beta, C, L>0$ and $r \in \mathbb{N}$ such that:

\noindent
\emph{(i)} For every $x,y\in\R^{d}$ and $\te\in\tte$ we have
\begin{equation}\label{eq:inward-weak}
\lla  b_{\te}(x)-b_{\te}(y), \, x-y\rra \le \beta - \alpha |x-y|^2 \quad \text{ and } \quad |b_{\te}(x)-b_{\te}(y)|\le L|x-y|
\end{equation}

\noindent
\emph{(ii)} For every $x\in\R^{d}$ and $\te\in\tte$ the following growth bound is satisfied:
\begin{equation}\label{eq:hyp-bnd-deriv-b}
|\partial_{\vartheta} b_{\te}(x)| \le C \lp  1+|x|^r \rp.
\end{equation}

The main part \eqref{eq:inward-weak} of assumption $\HUNW$ states that the coefficient $b$ is inward looking, except maybe on a compact set which is a neighborhood of 0. We now state the strong assumption $\HUNS$, which specifies that $b$ should be inward looking everywhere, and can be expressed as a particular case of $\HUNW$.

\noindent $\HUNS:$ Assumption $\HUNW$ holds with $\beta=0$. 

\smallskip

As mentioned above, Hypothesis $\HUNW$ combined with the invertibility of $\si$ (and therefore $\HUNS$) classically involves (see  $e.g.$ \cite{Ha}) the existence of a unique stationary solution for the solution of the following equation for any $\te\in\tte$:
\begin{equation}\label{eq:sdetheta}
dY_t^\te= b_{\te}(Y^\te_t) \, dt +  \si dB_t , \qquad t\in\ott.
\end{equation}
Notice that the system \eqref{eq:sdetheta} is identical to our original equation \eqref{eq:sde}. However, let us notice that  the fBm is unobserved. 
 Moreover, the uniqueness of the stationary measure must be understood in a weak sense. Namely there exists a unique distribution $\PE^{\te}$ on ${\cal C}([0,\infty),\R^d)$ such that if $(\by_t^{\te})$ denotes a process with distribution $\PE^{\te}$, then $(\by_t^{\te})_{t\ge0}$ is a stationary solution to \eqref{eq:sdetheta}, $i.e.$ shift-invariant (when one considers its canonical version). We denote by $\nu_{\te}$ the distribution of $\by^{\te}$ at any instant $t\ge0$, that is
 \begin{equation}\label{eq:defnute}
  \nu_{\te}={\cal L}(\by_0^{\te}).
 \end{equation}
 
 \begin{remark}\label{rem:invariant_dist}
Note that in this non-Markovian setting, $\nu_{\te}$ is not exactly the invariant distribution. More precisely, owing to \cite{Ha}, one can embed \eqref{eq:sdetheta} into an infinite-dimensional Markovian structure which allows the construction of an adapted ergodic theory. An invariant distribution $\bar{\nu}_{\te}$  is then defined on this enlarged structure.  Without going into the details, one can just say that in this theory, the probability $\nu_{\te}$ can be retrieved as a marginal of the ``true'' invariant distribution. In the sequel, we will thus talk about \textit{marginal invariant distribution} $\nu_{\te}$.
\end{remark}
{\begin{remark}\label{rem:invertsig} As mentioned before, the invertibility assumption on $\si$ combined with $\HUNW$ ensures uniqueness of the invariant distribution. However, even though this hypothesis is of first importance  under $\HUNW$ (in order to use irreducibility-type arguments), it could be entirely removed under $\HUNS$. Actually, in this case, 
the contraction assumption implies that two solutions of \eqref{eq:sdetheta} driven by the same fBm but starting from  different initial conditions come together at $\infty$, $a.s.$ and in $L^2$, which classically involves uniqueness (see $e.g.$  \cite[Lemma 3$(ii)$]{CP} for details). This remark also holds for the Euler scheme \eqref{eq:Euler_scheme} introduced in the next section but for sufficiently small step $\gamma$ (see again \cite[Lemma 3$(ii)$]{CP} for details).
\end{remark}
}
As said previously, we shall obtain our drift estimators through the analysis of the marginal invariant distribution $\nu_{\te}$ defined by \eqref{eq:defnute}. If we want this strategy to be successful, it is natural to assume that $\nu_{\te}$ characterizes $\te$. We thus label this hypothesis as follows.

\smallskip

\noindent $\HDEUX$: For all $\te\in\tte$, we have $\nu_\te=\nu_{\te_0}$  iff  $\te = \te_{0}$.

\noindent 
It is worth noticing that if $d$ denotes a distance on ${\cal M}_1(\R^d)$,  then one can recast $\HDEUX$ as: 
\begin{equation}\label{eq:I_w}
d(\nu_\te,\nu_{\te_0})=0
\quad\text{iff}\quad
\te = \te_{0}.
\end{equation}
We shall use this characterization in order to construct the estimator $\hte$ (see \eqref{eq:hat-theta-as-argmin} below). Also notice that $\HDEUX$ refers to a ``weak'' identifiability condition, which will be resorted to in order to derive the consistency of our estimator $\hte$. In contrast, the following ``strong'' identifiability assumption $\HTROIS$ defined for a given distance $d$ on ${\cal M}_1(\R^d)$ will be used to get rates of convergence.

\smallskip

\noindent $\HTROIS$ There exists a constant $C>0$ and a parameter $\varsigma\in(0,1]$ such that $\forall \te\in\tte$,
\begin{equation}\label{eq:I_s}
d(\nu_\te,\nu_{\te_0})\ge C|\te-\te_0|^\varsigma.
\end{equation}

\begin{remark} 
We will construct a class of equations, basically obtained as perturbations of Langevin type equations, for which our assumptions $\HDEUX$ and $\HTROIS$ are satisfied. See Section~\ref{sec:identif-inv-measure} below.
\end{remark}

\subsection{Statistical setting and construction of the estimator}
We wish to construct an estimator based on discrete observations. In this context, the simplest situation (which will mostly prevail in the paper) is to assume that the solution $(Y_t)_{t\ge0}$ of \eqref{eq:sdetheta} is discretely observed at some instants $\{t_k;\, 0\le k \le n\}$, with $t_{k+1}-t_{k}=\kappa$ for a  given time step  $\kappa > 0$. 
Under $\mathbf{(C_w)}$, it can be shown (see Proposition \ref{prop:ergodicSDE} below) that 
\begin{equation}\label{eq:cvgce-empirical-msr}
\frac{1}{n}\sum_{k=0}^{n-1} \delta_{Y_{t_k}}\overset{n\rightarrow+\infty}{\Longrightarrow} \nu_{\te_0}\quad a.s,
\end{equation}
where $\Longrightarrow$ stands for the weak convergence of probability measures in $\ER^d$. With this convergence in mind, the heart of our estimation method is then the following observation: under the identifiability assumption $\HDEUX$, the most natural way to estimate $\te_0$ is to consider 
\begin{equation}\label{eq:hat-theta-as-argmin}
\hat{\te}=\underset{\te\in\Theta}{\rm argmin}\; d\lp\frac{1}{n}\sum_{k=0}^{n-1} \delta_{Y_{t_k}},\nu_\te\rp,
\end{equation}
where $d$ is a given (arbitrary) distance on ${\cal M}_1(\ER^d)$. However, in spite of the fact that our formula~\eqref{eq:hat-theta-as-argmin} is simple enough, it is also easily understood that $\nu_\te$ is far from being explicitly known (except in some very particular cases such as  the Ornstein-Uhlenbeck process). In this paper, we propose to circumvent this difficulty by considering some estimators based on numerical approximations of $\nu_\te$. 

Specifically, the numerical approximations we will resort to are built 
through an Euler-type discretization of the stochastic process $Y^\te$ solution to \eqref{eq:sdetheta}.
Namely, let $(s_k)_{k\ge0}$ be an increasing sequence of numbers such that $s_0=0$,  and $\lim_{k\to\infty}s_k=+\infty$.
 The Euler-Maruyama scheme $Z^{\te}$ is then recursively defined by $Z^{\te}_{0}=z_{0}\in\ER^d$ and:
{
\begin{equation}\label{eq:Euler_scheme}
\text{For all } k\ge0,\quad  Z^{\te}_{s_{k+1}}
=Z^{\te}_{s_{k}}+
(s_{k+1}-s_k) b_{\te}(Z_{s_{k}}^{\te})  
+ \si \,  (\tilde{B}_{s_{k+1}}-\tilde{B}_{s_{k}}),
\end{equation}
}
where $\tilde{B}$ is a (simulated) $m$-dimensional fractional Brownian motion which is a priori different from the driving process $B$ in equation~\eqref{eq:sdetheta} (since $B$ is unobserved).
When $s_k=k\ga$ for a given $\gamma>0$, we say that the Euler scheme is a \textit{constant step} sequence and we denote it by $Z^{\te,\ga}$. 
When  $\gamma_k=s_k-s_{k-1}$ is  a non-increasing sequence such that $\ga_k\rightarrow0$ as $k\rightarrow+\infty$, the Euler scheme will be called \textit{decreasing step} Euler scheme. We will work with these two types of schemes in the sequel.

\begin{remark}
In practice it is natural to set  $Z^{\te}_{0}=Y_{0}$ where $Y_0$ is the first observation of the process $(Y_t)_{t\ge0}$.  Let us also remark that in the sequel, for notational sake, one usually denotes by $B$ the fBm related to the Euler scheme $Z^{\te}$. However, the reader has to keep in mind that the fact that the fBms in \eqref{eq:sdetheta} and in \eqref{eq:Euler_scheme} are different. This certainly prevents us from any pathwise comparison between the observed process and the simulated one.
 \end{remark}

\begin{remark} We refer to Section \ref{sec:simu} for background on the simulation of the increments of the fBm.
\end{remark}

Let us now give an explicit expression for the estimator we are considering in this article. We will focus on the constant step setting in \eqref{eq:Euler_scheme} for sake of simplicity. Observe that under $\mathbf{(C_w)}$, it can be shown (see Proposition \ref{prop:ergodicEuler} below) that
\begin{equation}\label{eq:convschemeeulerma}
\frac{1}{N}\sum_{k=0}^{N-1} \delta_{Z_{k\ga}^{\te,\ga}}\overset{n\rightarrow+\infty}{\Longrightarrow} \nu_{\te}^\gamma\quad a.s.
\end{equation}
where $\nu_\te^\gamma$ denotes the unique marginal invariant distribution of the Euler scheme $Z^{\te}$. By marginal, we mean again that $Z^{\te,\ga}$ can be endowed with a Feller infinite-dimensional Markov structure which admits a unique invariant distribution under $\mathbf{(C_w)}$ (see \cite{Var2} for details). The first marginal of this invariant distribution is  $\nu_\te^\gamma$.
Similarly to what we proposed in \eqref{eq:hat-theta-as-argmin}, such a result suggests to define our estimator for $\te$ as
\begin{equation}\label{eq:firstdefesti}
\hat{\te}_{N,n,\ga} = \underset{\te\in\Theta}{\rm argmin}\ d\lp\frac{1}{n}\sum_{k=0}^{n-1}\delta_{Y_{t_k}},~\frac{1}{N}\sum_{k=0}^{N-1}\delta_{Z_{k\ga}^{\te,\ga}}\rp,
\end{equation}
where $d$ is again a distance on ${\cal M}_1(\ER^d)$.
Note that in the decreasing step case analogous constructions may be carried out, and will be introduced later. Let us also remark that relation~\eqref{eq:firstdefesti} only involves one Euler scheme path, which is relevant for numerical implementations.

We are now in a position to state our main results. We divide the presentation in two parts. In the next section, we focus on strong consistence results related to the family $\{\hat{\te}_{N,n,\ga}; \, N\ge 1,n\ge 1,\ga>0\}$ defined by \eqref{eq:firstdefesti}, as well as its decreasing step counterpart.  Then Section \ref{sec:rateofconvergence} is dedicated to the rate of convergence of the estimator $\hat{\te}_{N,n,\ga}$. In particular, 
this second part will involve concentration results related to the SDE and to its Euler discretization. 

\subsection{Main consistency results}\label{subsection:main_consistency_results}
We begin with a first result involving the weak assumption~$\HUNW$, which requires to discretize the set $\tte$ in the following sense. According to our hypothesis $\HZERO$, the set $\tte$ is compact in $\R^q$. Therefore the Borel-Lebesgue property gives us the existence, for every $\varepsilon>0$, of $M_\varepsilon\in\N$ and $(\te^{(\varepsilon)}_i)_{1\leqslant i\leqslant M_\varepsilon}\in\tte^{M_\varepsilon}$ such that $\tte\subset\bigcup_{i=1}^{M_\varepsilon}B(\te^{(\varepsilon)}_i,\varepsilon)$. Thanks to this property, we define the following discretization for all $\varepsilon>0$ and $\te\in\tte$:
\begin{equation}\label{eq:discretize-parameter-space}
\te^{(\varepsilon)}:=\underset{\te'\in\{\te^{(\varepsilon)}_i\}}{\rm argmin}~|\te'-\te|.
\end{equation}
With this notation in hand, we can now state our first consistency theorem.

\begin{theorem}\label{thm:cvgce-estimator-dicrete-theta}
Assume $\HZERO$, $\HUNW$ and $\HDEUX$. Let $p$ be a strictly positive real number and consider a distance $d$ on ${\cal M}_1(\ER^d)$ which belongs to $\distdomi$ (recall that $\distdomi$ is defined by \eqref{set_distances}). 
We consider the family $\{\hat{\te}_{N,n,\ga}^{(\varepsilon)}; \, N\ge 1,n\ge 1,\ga>0, \ep>0\}$ defined by
\begin{equation}\label{def:hatthetanngamma}
\hat{\te}_{N,n,\ga}^{(\varepsilon)} = \underset{\te\in\tte^{(\varepsilon)}}{\rm argmin}\ d\lp\frac{1}{n}\sum_{k=0}^{n-1}\delta_{Y_{t_k}},\frac{1}{N}\sum_{k=0}^{N-1}\delta_{Z_{s_k}^{\te,\ga}}\rp, \quad N,n\in\mathbb{N},\;\gamma>0,\;\varepsilon>0
\end{equation}
where $\tte^{(\varepsilon)}:=\{\te^{(\varepsilon)}_i~;~1\leqslant i\leqslant M_\varepsilon\}$. Then $\hat{\te}_{N,n,\ga}^{(\varepsilon)}$ is a strong consistent estimator of $\te_0$. Specifically, we have
$$
\lim_{\varepsilon\to0}~\lim_{\ga\rightarrow0}~\lim_{N,n\rightarrow+\infty}
\hat{\te}_{N,n,\ga}^{(\varepsilon)}=\te_0\quad \text{a.s.}
$$
\end{theorem}

Let us remark that the discretization of $\tte$ given by \eqref{eq:discretize-parameter-space} is needed to get strong consistency, due to the fact that under $\HUNW$ we loose uniformity with respect to $\te$ in some of our convergence results. For instance, $\HUNW$ only warranties the simple convergence of $d(\nu^\ga_\te,\nu_\te)$ to $0$ as $\ga\rightarrow0$ (see Proposition \ref{prop:ergodicEuler}). The proof of Theorem \ref{thm:cvgce-estimator-dicrete-theta}
is achieved in Section \ref{subsection:proofth}.\smallskip

We now turn to our main estimator defined in \eqref{eq:firstdefesti}. The proof of the theorem below is detailed in Section \ref{subsection:proofth}.

\begin{theorem}\label{thm:cvgce-esti1} 
Assume $\HZERO$, $\HUNS$ and $\HDEUX$. As in Theorem \ref{thm:cvgce-estimator-dicrete-theta}, let $p$ be a strictly positive real number and consider a distance $d$ on ${\cal M}_1(\ER^d)$ which belongs to $\distdomi$. 
Then the family $\{\hat{\te}_{N,n,\ga}; \, N\ge 1,n\ge 1,\ga>0\}$  defined by \eqref{eq:firstdefesti} is a strong consistent estimator of $\te_0$ in the following sense:
$$
\lim_{\gamma\rightarrow0}\lim_{N,n\rightarrow+\infty}\hat{\te}_{N,n,\ga}=\te_0\quad \text{a.s.}
$$
\end{theorem}

We close this section with a result concerning the approximation of invariant measures by an Euler scheme with decreasing time step. Namely we consider an approximation scheme denoted also by $Z^{\te}$, which is defined similarly to $Z^{\te,\ga}$ in \eqref{eq:Euler_scheme} except for the fact that the sequence $(s_k)_{k\ge0}$ satisfies:
\begin{equation}\label{eq:def-steps-sk}
s_{k+1}-s_k=\gamma_{k+1}, \quad k\ge0,
\end{equation}
where $(\gamma_k)_{k\ge1}$ is a non-increasing sequence of positive numbers such that 
\begin{equation}\label{eq:condgamma}
\sum_{k\ge1}\gamma_k=+\infty\quad \textnormal{and}\quad \lim_{k\rightarrow+\infty}\gamma_k=0.
\end{equation}
The convergence theorem we obtain in the decreasing step case is the following. As for Theorem~\ref{thm:cvgce-esti1} , it is achieved under the strong coercivity assumption $\HUNS$. Its proof is developped in Section~\ref{subsection:proofth}.

\begin{theorem}\label{thm:cvgce-esti2}
Assume $\HZERO$, $\HUNS$ and $\HDEUX$. Let $p\ge 2$ and consider a distance $d$ on ${\cal M}_1(\ER^d)$ which belongs to $\distdomi$. Let $(\gamma_k)_{k\ge1}$ be a non-increasing sequence of positive numbers satisfying condition \eqref{eq:condgamma} and the technical condition \eqref{eq:hyp_gamma}.
 Denote by $Z^{\te}$ the related Euler scheme given by \eqref{eq:Euler_scheme}. We consider an estimator $(\hat{\te}_{N,n})_{N,n}$  defined as
\begin{equation}\label{def:hatthetanngamman}
\hat{\te}_{N,n} = \underset{\te\in\Theta}{\rm argmin}\ d\lp\frac{1}{n}\sum_{k=0}^{n-1}\delta_{Y_{t_k}},\frac{1}{s_N}\sum_{k=0}^{N-1}\gamma_{k+1}\delta_{Z_{s_k}^{\te}}\rp, \quad N\in\mathbb{N}.
\end{equation}
Then,  $(\hat{\te}_{N,n})_{N,n}$  is a strong consistent estimator of $\te_0$, namely:
$$\lim_{N,n\rightarrow+\infty}\hat{\te}_{N,n}=\te_0\quad \text{a.s.}
$$
\end{theorem}

\begin{remark} The technical condition  \eqref{eq:hyp_gamma}  ($\sum_{k\ge 1}\gamma_k^{p'H+1} s_k^{-1}<+\infty$ for a given $p'\ge p$) is true in a very large setting.
For instance, it can be checked that this is satisfied for any polynomial step sequence : $\gamma_k= \gamma  k^{-\rho}$ with $\rho\in(0,1]$ and $\gamma\in \ER_+^*$, but also 
 for less decreasing sequences such as $\gamma_k=\gamma (\log k)^{-1}$. However, this is not true in full generality (the condition does not hold when $\gamma_k=(\log(\log k))^{-1}$ for instance).
\end{remark}

\subsection{Rate of convergence}\label{sec:rateofconvergence}

Under our strong identifiability condition $\HTROIS$, we will be able to get a rate of convergence for some of our estimators. In order to carry out this task, we shall assume that condition $\HTROIS$ is verified for some specific distances on probability measures called respectively $d_{\ccf,p}$ and $d_s$. These distances are defined in the following way:
\begin{itemize}
\item[$(i)$] Let $X$ and $Y$ be $\R^d$-valued random variables and $p>(\frac{d}{2}\vee1)$. We consider the integrable kernel $g_p(\xi):=c_p(1+|\xi|^2)^{-p}$ where $c_p:=\left(\int_{\R^d}(1+|\xi|^2)^{-p}d\xi\right)^{-1}$. Then the $d_{\ccf,p}$ distance between $\mathcal{L}(X)$ and $\mathcal{L}(Y)$ is defined by:
\begin{equation}\label{eq:def_CFp}
d_{\ccf,p}(\cl(X),\cl(Y)):=\left(\int_{\R^d}\left(\be[e^{i\langle \xi,X\rangle}]-\be[e^{i\langle \xi,Y\rangle}]\right)^2g_p(\xi)d\xi\right)^{1/2}.
\end{equation}
\item[$(ii)$] Let $\{f_i~;~i\geq1\}$ be a family of $\mathcal{C}_b^1$, supposed to be dense in the space $\mathcal{C}_b^0$ of continuous and bounded functions and decreasing to $0$ at $\infty$. Consider two probability measures $\nu$ and $\mu$ in $\cm_1(\R^d)$. Then the distance $d_s$ between $\nu$ and $\mu$ is defined as:
\begin{equation}\label{eq:def_ds}
d_s(\nu,\mu):=\sum_{i=0}^{+\infty}2^{-i}(|\nu(f_i)-\mu(f_i)|\wedge1).
\end{equation}
\end{itemize}

\begin{remark} It is readily checked that both $d_{\ccf,p}$ and $d_s$ metrize the convergence in law (their induced topology are in fact exactly the one induced by the convergence in law). The distance $d_{\ccf,p}$ is technically convenient for our purpose and close in spirit to the smooth Wasserstein distance invoked in the Stein method literature (see e.g \cite{AMPS}). The distance $d_s$ is called weak-$\star$ distance in \cite{villani2008optimal} and also used in \cite{Xiong} for filtering problems. Also notice that both $d_{\ccf,p}$ and $d_s$ are elements of $\mathcal{D}_1$ where $\mathcal{D}_1$ is defined by \eqref{set_distances}.
\end{remark}

With the distances $d_{\ccf,p}$ and $d_s$ in hand, our main result about rates of convergence is the following:
\begin{theorem}\label{thm:rate_of_convergence}
Assume $\HZERO$, $\HUNS$ and $\HTROIS$ hold true, where \eqref{eq:I_s} in hypothesis $\HTROIS$ is considered for $d=d_s$ or $d=d_{\ccf,p}$ with $p>(q+d)/2$ and $\varsigma=2/q$ for a given $q\geq2$. Let $\hat{\te}_{N,n,\ga}$ be the estimator given by \eqref{eq:firstdefesti}. Then, we get the following rate of convergence: there exists $C_q>0$ such that
\begin{equation}
\ES\left[|\hat{\te}_{N,n,\ga}-\te_0|^2\right]
\leq C_q\left(n^{-\frac{q}{2}(2-(2H\vee1))}+\ga^{qH}+T^{-\tilde{\eta}}\right)
\end{equation}
with $\tilde{\eta}:=\frac{q^2}{2(q+d)}(2-(2H\vee1))$ and $T:=N\ga$.
\end{theorem}
\begin{remark}
This non-asymptotical bound  theoretically enables to calibrate the ``free parameters'' $\gamma$ and $N$ in terms of the number of observations $n$, which is fixed by the statistical setting. For instance, when $\varsigma=1$ ($i.e.$ when $q=2$), the first term is of order $n^{- (2-(2H\vee1))}$ and hence, in order to to preserve this   rate order, we have to fix $\gamma\le n^{-\frac{1-(H\vee \frac{1}{2})}{H}}$ and $N\ge n^{\frac{4+2d}{4}}\ga^{-1}$. More precisely, for these choices of parameters, the quadratic error induced by this estimator (when $\varsigma=1$) is of order $n^{-\frac{1}{2}}$ if $H<1/2$ and $n^{H-\frac{1}{2}}$ if $H>1/2$.
The constant $\varsigma$, which appears in Assumption $\HTROIS$, comes from the fact that, the bounds are first established on the distances between the invariant distributions $\nu_{\hat{\te}_{N,n,\ga}}$ and $\nu_{\te_0}$.  Nevertheless, except  some particular settings such as the Ornstein-Uhlenbeck process, this exponent $\varsigma$ is unfortunately difficult to compute in some general settings. 
Finally, let us remark that $L^p$-bounds can be easily deduced from the proof for any $p\ge 2$. However, since they do not modify significantly the results, we chose here to only state  the quadratic one.
\end{remark}

\section{Preliminary Results}\label{sec:preliminary}

In this section we label some basic results about equation \eqref{eq:sdetheta} and its invariant measure for further use. We first recall some ergodic properties of stochastic differential equations driven by a fBm, then we study the continuity of the invariant measure $\nu_{\te}$ with respect to the parameter $\te$. Under the strong coercivity assumption $\HUNS$, we quantify the distance between the empirical measures respectively related to the process $Y^{\te}$ and its Euler approximation $Z^{\te}$. Eventually we give some convergence results for the quantities involved in the right hand side of \eqref{def:hatthetanngamma}.

\subsection{Ergodic properties of the SDE and of the Euler scheme}

In this section we review several ergodic properties for equation \eqref{eq:sdetheta}. These properties are at the heart of our estimation procedure.

\subsubsection{Convergence of ${\cal L}(Y_t)$}
We start by giving the basic convergence in law towards the invariant distribution for our processes $Y^{\te}$.

\begin{proposition}\label{prop:conv-eq-TV}  Assume $\HZERO$ and $\HUNW$ and consider the family of processes $\{Y^{\te}; \, \te\in\tte\}$ defined by \eqref{eq:sdetheta}. Then the following properties hold true:

\noindent
$(i)$ Existence and uniqueness hold for the invariant distribution related to the dynamical system~\eqref{eq:sdetheta}. Furthermore, having in mind the notations introduced in \eqref{eq:defnute}, for all $\varepsilon>0$ there exists a constant $C_\varepsilon>0$ independent of $\te\in\tte$ such that
\begin{equation}\label{eq:cvgce-dist-Y-theta}
d_{\tv}\lp  \mathcal{L}\left(Y^\te_t\right), \,\nu_\te \rp
\leqslant 
C_\varepsilon t^{-(\alpha_H-\varepsilon)},
\end{equation}
with an exponent $\alpha_H$ given by
\begin{equation*}
\alpha_H=\left\{\begin{array}{ll}
\frac{1}{8}&\text{ if }\quad H\in(\frac{1}{4},1)\backslash\left\{\frac{1}{2}\right\}\\
H(1-2H)&\text{ if }\quad H\in(0,\frac{1}{4}).
\end{array}\right.
\end{equation*}

\noindent
$(ii)$ For any $p>0$ and for any distance $d\in\distdomi$, the following upper bound holds uniformly in $\te\in\tte$:
\begin{equation}\label{eq:rate-cvgce-in-Dp}
d\lp{\cal L}(Y_t^\te),\nu_\te\rp
\le
C \, t^{-\frac{\al_{H}}{4p}} ,
\end{equation}
for a strictly positive constant $C$.
In particular, for any $p>0$, we have
\begin{equation}\label{eq:controleuniverselmoment2}
\sup_{t\ge0}\ES[|Y_t|^p]<+\infty.
\end{equation}
\end{proposition}

\begin{proof} 
We prove the two statements of our proposition separately.

\noindent
\emph{Proof of item (i).}
The only difference between our claim and \cite[Theorems 1.2 and 1.3]{Ha} is the uniformity with respect to $\te\in\tte$ in the convergence in total variation result. However, following carefully the proof of \cite{Ha}, it can be shown that the constants therein do not depend on $\te$ if Hypothesis $\HUNW$ is satisfied. Therefore the constant $C_\varepsilon$ in \eqref{eq:cvgce-dist-Y-theta} is uniform in $\te$.

\noindent
\emph{Proof of item (ii).} 
Relation \eqref{eq:controleuniverselmoment2} is proved in Proposition \ref{prop:controlmoment} of the Appendix. In order to prove~\eqref{eq:rate-cvgce-in-Dp},  consider a couple $(X_1,X_2)$ of random variables 
such that $X_1\sim Y_t^{\te}$ and  $X_2$ has distribution $\nu_\te$.
By Cauchy-Schwarz inequality,
\begin{eqnarray}\label{b1}
\ES[|X_1-X_2|^p]^{\frac{1}{p}}&=&
\ES[|X_1-X_2|^p \, \1_{(X_1\neq X_2)}]^{\frac{1}{p}}   \notag\\
&\le& 
\left(\ES[|X_1|^{2p}]^{\frac{1}{2p}}+\ES[|X_2|^{2p}]^{\frac{1}{2p}}\right) \, \PE(X_1\neq X_2)^{\frac{1}{2p}}.
 \end{eqnarray}
We now bound separately the two terms on the right hand side of \eqref{b1}. If we denote by $\cac(Y^\te_t, \nu_{\te})$ the set of couplings between $\cl(Y^\te_t)$ and $\nu_{\te}$ (defined as in \eqref{eq:def-couplings}), then we have
\begin{equation*}
d_{\tv}\lp  \mathcal{L}\left(Y^\te_t\right), \,\nu_\te \rp
=
\inf\lcl  \PE(\tilde{X}_1\neq \tilde{X}_2) ;  \, (\tilde{X}_{1}, \tilde{X}_{2}) \in \cac(Y^\te_t, \nu_{\te}) \rcl .
\end{equation*}
Therefore, owing to \eqref{eq:cvgce-dist-Y-theta} we can choose a coupling  $(X_1,X_2)\in\cac(Y^\te_t, \nu_{\te})$ and a constant $C>0$ such that 
\begin{equation}\label{b2}
\PE(X_1\neq X_2) \le C t^{-(\alpha_H-\varepsilon)} .
\end{equation}
In addition, according to  \eqref{eq:controleuniverselmoment2} and a uniform integrability argument, we easily get the following inequality for the coupling $(X_1,X_2)$ we have chosen in \eqref{b2}:
\begin{equation}\label{b3}
\ES[|X_1|^{2p}]^{\frac{1}{2p}}+\ES[|X_2|^{2p}]^{\frac{1}{2p}}
\le 
2\sup_{\te\in\tte}~~\sup_{t\geqslant0}\ES [|Y^{\te}_t |^{2p}]^{\frac{1}{2p}}<+\infty.
\end{equation}
We plug \eqref{b2} and \eqref{b3}, applied to $\varepsilon=\alpha_H/2$, into \eqref{b1}. Going back to the definition \eqref{eq:defwp} of the distance ${\cal W}_p$, we obtain that there exists a strictly positive constant $C$ such that
for any $t\ge0$ and any $\te \in\tte$ we have
$$
{\cal W}_p({\cal L}(Y_t^\te),\nu_\te)\le C t^{-\frac{\alpha_H}{4p}}.
$$
The result \eqref{eq:rate-cvgce-in-Dp} follows.
\end{proof}

Next we observe that whenever $\HUNS$ is fulfilled, the polynomial convergence in \eqref{eq:rate-cvgce-in-Dp} can be replaced by an exponential rate. This is summarized in the following proposition.

\begin{proposition} Let $\{Y^\te;\te\in\tte\}$ be the family of processes defined by \eqref{eq:sdetheta}. Suppose that Hypothesis $\HZERO$ and Hypothesis $\HUNS$ are met. Let $d$ be a distance in $\distdomi$. Then, we have 
\begin{equation*}
d(\mathcal{L}(Y_t^\te),\nu_\te)\leq c_1e^{-c_2t},
\end{equation*}
with $c_2=\alpha/2$ where $\alpha$ is the constant featured in equation \eqref{eq:inward-weak}, and where $c_1=c_1(H,\alpha)$.
\end{proposition}

\begin{proof}
Let $\by^\te$ be the stationary solution of equation \eqref{eq:sdetheta}. One can easily show, by means of the same arguments as in \cite{GKN}, that 
\begin{equation}\label{eq:conv_exp}
\ES\left[|Y_t^\te-\by_t^\te|^p\right]^{1/p}\leq c_1e^{-c_2t}.
\end{equation}
The result follows trivially.
\end{proof}
\subsubsection{Ergodic Theorems for the SDE}

We now summarize the limit theorems obtained for equation \eqref{eq:sdetheta} which will be relevant for our purposes, with a special emphasis on the occupation measure $\delta_{Y_s^\te}$.

\begin{proposition}\label{prop:ergodicSDE}
 Let $Y^{\te}$ be the unique solution of \eqref{eq:sdetheta} obtained for a parameter $\te\in\tte$. Assume $\HZERO$ and $\HUNW$ hold true and let $\nu_\te$ be the measure defined by \eqref{eq:defnute}. Then for all $\te\in\tte$, for any $p>0$ and for any distance $d\in\distdomi$, we have
 
 \noindent
 $(i)$ The distance between $\nu_{\te}$ and the normalized occupation measure of $Y^{\te}$ converges to 0 as $t\to\infty$. That is
\begin{equation}\label{eq:cvgce-occp-measure}
\lim_{t\to\infty}
d\left(\frac{1}{t}\int_0^t \delta_{Y_s^\te} \, ds, \, \nu_\te\right) = 0.
\end{equation}
 In particular, we have an almost sure uniform bound for the $p$-th powers of $Y_s^\te$:
\begin{equation}\label{eq:contmomentocccont}
 \sup\limits_{\te\in\tte}\sup_{t\geq1} \frac{1}{t}\int_0^t |Y_s^\te|^p ds <\infty\quad {\rm a.s.}
 \end{equation}
 
\noindent
$(ii)$  
Some discrete versions of \eqref{eq:cvgce-occp-measure} and \eqref{eq:contmomentocccont} are also available. Specifically, let $\eta>0$ and set $t_k=k\eta$ for $k\ge0$.
Then for any $\te\in\tte$ we have
 \begin{equation*}
 \lim_{n\to\infty}
d\left(\frac{1}{n}\sum_{k=0}^{n-1} \delta_{Y_{t_k}^{\te}},\nu_\te\right)=0.
 \end{equation*}
 In particular, the following uniform bound holds true:
 \begin{equation} \label{eq:contmomentocc2}
 \sup\limits_{\te\in\tte}\sup_{n\ge1}\frac{1}{n}\sum_{k=0}^{n-1} |Y_{t_k}^{\te}|^p  <\infty\quad {\rm a.s.}
 \end{equation}
\end{proposition}
\begin{proof} 
Relations \eqref{eq:contmomentocccont} and \eqref{eq:contmomentocc2} are proved in Proposition \ref{prop:controlmoment}. As far as the identification of the limit is concerned, the proof follows the lines of \cite{CP} and is detailed in Section \ref{proof:weakconvergenceqsjkl}.
\end{proof}

\subsubsection{Ergodic Theorems for the Euler scheme}

Recall that we denote by $(Z_{k\gamma}^{\te,\ga})_{k\ge0}$ the Euler scheme with step $\gamma$ related  to $(Y_t^\te)_{t\ge0}$, as defined in \eqref{eq:Euler_scheme}.
This section focuses on the asymptotic behavior of $Z_{k\gamma}^{\te,\ga}$ as $k\to\infty$ and $\ga\to 0$.

\begin{proposition}\label{prop:ergodicEuler}
Let $Y^{\te}$ be the unique solution of \eqref{eq:sdetheta} and consider the Euler scheme $(Z_{k\gamma}^{\te,\ga})_{k\ge0}$ with step $\gamma$ related  to $(Y_t^\te)_{t\ge0}$. Assume $\HZERO$ and $\HUNW$ hold true. Then for all $\te\in\tte$, for any $p>0$ and for any distance $d\in\distdomi$, we have

\noindent
(i) There exists $\gamma_0>0$ and a unique family of measures $(\nu_{\te}^{\ga})_{\ga\le\ga_{0}}$ such that for all $\te\in\tte$ and $\gamma\in(0,\gamma_0]$, we have
 $$
\lim_{N\to\infty} 
d\left(\frac{1}{N}\sum_{k=0}^{N-1} \delta_{Z_{k\gamma}^{\te,\ga}},\nu_\te^\gamma\right)
=0.
 $$
 In particular, we get the following uniform bound for the $p$-th powers of $Z_{k\gamma}^{\te,\ga}$:
\begin{equation}\label{eq:contmomentocc}
\sup_{\te\in\tte} \sup_{N\ge1} \frac{1}{N}\sum_{k=0}^{N-1} |Z_{k\gamma}^{\te,\ga}|^p  <\infty\quad a.s.
 \end{equation}
 
\noindent
 (ii) The invariant measure $\nu_{\te}^{\ga}$ for the Euler scheme converges to the invariant measure $\nu_{\te}$ of $Y^{\te}$ as the mesh of the partition goes to 0:
 \begin{equation*}
\lim_{\ga\to 0} d\lp \nu_\te^\gamma,\nu_\te\rp = 0.
\end{equation*}
\end{proposition}
\begin{proof} The weak convergence of $\frac{1}{N}\sum_{k=0}^{N-1} \delta_{Z_{k\gamma}^{\te,\ga}}$ to $\nu_\te^\gamma$ as $n\rightarrow+\infty$, as well as the convergence of $\nu_\te^\gamma$ to $\nu_\te$ as $\gamma\rightarrow0$ are consequences of \cite[Theorem 1]{CPT}. The extension to distances $d$ dominated by ${\cal W}^p$ follows from Proposition \ref{prop:control_moment_Z}. More precisely for $(ii)$, we can deduce from Proposition~\ref{prop:control_moment_Z}  and from  Fatou's lemma that for any $M>0$, for any $\gamma\in(0,\gamma_0]$,
$$\nu_\te^\gamma(|x|^p\wedge M)\le \liminf_{N\rightarrow+\infty} \frac{1}{N}\sum_{k=0}^{N-1} \ES[|Z_{k\gamma}^{\te,\ga}|^p]\le C $$
where $C$ is a positive constant independent of $M$ and $\gamma$. Hence, taking limits as $M$ goes to $\infty$, this yields
$$
\sup_{\gamma \in(0,\gamma_0]}\nu_\te^\gamma(|x|^p)<+\infty ,
$$
for any $p>0$. Then one can conclude as in the proof of Proposition \ref{prop:conv-eq-TV} (see relation \eqref{b1}), where the distance $d_{\tv}(  \nu_{\te}^{\ga}, \,\nu_\te )$ is upper bounded thanks to \cite[Theorem 1]{CPT}.
\end{proof}

\subsection{Continuity of $\te\mapsto d(\nu_\te,\nu_{\te_0})$}\label{subsection:continuity}

The convergence of our estimator $\hat{\te}_{N,n,\ga}$ defined by \eqref{eq:firstdefesti} depends crucially on continuity properties of the family $\{\nu_{\te}; \, \te\in\tte\}$. 
To this aim, we first prove a basic result on the continuity of the map $\te\mapsto Y_{t}^{\te}$.

\begin{proposition}\label{prop:bnd-l2norm-deviation-Y-te}
Let $\te_1$ and $\te_2$ be elements of $\tte$, and consider the respective solutions $(Y^{\te_1}_t)_{t\geqslant0}$ and $(Y^{\te_2}_t)_{t\geqslant0}$ of equation \eqref{eq:sdetheta}. Assume hypothesis $\HZERO$ and $\HUNW$ are satisfied. Then for any $p>0$ and $T>0$, there exists
 $C_{T,p}>0$ independent of $\te_1$ and $\te_2$ such that
\begin{equation}\label{eq:bnd-dif-Y-theta}
\|Y^{\te_1}_t-Y^{\te_2}_t\|_{L^{p}(\oom)}
\leqslant C_{T,p}|\te_1-\te_2| .
\end{equation}
\end{proposition}

\begin{proof}
By monotonicity of the norms in $L^{p}(\oom)$, it is enough to consider the case $p\ge2$. Furthermore, it is readily seen from \eqref{eq:sdetheta} that we have 
\begin{equation}\label{c0}
Y^{\te_1}_t-Y^{\te_2}_t=\int_0^t(b_{\te_1}(Y^{\te_1}_s)-b_{\te_2}(Y^{\te_2}_s))ds.
\end{equation}
Starting from \eqref{c0}, we easily get the following identity for the square of $Y^{\te_1}_t- Y^{\te_2}_t$:
\begin{align}\label{c1}
\frac{d}{dt} \big |Y^{\te_1}_t- Y^{\te_2}_t  \big|^2  =  2
 \big \langle Y^{\te_1}_t- Y^{\te_2}_t,b_{\te_1}(Y^{\te_1}_t)-b_{\te_2}(Y^{\te_2}_t) \big  \rangle.
 \end{align}
We now invoke the fact that $b$ is Lipschitz continuous under $\HUNW$ plus inequality \eqref{eq:hyp-bnd-deriv-b} on $\partial_{\vartheta} b_{\te}$ in order to get
\begin{align*}
 \frac{d}{dt} \big |Y^{\te_1}_t- Y^{\te_2}_t  \big|^2  &=2
 \big \langle Y^{\te_1}_t- Y^{\te_2}_t,b_{\te_1}(Y^{\te_1}_t)-b_{\te_1}(Y^{\te_2}_t) \big  \rangle+2
 \big \langle Y^{\te_1}_t- Y^{\te_2}_t,b_{\te_1}(Y^{\te_2}_t)-b_{\te_2}(Y^{\te_2}_t) \big  \rangle\\
 &\leqslant c_1 \left|Y^{\te_1}_t-Y^{\te_2}_t\right|^2+2c_2|\te_1-\te_2|\left(1+\left|Y^{\te_2}_t\right|^r\right)\left|Y^{\te_1}_t-Y^{\te_2}_t\right|,
\end{align*}
where $c_1$ and $c_2$ are two strictly positive constants. Now apply the elementary inequality $|ab|\leqslant\frac{1}{2}(|a|^2+|b|^2)$ with $a=|\te_1-\te_2|(1+|Y^{\te_2}_t|^r)$
and $b=|Y^{\te_1}_t-Y^{\te_2}_t|$. We deduce the existence of a constant $c>0$ such that
\begin{align}\label{c2}
\frac{d}{dt} \big |Y^{\te_1}_t- Y^{\te_2}_t  \big|^2 
&\leqslant c \left(\left|Y^{\te_1}_t-Y^{\te_2}_t\right|^2+|\te_1-\te_2|^2\left(1+|Y^{\te_2}_t|^{2r}\right)\right).
\end{align}
With relation \eqref{c2} in hand, a standard application of Gronwall's Lemma yields
\begin{align}\label{c3}
\big |Y^{\te_1}_t- Y^{\te_2}_t  \big|^2\leqslant c |\te_1-\te_2|^2\int_0^t e^{c(t-s)}\left(1+|Y^{\te_2}_s|^{2r}\right) ds.
\end{align}

Let us now get some information about a generic $p$-th power of $Y^{\te_1}_t- Y^{\te_2}_t$ for $p\ge2$. To this aim, we resort to Jensen's inequality in relation \eqref{c3}. This gives the existence of a constant $c(T,p)$ such that for any $t\in[0,T]$ we have
\begin{align*}
\big |Y^{\te_1}_t- Y^{\te_2}_t  \big|^p\leqslant c(T,p) |\te_1-\te_2|^p\int_0^t \left(1+|Y^{\te_2}_s|^{2r}\right)^{\frac{p}{2}} ds.
\end{align*}
Taking the expectation, we finally get
\begin{align*}
\be\left[\big |Y^{\te_1}_t- Y^{\te_2}_t  \big|^p\right]^{1/p}\leqslant \tilde{c}(T,p)|\te_1-\te_2|\left(\int_0^t \be\left[1+|Y^{\te_2}_s|^{p r}\right] ds\right)^{1/p},
\end{align*}
where $\tilde{c}(T,p)$ is another finite constant. Hence our  result \eqref{eq:bnd-dif-Y-theta} follows from the bound \eqref{eq:controleuniverselmoment2}.
\end{proof}

We now state the announced continuity property for the family $\{\nu_{\te}; \, \te\in\tte\}$.

\begin{proposition}\label{prop:continuity-Q}
Let $\{Y^{\te}; \, \te\in\tte\}$ be the family of processes defined by \eqref{eq:sdetheta}.
Assume $\HZERO$ and $\HUNW$ hold true and consider the family $\{\nu_{\te}; \, \te\in\tte\}$ of invariant measures given by Proposition~\ref{prop:conv-eq-TV}. Let $p>0$ and  pick any distance $d \in \distdomi$, where we recall that $\distdomi$ is defined by \eqref{set_distances}. Then the map $\te\mapsto d(\nu_\te,\nu_{\te_0})$
is continuous on $\tte$.
\end{proposition}

\begin{proof} 
Owing to the very definition \eqref{set_distances} of $\distdomi$,
it is enough to prove the result for $d={\cal W}_p$ and for an arbitrary $p\ge1$. 
Next we apply the triangle inequality and the fact that $\cw_{p}$ is defined in~\eqref{eq:defwp} by an infimum over all couplings. This yields the following inequality, valid for any $t\ge0$:
\begin{equation}\label{c4}
{\cal W}_p(\nu_{\te_1},\nu_{\te_2})\le 2\sup_{\te\in\tte}{\cal W}_p({\cal L}(Y_t^\te),\nu_\te)
+\|Y^{\te_1}_t-Y^{\te_2}_t\|_{L^{p}(\oom)}.
\end{equation}
We now bound the two terms in the right hand side of \eqref{c4}. In order to handle the term ${\cal W}_p({\cal L}(Y_t^\te),\nu_\te)$, we consider a small parameter $\varepsilon>0$. By  Proposition  \ref{prop:conv-eq-TV}$(ii)$, there exists $t_0$ large enough such that 
\begin{equation}\label{c5}
 2\sup_{\te\in\tte}{\cal W}_p({\cal L}(Y_{t_{0}}^\te),\nu_\te)\le \frac{\varepsilon}{2}.
\end{equation}
We will fix this value of $t_{0}$ in the right hand side of \eqref{c4}. Then the difference $Y^{\te_1}_{t_{0}}-Y^{\te_2}_{t_{0}}$ is handled thanks to Proposition \ref{prop:bnd-l2norm-deviation-Y-te}. Namely consider $\delta>0$ such that (with the notations of Proposition \ref{prop:bnd-l2norm-deviation-Y-te}) we have $C_{t_0}\delta\le \frac{\varepsilon}{2}$.
We get that for all $(\te_1,\te_2)\in\tte^2$ such that $|\te_{1}-\te_{2}|\le\delta$, we have
\begin{equation}\label{c6}
\|Y^{\te_1}_t-Y^{\te_2}_t\|_{L^{p}(\oom)}
\le \frac{\varepsilon}{2}
\end{equation}
We conclude by plugging \eqref{c5} and \eqref{c6} into \eqref{c4}. This yields
\begin{equation}
{\cal W}_p(\nu_{\te_1},\nu_{\te_2})\le \varepsilon,
\end{equation}
for all $(\te_1,\te_2)\in\tte^2$ such that $|\te_{1}-\te_{2}|\le\delta$.
 The  continuity of $\te\mapsto {\cal W}_p(\nu_\te,\nu_{\te_0})$  on $\tte$ follows.
 \end{proof}

\subsection{Further controls  under $\HUNS$}
Up to now we have derived properties of the system~\eqref{eq:sdetheta} under the weak coercive assumption $\HUNW$.
In this section, we focus on possible additional bounds one can obtain under the stronger hypothesis  $\HUNS$. We will first see how $\HUNS$ guarantees a uniform control on the distance between the Euler scheme and the SDE, for a general decreasing sequence of time steps. Then we will show that $\HUNS$ ensures a some additional uniform continuity in $\te$ for the occupation measures of $Y^{\te}$.

We consider here Euler type approximations in continuous time, with time steps $\ga_{n}$ satisfying~\eqref{eq:condgamma}. In order to define this Euler approximation $(Z_t^{\te})_{t\ge0}$, we  set $s_0=0$ and $s_n=\sum_{i=1}^n\gamma_i$ for all $n\ge1$. Then for any $n\ge0$, the process $(Z^\te_t)_{t\ge0}$ is given recursively by
\begin{equation}\label{eq:sch}
Z^\te_{s_n+t}=Z^\te_{s_n}+t b_{\te}(Z^\te_{s_n})+\sigma(B_{s_n+t}-B_{s_n}),\quad t\in[0,s_{n+1}-s_n].
\end{equation}
Notice that the fractional Brownian motion $B$ in \eqref{eq:sch} is the same as the fBm driving equation~\eqref{eq:sdetheta}. The control we get on $Z^{\te}$ is summarized in the following proposition.
 
\begin{proposition}\label{prop:rapdisccont}
Let $Y^{\te}$ be the solution of equation \eqref{eq:sdetheta}, and consider the continuous-time Euler scheme $(Z^\te_t)_{t\ge0}$ with a time steps sequence $(\gamma_n)_{n\ge1}$ defined by \eqref{eq:sch}. We assume that
$\HUNS$ holds. Then the following assertions hold true.

\noindent (i) For any $p\ge 2$, there exist some positive constants $\rho$ and $C$ such that for any $n\ge1$ we have
\begin{equation}\label{eq:dif-Y-te-euler-recursive}
|Y_{s_n}^\te-Z^\te_{s_n}|^p\le e^{-\rho s_n}|Y^{\te}_{0}-Z^\te_{0}|^p+C\sum_{k=0}^{n-1}\phi_{k,p}(Z^\te_{s_k})e^{-\rho(s_n-s_{k+1})} ,
\end{equation}
where the function $\phi_{k,p}$ is defined, for any $k\ge0$, by
\begin{equation}\label{eq:def-phi-k-p}
\phi_{k,p}(z)= \gamma_{k+1}^{p+1} |b_\te(z)|^p+\int_0^{\gamma_{k+1}} |B_{s_k+t}-B_{s_k}|^pdt.
\end{equation}
\noindent (ii)  
Assume in addition that $\gamma_n\rightarrow0$ as $n\rightarrow+\infty$. Then for any $p\ge2$, there exists $n_0\in\mathbb{N}$ and some positive constants $\rho$ and $C$ such that for any $n\ge n_0$ we have
\begin{equation}\label{eq:dif-Y-te-euler-recursive-2}
|Y_{s_n}^\te-Z^\te_{s_n}|^p\le e^{-\rho(s_n-s_{n_0})}|Y^{\te}_{s_{n_0}}-Z^\te_{s_{n_0}}|^p+C\sum_{k=n_0}^{n-1}\phi_{k,p}(Y_{s_k})e^{-\rho(s_n-s_{k+1})}.
\end{equation}
\end{proposition}

\begin{proof}
Let $n\ge 0$ and consider the dynamics of $Y^{\te}-Z^{\te}$ on $[s_{n},s_{n+1})$. That is, set $\varepsilon_t:=Y^{\te}_{s_n+t}-Z^\te_{s_n+t}$ for $t\in[0,\ga_{n+1})$. Then $\ep_{t}$ verifies the relation 
\begin{equation}\label{d1}
\varepsilon_t
=
Y^{\te}_{s_n}-Z^\te_{s_n}+\int_{s_n}^{s_n+t}
\lp b_\te(Y^{\te}_{s})-b_\te(Z^\te_{s_n}) \rp ds.
\end{equation}
Starting from this equation, we divide the proof in several steps.

\noindent
\textit{Step 1: Contracting bound for $\ep_{t}$.}
Consider a parameter $\eta>0$. We wish to use the coercivity assumption $\HUNS (i)$ in order to get an upper-bound on the following derivative:
\begin{equation}\label{d2}
\left(e^{\eta t}|\varepsilon_t|^p\right)'
=e^{\eta t}\left(p|\varepsilon_t|^{p-2}\langle\varepsilon_t,\varepsilon'_t\rangle+\eta |\varepsilon_t|^p\right) .
\end{equation}
To this aim, observe that thanks to \eqref{d1} the quantity  $\langle\varepsilon_t,\varepsilon'_t\rangle$ can be expressed as:
\begin{eqnarray*}
\langle\varepsilon_t,\varepsilon'_t\rangle
&=&
\langle Y^{\te}_{s_n+t}-Z^\te_{s_n+t}, \, b_\te(Y^{\te}_{s_n+t})-b_\te(Z^\te_{s_n})\rangle \\
&=&
\langle Y^{\te}_{s_n+t}-Z^\te_{s_n+t}, \, b_\te(Y^{\te}_{s_n+t})-b_\te(Z^\te_{s_n+t})\rangle
+\langle Y^{\te}_{s_n+t}-Z^\te_{s_n+t}, \, b_\te(Z^\te_{s_n+t})-b_\te(Z^\te_{s_n})\rangle .
\end{eqnarray*}
Then we invoke  $\HUNS (i)$ and the elementary inequality $ab\le a^{2} + b^{2}$, valid for all $a,b\ge 0$. We obtain
\begin{eqnarray}\label{d21}
\langle\varepsilon_t,\varepsilon'_t\rangle
&\leqslant& 
-\alpha|Y^{\te}_{s_n+t}-Z^\te_{s_n+t}|^2+\frac{\alpha}{2}|Y^{\te}_{s_n+t}-Z^\te_{s_n+t}|^2+\frac{2}{\alpha}|b_\te(Z^\te_{s_n+t})-b_\te(Z^\te_{s_n})|^2 \notag\\
&\leqslant& 
-\frac{\alpha}{2}|\ep_{t}|^2+\frac{2L^2}{\alpha}|Z^\te_{s_n+t}-Z^\te_{s_n}|^2,
\end{eqnarray}
where the second inequality is due to relation \eqref{eq:inward-weak} and the definition \eqref{d1} of $\ep_{t}$. We now plug relation \eqref{eq:sch} into this inequality in order to get
\begin{equation*}
\langle\varepsilon_t,\varepsilon'_t\rangle
\leqslant
-\frac{\alpha}{2}\left|\ep_{t}\right|^2
+\frac{2L^2}{\alpha} \left|tb(Z^\te_{s_n})+\sigma(B_{s_n+t}-B_{s_n})\right|^2,
\end{equation*}
from which we easily end up with
\begin{equation}\label{ineq:4_proof_error_nu_gamma_nu}
p|\varepsilon_t|^{p-2}\langle\varepsilon_t,\varepsilon'_t\rangle
\leqslant 
- \frac{p\alpha}{2}|\ep_{t}|^p
+ \frac{4pL^2}{\alpha}|\varepsilon_t|^{p-2}
\lp t^2|b_\te(Z^\te_{s_n})|^2+|\sigma|^2|B_{s_n+t}-B_{s_n}|^2 \rp .
\end{equation}
Eventually we apply Young's inequality with parameters $\bar{p}=p/(p-2)$ and $\bar{q}=p/2$ and some appropriate weights to \eqref{ineq:4_proof_error_nu_gamma_nu}. This yields the existence of a constant $C=C_{p,L}$ such that
\begin{equation}\label{d3}
p|\varepsilon_t|^{p-2}\langle\varepsilon_t,\varepsilon'_t\rangle
\leqslant 
-\frac{p\alpha}{4}|\ep_{t}|^p+C\left(t^p|b_\te(Z^\te_{s_n})|^p+|B_{s_n+t}-B_{s_n}|^p\right).
\end{equation}
We are now ready to give some information about expressions of the form $e^{\eta t}|\ep_{t}|^{p}$. Namely we set $\eta=p\alpha/4$, then we apply identity \eqref{d2} and inequality  \eqref{d3}. This easily yields 
\begin{equation}\label{d4}
\left(e^{\eta t}|\varepsilon_t|^p\right)'\leqslant e^{\eta t}~C(t^p|b_\te(Z^\te_{s_n})|^p+|B_{s_n+t}-B_{s_n}|^p).
\end{equation}

\noindent
\textit{Step 2: Inductive procedure.}
Let us integrate \eqref{d4} on the interval $[0,\gamma_{n+1}]$, where we recall that $\gamma_{n+1}=s_{n+1}-s_n$. With the definition \eqref{d1} in mind, this gives
\begin{equation*}
|Y^{\te}_{s_{n+1}}-Z^\te_{s_{n+1}}|^p
\leqslant 
e^{-\eta\gamma_{n+1}}|Y^{\te}_{s_n}-Z^\te_{s_n}|^p+C\int_{0}^{\gamma_{n+1}} e^{-\eta(\gamma_{n+1}-t)}(t^p|b_\te(Z^\te_{s_n})|^p+|B_{s_n+t}-B_{s_n}|^p)dt, 
\end{equation*}
which yields
\begin{equation}\label{kljlkdjdq}
|Y^{\te}_{s_{n+1}}-Z^\te_{s_{n+1}}|^p\leqslant e^{-\eta\gamma_{n+1}}|Y^{\te}_{s_n}-Z^\te_{s_n}|^p+C\left(\gamma_{n+1}^{p+1} |b_\te(Z^\te_{s_n})|^p+\int_0^{\gamma_{n+1}} |B_{s_n+t}-B_{s_n}|^p)dt\right).
\end{equation}
Thus, setting $\phi_{n,p}(x)= \left(\gamma_{n+1}^{p+1} |b_\te(x)|^{p}+\int_0^{\gamma_{n+1}} |B_{s_n+t}-B_{s_n}|^p)dt\right)$, an elementary induction procedure yields the following relation for every $n\ge1$:
$$
|Y^{\te}_{s_n}-Z^\te_{s_n}|^p\le |Y^{\te}_{0}-Z^\te_{0}|^p e^{-\eta s_n}+\sum_{k=0}^{n-1}\phi_{k,p}(Z^\te_{s_k})e^{-\eta(s_n-s_{k+1})}.
$$
This proves our claim \eqref{eq:dif-Y-te-euler-recursive}.

\noindent
\textit{Step 3: Proof of \eqref{eq:dif-Y-te-euler-recursive-2}.}
In order to obtain our second statement \eqref{eq:dif-Y-te-euler-recursive-2}, one needs to go back to inequality \eqref{kljlkdjdq}. Then observe that hypothesis \eqref{eq:inward-weak} yields
$$ 
|b(Z^\te_{s_n})|^p\le  2^{p-1}\left(|b(Y_{s_n}^{\te})|^p+ L^p |Y_{s_n}^\te-Z^\te_{s_n}|^p\right),
$$
which leads for any $p$ to the existence of a constant $C$ such that
\begin{equation}\label{d5}
|Y^{\te}_{s_{n+1}}-Z^\te_{s_{n+1}}|^p\leqslant \left(e^{-\eta\gamma_{n+1}}+C\ga_{n+1}^{p+1}\right)|Y^{\te}_{s_n}-Z^\te_{s_n}|^p+C\phi_{n,p}(Y_{s_n}).
\end{equation}
Since $\lim_{n\to\infty}\ga_n =0$, one checks easily that there exists a $n_0$ such that for any $n\ge n_0$ we have the following inequality: 
$$
e^{-\eta\gamma_{n+1}}+C\ga_{n+1}^{p+1}\le e^{-\frac{\eta}{2}\gamma_{n+1}}.
$$
Plugging this information into \eqref{d5}, we end up with
\begin{equation*}
|Y^{\te}_{s_{n+1}}-Z^\te_{s_{n+1}}|^p
\leqslant 
e^{-\frac{\eta}{2}\gamma_{n+1}} |Y^{\te}_{s_n}-Z^\te_{s_n}|^p+C\phi_{k,p}(Y_{s_n}).
\end{equation*}
Our assertion \eqref{eq:dif-Y-te-euler-recursive-2} then follows by an induction procedure exactly as for Step 2.
\end{proof}

We now give a continuity results (with respect to the parameter $\te$) for some occupation measures related to our processes of interest. The proofs are postponed to Appendix \ref{append:B}.

\begin{proposition}\label{prop:contracttt} 
As in Proposition \ref{prop:continuity-Q}, let $Y^{\te}$ be the solution of equation \eqref{eq:sdetheta} and consider the Euler scheme $Z^{\te,\ga}$ defined by \eqref{eq:Euler_scheme}. Also consider $p>0$ and $d\in\distdomi$. We assume that $\HUNS$ holds true. Then,

\noindent
(i) The occupation measures of the process $Y^{\te}$ are Lipschitz with respect to $\te$, that is there exists a positive random variable $C_p$ such that for all $t\geq1$:
\begin{equation}\label{eq:Lip_occ_meas_continuous}
d\left(\frac{1}{t}\int_0^t \delta_{Y_s^{\te_1}} ds,\frac{1}{t}\int_0^t \delta_{Y_s^{\te_2} }ds\right)\leq C_p|\te_1-\te_2|.
\end{equation}

\noindent
(ii) The occupation measures of the Euler approximation $Z^{\te,\ga}$ are also Lipschitz with respect to $\te$. Namely there exists $\gamma_0>0$ such that: for any $\gamma\in(0,\gamma_0]$, there exists a positive random variable $C_p(\ga)$ such that for all $N\geq1$
\begin{equation}\label{eq:Lip_occ_meas_discrete}
d\left(\frac{1}{N}\sum_{k=0}^{N-1} \delta_{Z_{k\gamma}^{\te_1,\ga}} ,\frac{1}{N}\sum_{k=0}^{N-1} \delta_{Z_{k\gamma}^{\te_2,\ga}} \right)\leq C_p(\ga)|\te_1-\te_2|.
\end{equation}
\end{proposition}

\begin{remark}\label{rem:strategy_using_Dp} In the sequel we will analyze several quantities like \eqref{eq:Lip_occ_meas_discrete}, where we compare two discrete random measures on $\R^d$ $\nu_1=\frac{1}{N}\sum_{i=1}^N\delta_{X_i}$ and $\nu_2=\frac{1}{N}\sum_{i=1}^N\delta_{Y_i}$ with $X_i=X_i(\omega)$ and $Y_i=Y_i(\omega)$. In this context we will always upper bound quantities of the form $d(\nu_1,\nu_2)$ for a distance $d$ in $\distdomi$. To this aim, resorting a trivial coupling between $\nu_1$ and $\nu_2$, it is enough to prove an almost sure upper bound on 
$$\frac{1}{N}\sum_{i=1}^N|X_i-Y_i|^p.$$
We will adopt this strategy throughout the paper, the typical outcome being an a.s. bound on $d(\nu_1,\nu_2)$. Straightforward extensions to a continuous time setting allow to handle quantities of the form \eqref{eq:Lip_occ_meas_continuous}.
\end{remark}

\section{Proof of the consistency theorems}\label{sec:consistencyproof}
The aim of this section is to achieve the proof of  Theorems \ref{thm:cvgce-estimator-dicrete-theta}, \ref{thm:cvgce-esti1} and \ref{thm:cvgce-esti2}. We first establish a general asymptotic result for a family of contrasts in Sections~\ref{sec:unif-cvgce-contrast} and~\ref{sec:general-convergence}. Then we will combine this general proposition with our preliminary results of Section~\ref{sec:preliminary} in order to prove our main claims.

\subsection{Uniform convergence of the contrast}\label{sec:unif-cvgce-contrast}
In this section we state some uniform convergence results for the \textit{contrast}, i.e for the function involved in the definition of estimators such as~\eqref{eq:firstdefesti}. We should notice at this point that our uniform convergence results hold only under the assumption $\HUNS$. 
In case of a constant time step Euler scheme, we get the following result.

\begin{proposition} \label{prop:unifconst}
We consider the same setting as in Proposition \ref{prop:contracttt}. In particular, we assume that $\HUNS$ holds true for the coefficients of equation \eqref{eq:sdetheta}. Then the following assertions hold true.

\noindent
$(i)$ 
The invariant measure $\nu^{\ga}$ of the Euler scheme converges uniformly to the invariant measure of $Y^{\te}$. Namely
\begin{equation*}
\limsup_{\ga\rightarrow0}\ga^{-H} \sup_{\te\in\tte} d(\nu_\te,\nu_\te^\gamma)<+\infty.
\end{equation*}
In particular, we have
\begin{equation*}
\lim_{\ga\rightarrow0} \, \sup_{\te\in\tte} d(\nu_\te,\nu_\te^\gamma) = 0.
\end{equation*}

\noindent
$(ii)$ 
The occupation measure of the Euler scheme converges to the invariant measure $\nu_\te^\gamma$ as the number of steps goes to $\infty$, that is:
\begin{equation*}
\lim_{N\rightarrow+\infty} \, \sup_{\te\in\tte} 
d\left(\frac{1}{N}\sum_{k=0}^{N-1}\delta_{Z_{s_{k}}^{\te,\ga}}, \nu_\te^\gamma\right)
=0.
\end{equation*}

\noindent
$(iii)$ We have
$$\lim_{\gamma\rightarrow 0}\lim_{N,n\rightarrow+\infty}
\sup\limits_{\te\in\tte}\left|d\left(\frac{1}{n}\sum_{k=0}^{n-1} \delta_{Y_{t_k}},\frac{1}{N}\sum_{k=0}^{N-1} \delta_{Z_{s_k}^{\te,\gamma}} \right)
- 
d(\nu_{\te_0},\nu_\te)\right|=0.
$$
\end{proposition}

\begin{proof} 
We prove the three items separately.

\noindent
\textit{Proof of (i).} For sake of simplicity, we only detail the proof for $p=2$. The extension to a general $p$ does not generate particular difficulties and can be done as in Proposition~\ref{prop:rapdisccont}.
We start from the following inequality:
\begin{equation}\label{ineq:1_proof_error_nu_gamma_nu}
d\left(\nu^\ga_{\te},\nu_\te\right)\leqslant 
d\left(\mathcal{L}\left(Z^{\te,\ga}_{N\gamma}\right),\mathcal{L}\left(Y^{\te}_{N\gamma}\right)\right)+d\left(\nu^\ga_{\te},\mathcal{L}\left(Z^{\te,\ga}_{N\gamma}\right)\right)+d\left(\mathcal{L}\left(Y^{\te}_{N\gamma}\right),\nu_\te\right).
\end{equation}
Let us consider the three terms of the right-hand side of \eqref{ineq:1_proof_error_nu_gamma_nu} successively. First, without loss of generality, we can assume that 
$Y_0^\te=Z_0^{\te,\ga}$. Furthermore, we have that $Z_{n\gamma}^{\te,\ga}=Z^\te_{n\gamma}$ for all $n\ge 1$, where $Z^\te$ is defined by \eqref{eq:sch}. Then, by Proposition \ref{prop:rapdisccont} $(i)$ applied with $\gamma_n=\gamma$,
we have
$$\ES[|Y_{n\gamma}^\te-Z_{n\gamma}^{\te,\ga}|^2]\le C\sum_{k=0}^{n-1}\ES[\phi_{k,2}(Z_{k\gamma}^{\te,\ga})]e^{-\rho\gamma(n-1-k)},
$$
where we recall that $\phi_{k,p}$ is defined by \eqref{eq:def-phi-k-p}.
Using that $b_\te$ is sublinear (uniformly in $\te$) and the fact that the increments of $B$ satisfy relation \eqref{eq:var-increments-fbm}, one obtains
$$\ES[|Y_{n\gamma}^\te-Z_{n\gamma}^{\te,\ga}|^2]
\le 
C\sum_{k=0}^{n-1} 
\left(\ga^3\left(1+\ES[|Z_{k\gamma}^{\te,\ga}|^2]\right)
+\ga^{2H+1}\right)e^{-\rho\gamma(n-1-k)}.
$$
It follows easily that
\begin{equation*}
\sup_{\te\in\tte,\ga\in(0,\ga_0]}\limsup_{n\rightarrow+\infty}
\ga^{-2H}\ES[|Y_{n\gamma}^\te-Z_{n\gamma}^{\te,\ga}|^2]
\le 
1+
\sup_{\te\in\tte,\ga\in(0,\ga_0]}\limsup_{n\rightarrow+\infty}\ES[|Z_{k\gamma}^{\te,\ga}|^2].
\end{equation*}
Moreover, owing to Proposition \ref{prop:control_moment_Z}, there exists $\ga_0$ such that 
\begin{equation*}
\limsup_{n\rightarrow+\infty}\ES[|Z_{k\gamma}^{\te,\ga}|^2]
<\infty.
\end{equation*}
Summarizing, we have obtained
\begin{equation*}
\sup_{\te\in\tte,\ga\in(0,\ga_0]}\limsup_{n\rightarrow+\infty}
\ga^{-2H}\ES[|Y_{n\gamma}^\te-Z_{n\gamma}^{\te,\ga}|^2]
<\infty.
\end{equation*}

We now consider the term $d(\nu^\ga_{\te},\mathcal{L}(Z^{\te,\ga}_{N\gamma}))$ in the right hand side of \eqref{ineq:1_proof_error_nu_gamma_nu}.  Using that $\nu^\ga_\te$ is invariant, we remark that
$$d\left(\nu^\ga_{\te},\mathcal{L}\left(Z^{\te,\ga}_{N\gamma}\right)\right)\le \ES[|Z^{\te,\ga}_{N\gamma}-\bar{Z}^{\te,\ga}_{N\gamma}|^2]^{\frac{1}{2}}$$
where $\bar{Z}^{\te,\ga}$ denotes a stationary Euler scheme built with the same noise process as for for $Z^{\te,\ga}$. Thus, thanks to the fact that 
\begin{equation*}
Z^{\te,\ga}_{k\gamma}-\bar{Z}^{\te,\ga}_{k\gamma}=Z^{\te,\ga}_{(k-1)\gamma}-\bar{Z}^{\te,\ga}_{(k-1)\gamma}+\ga\lp b(Z^{\te,\ga}_{(k-1)\gamma})-b(\bar{Z}^{\te,\ga}_{(k-1)\gamma})\rp
\end{equation*}
a straightforward induction under assumption $\HUNS$, similar to \eqref{d21}, leads to
\begin{equation}
\left|Z^{\te,\ga}_{k\gamma}-\bar{Z}^{\te,\ga}_{k\gamma}\right|^2\leqslant (1-2\gamma \alpha+\gamma^2 L^2)^{k} \left|Z^{\te,\ga}_{0}-\bar{Z}^{\te,\ga}_{0}\right|^2.
\end{equation}
We choose $\ga_0=\alpha/L^2$ in such a way that $2 \alpha-\gamma L^2\ge\alpha$ for any $\ga\in(0,\ga_0]$. In addition, recall that $Z^{\te,\ga}_{0}=z_0$. Then, by possibly picking a smaller value of $\ga_0$, we deduce from Proposition~\ref{prop:control_moment_Z} that
$$
\limsup_{N\rightarrow+\infty}\sup_{\te\in\tte}d\left(\nu^\ga_{\te},\mathcal{L}\left(Z^{\te,\ga}_{N\gamma}\right)\right)=0.
$$

Eventually, the last term in the right hand side of \eqref{ineq:1_proof_error_nu_gamma_nu} tends to $0$ uniformly in $\te$ as $N\rightarrow+\infty$ by Proposition \ref{prop:conv-eq-TV}. This concludes the proof of $(i)$.

\noindent
\textit{Proof of (ii).}
By Proposition \ref{prop:ergodicEuler}, the convergence of $d(\frac{1}{N}\sum_{k=0}^{N-1}\delta_{Z_{k\gamma}^{\te,\ga}}, \nu_\te^\gamma)$ to 0 is true for the simple convergence. In order to extend this result to a uniform convergence in $\te$, we use
 Proposition~\ref{prop:contracttt}$(ii)$ to obtain that the family $\{\te\mapsto d(\frac{1}{N}\sum_{k=0}^{N-1} \delta_{Z_{k\ga}^{\te,\ga}},\nu_\te^\ga); \,
 N\ge1,\te\in\tte\}$ is equicontinuous for a fixed $\ga\in(0,\ga_0]$. Actually, for some given $\te_1$ and $\te_2$,
 $$\left|d\left(\frac{1}{N}\sum_{k=0}^{N-1} \delta_{Z_{k\ga}^{\te_1,\ga}},\nu_{\te_1}^\ga\right)-d\left(\frac{1}{N}\sum_{k=0}^{N-1} \delta_{Z_{k\ga}^{\te_2,\ga}},\nu_{\te_2}^\ga\right)\right|\le d(\nu_{\te_1}^\ga,\nu_{\te_2}^\ga)+ d\left(\frac{1}{N}\sum_{k=0}^{N-1} \delta_{Z_{k\ga}^{\te_1,\ga}},\frac{1}{N}\sum_{k=0}^{N-1} \delta_{Z_{k\ga}^{\te_2,\ga}}\right).$$
 The second term goes to $0$ as $\te_1-\te_2\rightarrow0$ by Proposition \ref{prop:contracttt}$(ii)$. This is also the case for the first one by letting
 $N$ go to $\infty$ in Proposition \ref{prop:contracttt}$(ii)$ (for instance).\\
 
 \noindent
\textit{Proof of (iii).} 
This is a simple consequence of the two previous statements and of Proposition \ref{prop:ergodicSDE}$(ii)$.
 \end{proof}

 We now generalize Proposition \ref{prop:unifconst} to the case of a decreasing time step for the Euler scheme~\eqref{eq:Euler_scheme}.
 
 \begin{proposition} \label{prop:unifdecreas}
 We consider the same setting as in Proposition \ref{prop:contracttt}. In particular, we assume that $\HUNS$ holds true for the coefficients of equation \eqref{eq:sdetheta}.
 Let $p\ge 2$ and consider $d\in\distdomi$. Let $\{s_{k}; \, k\ge 0\}$ be the sequence of time steps defined by \eqref{eq:def-steps-sk}, which is assumed to verify
 \begin{equation}\label{eq:hyp_gamma}
\exists \,p'\ge p\;\textnormal{such that}\;\sum_{k=1}^{+\infty} \frac{\gamma_{k+1}^{{p'H}+1}}{s_k}<+\infty .
\end{equation}
Then we have
\begin{equation}\label{eq:conv_unif_decrease_step}
\lim_{N,n\rightarrow+\infty}\sup_{\te\in\tte}\left|d\left(\frac{1}{n}\sum_{k=0}^{n-1} \delta_{Y_{t_k}},\frac{1}{s_N}\sum_{k=0}^{N-1} \gamma_{k+1}\delta_{Z_{s_k}^{\te}} \right)- d(\nu_{\te_0},\nu_\te)\right|=0.
\end{equation}
\end{proposition}

\begin{proof} For notational convenience, the proof will be detailed for the continuous-time Euler approximation $(Z^\te_t)_{t\ge0}$ defined by \eqref{eq:sch}, with step sequence $(\gamma_n)_{n\ge1}$. An application of the triangular inequality allows us to bound the left hand side of \eqref{eq:conv_unif_decrease_step} as follows:
\begin{equation}\label{eq:split_d}
\left| d\left(\frac{1}{n}\sum_{k=0}^{n-1} \delta_{Y_{t_k}},\frac{1}{s_N}\sum_{k=0}^{N-1} \gamma_{k+1}\delta_{Z_{s_k}^{\te}} \right)-d(\nu_{\te_0},\nu_\te)\right|
\leqslant A_{1,n}+A_{2,N}(\te).
\end{equation}
with
\begin{equation*}
A_{1,n}:=d\left(\frac{1}{n}\sum_{k=0}^{n-1} \delta_{Y_{t_k}},\nu_{\te_0}\right),\qquad
A_{2,N}(\te):=d\left(\frac{1}{s_N}\sum_{k=0}^{N-1} \gamma_{k+1}\delta_{Z_{s_k}^{\te}},\nu_\te \right).
\end{equation*}
Our claim can thus be reduced to prove that
$$\lim\limits_{n\to+\infty}A_{1,n}=0,
\quad\text{ and }\quad
\lim\limits_{N\to+\infty}\sup\limits_{\te\in\tte}A_{2,N}(\te)=0.
$$
Furthermore, the fact that $\lim_{n\to+\infty}A_{1,n}=0$ is a direct consequence of Proposition \ref{prop:ergodicSDE}-(ii). We thus focus on the asymptotic behavior of $A_{2,N}$ in the remainder of the proof.

In order to bound $A_{2,N}$, let us set $\underline{s}=\max\{s_k, s_k\le s\}$. Then we observe that 
$$
\frac{1}{s_N}\sum_{k=0}^{N-1} \gamma_{k+1}\delta_{Z_{s_k}^{\te}}=\frac{1}{s_N}\int_0^{s_N} \delta_{Z_{\underline{s}}^\te}ds.
$$
Therefore we can split $A_{2,N}$ into 
\begin{equation}\label{eq:split_A2N}
A_{2,N}(\te)\leqslant A_{21,N}(\te)+A_{22,N}(\te)+A_{23,N}(\te)
\end{equation}
where the quantities $A_{2j,N}(\te)$ for $j=1,2,3$ are defined as follows:
\begin{align}
&A_{21,N}(\te):=d\left(\frac{1}{s_N}\int_0^{s_N} \delta_{Z^\te_{\underline{s}}}ds,\frac{1}{s_N}\int_0^{s_N} \delta_{Y^\te_{\underline{s}}} ds\right),\nonumber\\
&A_{22,N}(\te):=d\left(\frac{1}{s_N}\int_0^{s_N} \delta_{Y^\te_{\underline{s}}}ds,\frac{1}{s_N}\int_0^{s_N} \delta_{Y^\te_{s}} ds\right),\label{eq:A22N}\\
&A_{23,N}(\te):=d\left(\frac{1}{s_N}\int_0^{s_N} \delta_{Y^\te_{s}}ds,\nu_\te\right).\label{eq:A23N}
\end{align}
We will now treat those three terms separately.

The term $A_{23,N}$ is easily handled by applying Proposition \ref{prop:ergodicSDE} $(i)$ (simple convergence) and Proposition \ref{prop:unifconst} $(i)$ (equicontinuity). We thus get 
$$
\lim\limits_{N\to+\infty}\sup\limits_{\te\in\tte}A_{23,N}(\te)=0.
$$
Now, let us  prove that $\lim\limits_{N\to+\infty}\sup\limits_{\te\in\tte}A_{22,N}(\te)=0.$ Since $d\in {\cal D}_p\subset{\cal D}_{p'}$, we can assume without loss of generality that
$d={\cal W}_{p'}$. We invoke the strategy outlined in Remark \ref{rem:strategy_using_Dp}. This means that we are reduced to prove the following limit :  
\begin{equation}\label{eq:new_convergence_to_prove}
\sup\limits_{\te\in\tte}\frac{1}{s_N}\int_{0}^{s_N}|Y_t^\te-Y_{\underline{t}}^\te|^{p'}dt \xrn{N\rightarrow+\infty}0\quad a.s.
\end{equation}
To this aim, we first note that 
\begin{equation}\label{eq:c1}
|Y_t^\te-Y_{\underline{t}}^\te|\le \int_{\underline{t}}^t |b_\te(Y_s^\te)|ds+\|\sigma\||B_{t}-B_{\underline{t}}| \, ,
\end{equation}
and thus there exists a constant $c_{p',\sigma}$ such that
\begin{equation}\label{eq:c}
\frac{1}{s_N}\int_{0}^{s_N}|Y_t^\te-Y_{\underline{t}}^\te|^{p'}dt\le c_{p,\sigma}\left[\frac{1}{s_N}\int_{0}^{s_N}\left|\int_{\underline{t}}^t |b_\te(Y_s^\te)|ds\right|^{p'} dt+\frac{1}{s_N}\int_0^{s_N}|B_{t}-B_{\underline{t}}|^{p'}dt\right].
\end{equation}
Next we upper bound the first term in the right hand side of \eqref{eq:c} invoking successively Jensen's inequality and Fubini's theorem. We obtain that 
\begin{eqnarray*}
\int_{0}^{s_N}\left|\int_{\underline{t}}^t |b_\te(Y_s^\te)|ds\right|^{p'} dt
&\leq& \int_{0}^{s_N}|t-\underline{t}|^{p'-1}\left(\int_{\underline{t}}^t |b_\te(Y_s^\te)|^{p'}ds\right) dt\\
&=&\int_{0}^{s_N}|b_\te(Y_s^\te)|^{p'}\left(\int_{s}^{\bar{s}} |t-\underline{s}|^{p'-1}dt\right) ds \, ,
\end{eqnarray*}
where we have introduced the additional notation $\bar{s}=\min\{s_k,s_k>s\}.$ Taking into account the fact that $|t-\underline{s}|\le \gamma_{k+1}$ for any $s,t\in[s_k,s_{k+1})$, we end up with
\begin{equation}\label{eq:term_b_theta}
\int_{0}^{s_N}\left|\int_{\underline{t}}^t |b_\te(Y_s^\te)|ds\right|^{p'} dt\leq \sum_{k=0}^{N-1}\ga_{k+1}^{p'}\int_{s_k}^{s_{k+1}}|b_\te(Y_s^\te)|^{p'}ds.
\end{equation}
Furthermore, since $b_\te$ is uniformly sublinear in $\te$ and owing to \eqref{eq:contmomentocccont} we know that $$C_1:=\sup\limits_{\te\in\tte}\sup\limits_{N\geq 1}\frac{1}{s_N}\int_{0}^{s_{N}}|b_\te(Y_s^\te)|^{p'}ds=\sup\limits_{\te\in\tte}\sup\limits_{N\geq 1}\frac{1}{s_N}\sum_{k=0}^{N-1}\int_{s_k}^{s_{k+1}}|b_\te(Y_s^\te)|^{p'}ds<+\infty\quad a.s.$$
Moreover, $\lim_{k\to+\infty}\ga_{k+1}^{p'}=0$. Hence, for all $\varepsilon>0$ there exists $k_0\geq0$ such that for all $k\geq k_0$, we have $\ga^{p'}_{k+1}<\varepsilon$. One can thus deduce that
\begin{align}\label{eq:f}
\sup\limits_{\te\in\tte}\frac{1}{s_N}\sum_{k=0}^{N-1}\ga_{k+1}^{p'}&\int_{s_k}^{s_{k+1}}|b_\te(Y_s^\te)|^{p'}ds\nonumber\\
&\leqslant\sup\limits_{\te\in\tte}\frac{1}{s_N}\sum_{k=0}^{k_0-1}\ga_{k+1}^{p'}\int_{s_k}^{s_{k+1}}|b_\te(Y_s^\te)|^{p'}ds+\sup\limits_{\te\in\tte}\frac{1}{s_N}\sum_{k=k_0}^{N-1}\ga_{k+1}^{p'}\int_{s_k}^{s_{k+1}}|b_\te(Y_s^\te)|^{p'}ds\nonumber\\
&\leqslant C_1\left(\frac{\ga_1^{p'}s_{k_0}}{s_N}+\varepsilon\right),
\end{align}
from which it is easily seen that
\begin{equation}\label{eq:d}
\sup\limits_{\te\in\tte}\frac{1}{s_N}\sum_{k=0}^{N-1}\ga_{k+1}^{p'}\int_{s_k}^{s_{k+1}}|b_\te(Y_s^\te)|^{p'}ds\xrn{N\rightarrow+\infty}0\quad a.s.
\end{equation}
We can conclude that the first term in the right hand side of \eqref{eq:c} vanishes as $N\to+\infty$ due to our identity \eqref{eq:term_b_theta}.

In order to prove that the second term in the right hand side of \eqref{eq:c} converges to $0$, we invoke Kronecker's lemma (in its continuous version, see \cite[Theorem 2.1]{E}). We get that it is sufficient to prove that:
$$\int_0^{+\infty} \frac{|B_{t}-B_{\underline{t}}|^{p'}}{1+t} dt<+\infty\quad a.s.$$
However, due to the fact that $(t-\underline{t})\leqslant\ga_{k+1}$ if $t\in[s_k,s_{k+1})$, we have,
$$\ES\left[\int_0^{+\infty} \frac{|B_{t}-B_{\underline{t}}|^{p'}}{1+t} dt\right]\le \int_0^{+\infty}\frac{1}{1+t} (t-\underline{t})^{{p'H}} dt\le c\sum_{k=1}^{+\infty}\frac{\gamma_{k+1}^{{p'H}+1}}{s_k}<+\infty$$
where $c$ is a positive constant and where the last inequality stems from hypothesis \eqref{eq:hyp_gamma}. Hence, Kronecker's lemma yields 
\begin{equation}\label{eq:e}
\frac{1}{s_N}\int_0^{s_N}|B_t-B_{\underline{t}}|^{p'}ds\xrn{N\rightarrow+\infty}0\quad a.s.
\end{equation}
Now inequality \eqref{eq:c} combined with \eqref{eq:d} and \eqref{eq:e} easily yields \eqref{eq:new_convergence_to_prove}.
We conclude that $~\lim\limits_{N\to+\infty}\sup\limits_{\te\in\tte}A_{22,N}(\te)=0~$.\\

Going back to our decomposition \eqref{eq:split_A2N}, we still have to prove that $~\lim_{N\to+\infty}\sup_{\te\in\tte}A_{21,N}(\te)=0.$
To this end, let us write $A_{21,N}(\te)$ in its discrete form:
\begin{equation}\label{eq:discrete_A21N}
A_{21,N}(\te)=d\left(\frac{1}{s_N}\sum_{k=0}^{N-1}\ga_{k+1}\delta_{Z^\te_{s_k}}, \frac{1}{s_N}\sum_{k=0}^{N-1}\ga_{k+1}\delta_{Y^\te_{s_k}}\right).
\end{equation} 
Then invoking Remark \ref{rem:strategy_using_Dp} once more, we are reduced to prove 
\begin{equation}\label{eq:new_convergence_to_prove2}
\frac{1}{s_N}\sum_{k=0}^{N-1}\gamma_{k+1}|Y_{s_{k}}^\te-Z_{s_{k}}^\te|^{p'}\xrn{N\rightarrow+\infty}0\quad a.s.
\end{equation}
In order to achieve \eqref{eq:new_convergence_to_prove2} consider the integer $n_0$ given by Proposition \ref{prop:rapdisccont} $(ii)$. For $k\leqslant n_0$ we trivially bound $|Y^\te_{s_k}-Z^\te_{s_k}|$ using \eqref{eq:dif-Y-te-euler-recursive}. Since \eqref{eq:sch} asserts that $\sup_{\te\in\tte,k\in\{1,\ldots,n_0\}}|Z_{s_k}^\te|<+\infty$, we get that
\begin{equation}\label{eq:bound_n_0}
\sup_{\te\in\tte,k\in\{1,\ldots,n_0\}}|Y_{s_k}^\te-Z_{s_k}^\te|=:C(\omega)<+\infty\quad a.s.
\end{equation}
We now bound the right hand side of \eqref{eq:new_convergence_to_prove2} by means of \eqref{eq:bound_n_0} whenever $k\leqslant n_0$ and invoking~\eqref{eq:dif-Y-te-euler-recursive-2} when $k\geqslant n_0+1$.
This gives
\begin{align}\label{eq:d1}
\frac{1}{s_N}\sum_{k=0}^{N-1}\gamma_{k+1}|Y_{s_{k}}^\te-Z_{s_{k}}^\te|^{p'}\le ~& C(\omega)^{p'}\frac{1}{s_N}\left(\sum_{k=0}^{n_0} \gamma_{k+1}+\sum_{k=n_0+1}^{N-1} \gamma_{k+1} e^{-\rho(s_k-s_{n_0})}\right)\nonumber\\
&+\frac{1}{s_N}\sum_{k=n_0+1}^{N-1} \gamma_{k+1}\sum_{\ell=n_0}^{k-1} \phi_{\ell,p'}(Y_{s_\ell}^\te) e^{-\rho(s_k-s_{\ell+1})}.
\end{align}
The first term in the right-hand side of \eqref{eq:d1} is clearly evanescent as $N\rightarrow+\infty$. Let us focus on the second one. By a Fubini type transformation, we get
\begin{align*}
\sum_{k=n_0+1}^{N-1} \gamma_{k+1}\sum_{\ell=n_0}^{k-1} \phi_{\ell,p'}(Y_{s_\ell}^\te) e^{-\rho(s_k-s_{\ell+1})}
&=\sum_{\ell=n_0}^{N-2} \phi_{\ell,p'}(Y_{s_\ell}^\te)\sum_{k=\ell+1}^{N-1}\gamma_{k+1}  e^{-\rho(s_k-s_{\ell+1})}\\
&\le
\sum_{\ell=0}^{N-2} \phi_{\ell,p'}(Y_{s_\ell}^\te) \sum_{k=\ell+1}^{N-1} \gamma_k e^{-\rho(s_k-s_{\ell+1})}
\end{align*}
where the last inequality is due to the fact that $(\ga_k)$ is non increasing.
Since $x\mapsto e^{-\rho x}$ is a non-increasing function, we remark that for $\ell\in\{0,\dots,N-2\}$,
$$\sum_{k=\ell+1}^{N-1} \gamma_k e^{-\rho(s_k-s_{\ell+1})}\le e^{\rho s_{\ell+1}}\int_{s_{\ell}}^{s_{N-1}} e^{-\rho x} dx=\frac{1}{\rho}\left(e^{\rho\ga_{\ell+1}}-e^{-\rho(s_{N-1}-s_{\ell+1})}\right)\le C_\rho.$$
Thus, in order to see that the right hand side of \eqref{eq:d1} vanishes as $N\to+\infty$, it remains to show that 
\begin{equation}\label{eq:laststep34}
\lim_{N\to+\infty}\sup_{\te\in\tte}\frac{1}{s_N}\sum_{\ell=0}^{N-1} \phi_{\ell,p'}(Y_{s_\ell}^\te)=0
\end{equation}
where we recall that $\phi_{\ell,p'}$ is defined by \eqref{eq:def-phi-k-p} and thus
\begin{equation}\label{eq:dec_phi}
\frac{1}{s_N}\sum_{\ell=0}^{N-1} \phi_{\ell,p'}(Y_{s_\ell}^\te)=\frac{1}{s_N}\sum_{\ell=0}^{N-1}\ga_{\ell+1}^{p'+1}|b_\te(Y_{s_\ell}^\te)|^{p'}+\frac{1}{s_N}\sum_{\ell=0}^{N-1}\int_0^{\ga_{\ell+1}}|B_{s_\ell+t}-B_{s_\ell}|^{p'}dt.
\end{equation}

Let us begin by the term involving $b_\te$ in \eqref{eq:dec_phi}. First notice that $$\mathcal{W}_{p'}\left(\frac{1}{s_N}\int_0^{s_N}\delta_{Y^\te_{\underline{s}}}ds, \nu_\te\right)\leqslant A_{22,N}(\te)+A_{23,N}(\te),$$
where $A_{22,N}$ and $A_{23,N}$ are respectively defined by \eqref{eq:A22N} and \eqref{eq:A23N} with $d=\mathcal{W}_{p'}$.
Furthermore, we have seen that  $~\lim_{N\to+\infty}\sup_{\te\in\tte}A_{22,N}(\te)=0~$ and $~\lim_{N\to+\infty}\sup_{\te\in\tte}A_{23,N}(\te)=0~$. We immediately deduce that 
$$\lim_{N\to+\infty}\sup_{\te\in\tte}\mathcal{W}_{p'}\left(\frac{1}{s_N}\int_0^{s_N} \delta_{Y^\te_{\underline{s}}}ds,\nu_\te\right)=0.$$ 
Moreover, by Proposition \ref{prop:controlmoment} for instance,  we have $~\sup_{\te\in\tte}\nu_\te(|\cdot|^{p'})<\infty$. We thus deduce that
$$\sup_{\te\in\tte}\sup_{N\geqslant1}\frac{1}{s_N}\sum_{\ell=0}^{N-1}\gamma_{\ell+1} |Y_{s_{\ell}}^\te|^{p'}<+\infty\quad a.s.$$
Furthermore, since $b_\te$ is uniformly sublinear in $\te$ and $\lim_{\ell\to+\infty}\ga^{p'}_{\ell+1}=0$, it easily follows along the same lines as for \eqref{eq:d} that:
\begin{equation}\label{eq:g}
\sup_{\te\in\tte}\frac{1}{s_N}\sum_{\ell=0}^{N-1}\gamma_{\ell+1}^{p'+1} |b_\te(Y_{s_{\ell}})|^{p'}\xrn{N\rightarrow+\infty}0.
\end{equation}
We now turn to the term in \eqref{eq:dec_phi} involving the fBm, for which we use the classical discrete Kronecker lemma. To this aim, we remark that 
\begin{equation}\label{eq:h}
\ES\left[\sum_{k=1}^{+\infty} \frac{1}{s_k}\int_0^{\gamma_{k+1}} |B_{s_k+t}-B_{s_k}|^{p'}dt\right]=\frac{1}{p'H+1}\sum_{k=1}^{+\infty} \frac{\gamma_{k+1}^{{p'H}+1}}{s_k}<+\infty,
\end{equation}
where the last inequality stems from assumption \eqref{eq:hyp_gamma}. From \eqref{eq:h}, it is easily seen that 
$$\sum_{\ell=0}^{+\infty} \frac{1}{s_\ell}\int_0^{\gamma_{\ell+1}} |B_{s_\ell+t}-B_{s_\ell}|^{p'}dt<+\infty\quad a.s.$$
We are thus in position to apply Kronecker's lemma to the sequence $(\int_0^{\gamma_{\ell+1}} |B_{s_\ell+t}-B_{s_\ell}|^{p'}dt)_{\ell\geqslant0}$. This yields
\begin{equation}\label{eq:i}
\frac{1}{s_N}\sum_{\ell=0}^{N-1} \int_0^{\gamma_{\ell+1}} |B_{s_\ell+t}-B_{s_\ell}|^{p'}dt\xrightarrow{N\rightarrow+\infty}0.
\end{equation}
Let us summarize our computations so far. Gathering \eqref{eq:g}, \eqref{eq:i} and \eqref{eq:dec_phi}, we have obtained relation~\eqref{eq:laststep34}. This easily implies that \eqref{eq:new_convergence_to_prove2} holds true and hence $A_{21,N}(\te)$ defined by \eqref{eq:discrete_A21N} satisfies 
$$\lim\limits_{N\to+\infty}\sup\limits_{\te\in\tte}A_{21,N}(\te)=0.$$
Since similar results have also been shown for  $A_{22,N}(\te)$ and $A_{23,N}(\te)$, relation \eqref{eq:split_A2N} gives
$$\lim\limits_{N\to+\infty}\sup\limits_{\te\in\tte}A_{2,N}(\te)=0.$$
Eventually, plugging this information into \eqref{eq:split_d} yields our claim \eqref{eq:conv_unif_decrease_step}.
\end{proof}

\subsection{A general convergence result}\label{sec:general-convergence}
The consistence of our estimators will rely on the following general proposition about convergence of minimizers for a sequence of random functions. Observe that our sequences below are indexed by a generic $r$ which sits in an unspecified set. This simplifies the subsequent applications of the proposition to our indices $N,n,\ga$ in the remainder of the section.
\begin{proposition}\label{prop:consistency_criteria}
Let $\tte$ be a compact set and $(\te\mapsto L_r(\te))_r$ denote a family of  non-negative random functions. 
Assume that:
 \ben
\item With probability one, ${\lim}_r L_r(\te)=  L(\te)$ uniformly in $\te\in\tte$.
\item The function $\te\mapsto L(\te)$ is non-random and continuous on $\tte$. 
\item For any $r$, the set $~{\rm argmin}\{ L_r(\te),{\te\in\tte}\}~$ is nonempty. 
\een
Then, for a fixed $r$, let $\hat{\te}_r \in {\rm argmin}\{ L_r(\te),{\te\in\tte}\}$. Let ${\cal A}$ denote the limit points of $(\hat{\te}_r)_r$.
 Then we have  $${\cal A}\subset  {\rm argmin}\{ L(\te),{\te\in\tte}\}.$$
 In particular, if $L$ attains its minimum for a unique $\te^\star$, then $\lim_r \hat{\te}_r=\te^\star$.
 \end{proposition}

 \begin{proof}
 Let $\te^\star$ be an element of $~{\rm argmin}\{ L(\te),{\te\in\tte}\}$. We consider a generic element $\te^\infty\in {\cal A}$ and its related convergent subsequence $( \hat{\te}_{r_n})_{n\geqslant0}$.
 Then we can upper bound $L(\te^\infty)$ as follows:
\begin{equation}\label{eq:bound_L_te_infty}
L(\te^\infty)\le L_{r_n}( \hat{\te}_{r_n})+|L(\te^\infty)-L_{r_n}(\hat{\te}_{r_n})|.
\end{equation} 
We now bound the two terms in the right hand side of \eqref{eq:bound_L_te_infty}.
On the one hand, by definition of $\hat{\te}_{r}$,
\begin{equation*}
L_{r_n}(\hat{\te}_{r_n})\le L_{r_n}(\te^\star).
\end{equation*}
Hence, thanks to the fact that $\lim_r L_r(\te)=L(\te)$ for all $\te\in\tte$, we get
 \begin{equation}\label{eq:bound_t1}
 \limsup_{n\to+\infty} L_{r_n}(\hat{\te}_{r_n})\le L(\te^\star).
 \end{equation}
 On the other hand, we also have
\begin{equation}\label{eq:bound_t2}
|L(\te^\infty)-L_{r_n}(\hat{\te}_{r_n})|\le  |L(\te^\infty)-L(\hat{\te}_{r_n})|+ \sup_{\te\in\tte}|L(\te)-L_{r_n}(\te)|.
\end{equation}
Therefore we can invoke the continuity of $L$ to bound the first term in the right hand side of \eqref{eq:bound_t2}, plus the uniform convergence of $L_r$ to $L$ in order to handle the second term. This yealds
\begin{equation}\label{eq:bound_3}
\limsup_{n\to+\infty}|L(\te^\infty)-L_{r_n}(\hat{\te_{r_n}})|=0
\end{equation}
Plugging \eqref{eq:bound_3} and \eqref{eq:bound_t1} into \eqref{eq:bound_L_te_infty}, we obtain that $L(\te^\infty)\le L(\te^\star)$ and thus $\te^{\infty}$ belongs to the set ${\rm argmin}\{ L(\te),{\te\in\tte}\}$. This finishes the proof.
\end{proof}

\subsection{Proofs of the convergence theorems}\label{subsection:proofth}
With all our preliminary considerations in hand, we are now ready to prove the main convergence results for our estimators. This is briefly outlined below.

\begin{proof}[Proof of Theorem \ref{thm:cvgce-esti1}]
Recall that the family $\{\hat{\te}_{N,n,\ga},~ (N,n,\ga)\in\N^2\times \ER_+^*\}$ is defined by \eqref{eq:firstdefesti}.
We wish to apply Proposition \ref{prop:consistency_criteria} with $r=(N,n,\gamma)\in \mathbb{N}^2\times \ER_+^*$. We set $L_{N,n,\gamma}(\te)=d(\frac{1}{n}\sum_{k=0}^{n-1} \delta_{Y_{t_k}},\frac{1}{N}\sum_{k=0}^{N-1} \delta_{Z_{s_k}^{\te,\gamma}} )$.
By Proposition \ref{prop:unifconst}, we have uniformly in $\te\in\tte$,
\begin{equation}\label{def:ldteo}
\lim_{\ga\rightarrow0}\lim_{N,n\rightarrow+\infty} L_{N,n,\gamma}(\te) =d(\nu_{\te_0},\nu_{\te})=:L(\te).
\end{equation}
In addition, owing to Proposition \ref{prop:continuity-Q} and Assumption $\HDEUX$, $L$ is continuous and $\te_0$ is the unique minimum of $L$. We have thus checked that the hypothesis of Proposition \ref{prop:consistency_criteria} are fulfilled, from which Theorem  \ref{thm:cvgce-esti1} is easily deduced.
\end{proof}

\begin{proof}[Proof of Theorem \ref{thm:cvgce-esti2} ]
The proof goes along the same lines as for Theorem \ref{thm:cvgce-esti1}. Namely we apply Proposition \ref{prop:consistency_criteria} to the sequence $\{\hat{\te}_{N,n},~ (N,n)\in\N^2\}$ defined by \eqref{def:hatthetanngamman}. To this aim, we set $$L_{N,n}(\te)=d\left(\frac{1}{n}\sum_{k=0}^{n-1} \delta_{Y_{t_k}},\frac{1}{N}\sum_{k=0}^{N-1} \ga_{k+1}\delta_{Z_{s_k}^{\te}} \right).$$
Then according to Proposition \ref{prop:unifdecreas}, the sequence $(L_{N,n})_{N,n}$ converges uniformly to $L$  defined by  \eqref{def:ldteo} when $N,n\rightarrow+\infty$. Furthermore, the continuity of $L$ follows as in the proof of Theorem \ref{thm:cvgce-esti1}. Our claim is thus easily deduced.
\end{proof}

\begin{proof}[Proof of Theorem \ref{thm:cvgce-estimator-dicrete-theta} ]
We still wish to apply Proposition \ref{prop:consistency_criteria} to the family $\{\hat{\te}^{(\varepsilon)}_{N,n,\ga},~(N,n,\ga)\in\N^2\times \ER_+^*\}$ defined by \eqref{def:hatthetanngamma}. However, since we only assume $\HUNW$ instead of $\HUNS$, one is only able to obtain simple convergence properties on $\tte$.
In order to circumvent this problem, we have restricted our analysis to the discretized parameter set $\tte^{(\varepsilon)}$ introduced in \eqref{def:hatthetanngamma}.
For a given $\varepsilon>0$,
$\tte^{(\varepsilon)}$ is finite and hence, one deduces from Propositions \ref{prop:ergodicSDE}$(ii)$ and \ref{prop:ergodicEuler}, that 
$$\lim\limits_{\ga\to0}\lim\limits_{N,n\to+\infty}\sup_{\te\in \tte^{(\varepsilon)}}|L_{N,n,\gamma}(\te)-L(\te)|=0,$$
where $L$ is defined by \eqref{def:ldteo}. 
Now, denote by ${\cal A}^{(\varepsilon)}$ the set of limit points of $(\hat{\te}_{N,n,\gamma}^{(\varepsilon)})_{N,n,\ga}$. 
From Proposition \ref{prop:consistency_criteria}, one deduces that
$${\cal A}^{(\varepsilon)}\subset  {\rm argmin}\{L(\te),\te\in \tte^{(\varepsilon)}\}.$$
Furthermore, $L$ is a continuous function such that $L(\te_0)=0$. Thus, since ${\rm dist}(\te_0,\tte^{(\varepsilon)})\rightarrow0$ as $\varepsilon\rightarrow0$,
one deduces that  $\min_{\te\in\tte^{(\varepsilon)}} L(\te)\rightarrow0$ as $\varepsilon\rightarrow0$. Owing to $\HDEUX$, this implies that any sequence $(\te^{(\varepsilon)})_{\varepsilon}$ of ${\cal A}^{(\varepsilon)}$
converges to $\te_0$. This concludes the proof.
\end{proof}

\section{Rate of convergence: proof of Theorem \ref{thm:rate_of_convergence}}
All along this section, we assume $\HUNS$ and $\HTROIS$. Our aim is to bound the quantity $\ES[|\hat{\te}_{N,n,\ga}-\te_0|^2]$ where $\hat{\te}_{N,n,\ga}$ is defined by \eqref{eq:firstdefesti}.
Owing to $\HTROIS$, we are reduced to study
\begin{equation}\label{eq:q_error_distance}
\ES\left[d\left(\nu_{\te_0},\nu_{\hat{\te}_{N,n,\ga}}\right)^q\right]
\end{equation}
where $q:=2/\varsigma$ and $\varsigma\in(0,1]$ is given in $\HTROIS$.
Our strategy of proof is based on the following decomposition
\begin{lemma} Let $\hat{\te}_{N,n,\ga}$ be the estimator defined by \eqref{eq:firstdefesti} and recall that $\nu_\te$ is defined by \eqref{eq:defnute} for all $\te\in\tte$. Then $d\left(\nu_{\te_0},\nu_{\hat{\te}_{N,n,\ga}}\right)$ can be decomposed as 
\begin{equation}\label{eq:majo_d_nute0_nuhatte}
d\left(\nu_{\te_0},\nu_{\hat{\te}_{N,n,\ga}}\right)\leq 2D^{(1)}_{n}+2\sup_{\te\in\tte}D^{(2)}_{N,\ga}(\te)+2\sup_{\te\in\tte}D^{(3)}_{N,\ga}(\te)
\end{equation}
where $D^{(1)}_{n}$, $D^{(2)}_{N,\ga}(\te)$, and $D^{(3)}_{N,\ga}(\te)$ are respectively by
\begin{align}
&D^{(1)}_{n}:=d\left(\nu_{\te_0},\frac{1}{n}\sum_{k=0}^{n-1}\delta_{Y_{t_k}}\right),\label{eq:d_term1}\\
&D^{(2)}_{N,\ga}(\te):=d\left(\frac{1}{N}\sum_{k=0}^{N-1}\delta_{Z_{k\ga}^{\te,\ga}},\frac{1}{N}\sum_{k=0}^{N-1}\delta_{Y_{k\ga}^{\te}}\right),\label{eq:d_term2}\\
&D^{(3)}_{N,\ga}(\te):=d\left(\frac{1}{N}\sum_{k=0}^{N-1}\delta_{Y_{k\ga}^{\te}},\nu_{\te}\right).\label{eq:d_term3}
\end{align} 
\end{lemma}

\begin{proof}
Let us write $\hat{\te}$ for $\hat{\te}_{N,n,\ga}$ throughout the proof in order to ease notations. We first apply the triangular inequality, which yields
\begin{equation*}
d\left(\nu_{\te_0},\nu_{\hat{\te}}\right)
\le d\left(\nu_{\te_0},\frac{1}{n}\sum_{k=0}^{n-1}\delta_{Y_{t_k}}\right)
+d\left(\frac{1}{n}\sum_{k=0}^{n-1}\delta_{Y_{t_k}},\frac{1}{N}\sum_{k=0}^{N-1}\delta_{Z_{k\ga}^{\hat{\te},\ga}}\right)
+d\left(\frac{1}{N}\sum_{k=0}^{N-1}\delta_{Z_{k\ga}^{\hat{\te},\ga}},\nu_{\hat{\te}}\right)
\end{equation*}
where $(Y_t)_{t\geq0}$ is the observation process given by \eqref{eq:sde}. Next, we invoke the fact that $\hat{\te}$ minimizes the quantity $d\left(\frac{1}{n}\sum_{k=0}^{n-1}\delta_{Y_{t_k}},\frac{1}{N}\sum_{k=0}^{N-1}\delta_{Z_{k\ga}^{\te,\ga}}\right)$ in $\tte$, which gives
\begin{equation}\label{eq:majo_d_nute0_nuhatte1}
d\left(\nu_{\te_0},\nu_{\hat{\te}}\right)
\le d\left(\nu_{\te_0},\frac{1}{n}\sum_{k=0}^{n-1}\delta_{Y_{t_k}}\right)
+d\left(\frac{1}{n}\sum_{k=0}^{n-1}\delta_{Y_{t_k}},\frac{1}{N}\sum_{k=0}^{N-1}\delta_{Z_{k\ga}^{\te_0,\ga}}\right)
+\sup\limits_{\te\in\tte}d\left(\frac{1}{N}\sum_{k=0}^{N-1}\delta_{Z_{k\ga}^{\te,\ga}},\nu_{\te}\right).
\end{equation}
We further split the second term in the right hand side of \eqref{eq:majo_d_nute0_nuhatte1} as follows:
\begin{align}\label{eq:majo_d_nute0_nuhatte2}
&d\left(\frac{1}{n}\sum_{k=0}^{n-1}\delta_{Y_{t_k}},\frac{1}{N}\sum_{k=0}^{N-1}\delta_{Z_{k\ga}^{\te_0,\ga}}\right)\nonumber\\
&\quad\le d\left(\frac{1}{n}\sum_{k=0}^{n-1}\delta_{Y_{t_k}},\nu_{\te_0}\right)+d\left(\nu_{\te_0},\frac{1}{N}\sum_{k=0}^{N-1}\delta_{Y_{k\ga}^{\te_0}}\right)+d\left(\frac{1}{N}\sum_{k=0}^{N-1}\delta_{Y_{k\ga}^{\te_0}},\frac{1}{N}\sum_{k=0}^{N-1}\delta_{Z_{k\ga}^{\te_0,\ga}}\right),
\end{align}
and resort to a similar decomposition for the third term in the right hand side of \eqref{eq:majo_d_nute0_nuhatte1}. It is then readily checked that plugging \eqref{eq:majo_d_nute0_nuhatte2} into \eqref{eq:majo_d_nute0_nuhatte1} we end up with our claim \eqref{eq:majo_d_nute0_nuhatte}.
\end{proof}
In the remainder of the section, we shall handle the $L^q$-moments of $D^{(1)}_{n}$, $\sup_{\te\in\tte}|D^{(2)}_{N,\ga}(\te)|$ and $\sup_{\te\in\tte}|D^{(3)}_{N,\ga}(\te)|$ separately, respectively in Sections \ref{subsection:first_term}, \ref{subsection:second_term} and \ref{subsection:third_term}.

\subsection{$L^q$ bound on $D^{(1)}_n$}\label{subsection:first_term}
We start this section by giving a notation concerning expectations of empirical measures. 
\begin{notation} Let $Y$ be the solution of equation \eqref{eq:sdetheta} and $t\geq0$. As previously, $\delta_{Y_t}$ denotes the Dirac measure at $Y_t$, considered as a random measure. Then $\be[\delta_{Y_t}]$ is the deterministic measure such that for all continuous and bounded $f:\R^d\to \R$ we have
$$\be[\delta_{Y_t}](f)=\be[f(Y_t)].$$
\end{notation}
With this notation in mind, we can now deliver our $L^q$ estimate for $D^{(1)}_n$.
\begin{lemma}
Let $D^{(1)}_n$ be the random variable defined by \eqref{eq:d_term1}. Then, whenever $d$ is given by $d_{\ccf,p}$ or $d_s$, defined respectively by \eqref{eq:def_CFp} with $p>(q+d)/2$ and \eqref{eq:def_ds}, we have:
\begin{equation}\label{eq:final_term1}
\ES\left[|D^{(1)}_{n}|^q\right]\leqslant C_q\left(n^{-q}+n^{-\frac{q}{2}(2-(2H\vee1))}\right).
\end{equation}
\end{lemma}
\begin{proof}
We decompose $D^{(1)}_n$ as follows: 
\begin{equation}\label{eq:proof_term1}
D^{(1)}_{n}\leq d\left(\nu_{\te_0},\frac{1}{n}\sum_{k=0}^{n-1}\ES[\delta_{Y_{t_k}}]\right)+d\left(\frac{1}{n}\sum_{k=0}^{n-1}\ES[\delta_{Y_{t_k}}],\frac{1}{n}\sum_{k=0}^{n-1}\delta_{Y_{t_k}}\right)=:D^{(11)}_n+D^{(12)}_n.
\end{equation}
For the term $D^{(11)}_n$, we can use the contractivity assumption $\HUNS$ on the drift $b$ which implies that two solutions of the SDE \eqref{eq:sde} with different initial conditions converge exponentially pathwise to each other as $t\to+\infty$ (see e.g. \cite{GKN}).
More specifically, we have already seen in \eqref{eq:conv_exp} that the arguments of \cite{GKN} entail $\Vert Y_t-\bar{Y}_t\Vert_{L^p(\Omega}\leq c_1e^{-c_2 t}$ for two positive constants $c_1,c_2$, where we recall that $\bar{Y}$ designates the stationary solution of \eqref{eq:sdetheta}. Hence for a Lipschitz function $f:\R^d\to\R$ we easily get the existence of a constant $C>0$ such that
\begin{equation}\label{eq:term1_traj}
\left|\frac{1}{n}\sum_{k=0}^{n-1}\ES[f(Y_{t_k})]-\nu_{\te_0}(f)\right|\leq\frac{1}{n}\sum_{k=0}^{n-1}\ES[|f(Y_{t_k})-f(\bar{Y}_{t_k})|]\leq\frac{C}{n}\|f\|_{\rm Lip}.
\end{equation}
It remains to take into account the distance $d$. Recall that we only consider the two distances $d_{\ccf,p}$ and $d_s$ defined in Subsection \ref{sec:rateofconvergence}.
We thus easily deduce from \eqref{eq:term1_traj} and the definitions of $d_{\ccf,p}$ and $d_s$ that there exists a positive constant $\tilde{C}$ such that
\begin{equation}\label{eq:term1_traj_dist}
D^{(11)}_n\leq\max\left\{d_{\ccf,p}\left(\nu_{\te_0},\frac{1}{n}\sum_{k=0}^{n-1}\ES[\delta_{Y_{t_k}}]\right),d_s\left(\nu_{\te_0},\frac{1}{n}\sum_{k=0}^{n-1}\ES[\delta_{Y_{t_k}}]\right)\right\}\leq \frac{\tilde{C}}{n}.
\end{equation}
The term $D^{(12)}_n$ is handled in Proposition \ref{prop:term1_concentration_dist} below and specifically in relation \eqref{eq:term1_concentration_dist}. Therefore, plugging \eqref{eq:term1_traj_dist} and \eqref{eq:term1_concentration_dist} into \eqref{eq:proof_term1}, relation \eqref{eq:final_term1} is proved.
\end{proof}
\begin{proposition}\label{prop:term1_concentration_dist}
Let $d$ be one of the two distances $d_s$ and $d_{\ccf,p}$ with $p>(q+d)/2$. Then,
\begin{equation}\label{eq:term1_concentration_dist}
\ES\left[d\left(\frac{1}{n}\sum_{k=0}^{n-1}\delta_{Y_{t_k}},\frac{1}{n}\sum_{k=0}^{n-1}\ES[\delta_{Y_{t_k}}]\right)^q~\right]\leq C_qn^{-\frac{q}{2}(2-(2H\vee1))}.
\end{equation}
\end{proposition}
The proof of the proposition is based on the following lemma:
\begin{lemma}\label{lem:moment_concentration} Recall that $(Y_t)_{t\geq0}$ is given by \eqref{eq:sde}. Then for all $q\geqslant1$ and for all Lipschitz function $f:\R^d\to\R$,
\begin{equation}\label{eq:moment_concentration}
\ES\left[\left| \frac{1}{n}\sum_{k=0}^{n-1}f(Y_{t_k})-\ES[f(Y_{t_k})]\right|^q~\right]\leqslant C_q\|f\|_{\rm Lip}^q n^{-\frac{q}{2}(2-(2H\vee1))}.
\end{equation}
\end{lemma}
\begin{proof}We invoke a concentration result for large time borrowed from \cite[Theorem 2.3]{Var}. This result asserts that: there exists $C>0$ such that for all Lipschitz functions $f:\R^d\to\R$ and for all $r\geqslant0$, 
\begin{equation}\label{eq:term1_concentration}
\PE\left( \frac{1}{n}\sum_{k=0}^{n-1}f(Y_{t_k})-\ES[f(Y_{t_k})]\ge r\right)\le C\exp(-C\|f\|^{-2}_{\rm Lip}r^{2}n^{2-(2H\vee1)}).
\end{equation}
Therefore, one can check that \eqref{eq:moment_concentration} holds true by plugging inequality \eqref{eq:term1_concentration} into the classical formula $$\ES[X^q]=\int_0^{+\infty}qx^{q-1}\PE(X\geq x)dx,$$ which is valid for any positive random variable $X$.
\end{proof}

\begin{proof}[Proof of Proposition \ref{prop:term1_concentration_dist}]We will only give details about $d_{\ccf,p}$ since $d_s$ can be treated exactly along the same lines. Furthermore, since our parameter $q$ is greater than $2$, by using Jensen inequality and the linearity of $\ES$ into the definition \eqref{eq:def_CFp} of $d_{\ccf,p}$, we get:
$$\ES\left[d_{\ccf,p}\left(\nu_{\te_0},\frac{1}{n}\sum_{k=0}^{n-1}\ES[\delta_{Y_{t_k}}]\right)^q~\right]\leqslant\int_{\R^d}\ES\left[\left| \frac{1}{n}\sum_{k=0}^{n-1}f_\xi(Y_{t_k})-\ES[f_\xi(Y_{t_k})]\right|^q\right]g_p(\xi)d\xi,$$
where $f_\xi(x)=e^{i\langle\xi,x\rangle}$. Since $\|f_{\xi}\|_{\rm Lip}\leq |\xi|$, we thus deduce from Lemma \ref{lem:moment_concentration} that
$$\ES\left[d_{\ccf,p}\left(\nu_{\te_0},\frac{1}{n}\sum_{k=0}^{n-1}\ES[\delta_{Y_{t_k}}]\right)^q~\right]\leqslant C_q n^{-\frac{q}{2}(2-(2H\vee1))}\int_{\R^d}|\xi|^qg_p(\xi)d\xi.$$
The integral in the last inequality is finite owing to the fact that we chose $p>(q+d)/2$. Our claim thus follows.
\end{proof}

\subsection{$L^q$ bound on $D^{(2)}_{N,\ga}$}\label{subsection:second_term}
Our aim in this section is to get an equivalent of relation \eqref{eq:final_term1} for the term $D^{(2)}_{N,\ga}$ defined by \eqref{eq:d_term2}, namely 
\begin{equation}\label{eq:final_term2}
\ES\left[\sup\limits_{\te\in\tte}|D^{(2)}_{N,\ga}(\te)|^q\right]\leq C\ga^{qH}.
\end{equation}
where the distance $d$ in the definition of $D^{(2)}_{N,\ga}$ is either $d_s$ or $d_{\ccf,p}$. To this end, resorting to the fact that $d_{\ccf,p}$ and $d_s$ are both elements of $\mathcal{D}_1$, the quantity \eqref{eq:d_term2} can be upper-bounded as follows:
$$\sup\limits_{\te\in\tte}D^{(2)}_{N,\ga}(\te)\leq \frac{1}{N}\sum_{k=0}^{N-1}\sup\limits_{\te\in\tte}|Y_{k\ga}^\te-Z_{k\ga}^{\te,\ga}|.$$
We thus deduce that
$$\ES\left[\sup\limits_{\te\in\tte}|D^{(2)}_{N,\ga}(\te)|^q\right]\leq \ES\left[\left(\frac{1}{N}\sum_{k=0}^{N-1}\sup\limits_{\te\in\tte}|Y_{k\ga}^\te-Z_{k\ga}^{\te,\ga}|\right)^{q}\right].
$$
Recall that $q\geq2$. Hence a direct application of Jensen's inequality gives
\begin{equation}\label{eq:majo_jensen_term2}
\ES\left[\sup\limits_{\te\in\tte}|D^{(2)}_{N,\ga}(\te)|^q\right]\leq \frac{1}{N}\sum_{k=0}^{N-1}\ES\left[\sup\limits_{\te\in\tte}|Y_{k\ga}^\te-Z_{k\ga}^{\te,\ga}|^q\right].
\end{equation}
Now according to  Proposition \ref{prop:rapdisccont} $(i)$ and Proposition \ref{prop:control_moment_Z} $(ii)$ (see also the proof of Proposition \ref{prop:unifconst} $(i)$), it is readily checked that
$$\sup_{k\geq0}\ES\left[\sup\limits_{\te\in\tte}|Y_{k\ga}^\te-Z_{k\ga}^{\te,\ga}|^{q}\right]\leqslant C\ga^{qH}.$$
Gathering this information with \eqref{eq:majo_jensen_term2}, inequality \eqref{eq:final_term2} is easily deduced.

\subsection{$L^q$ bound on $D^{(3)}_{N,\ga}$}\label{subsection:third_term}
The quantity \eqref{eq:d_term3} is the hardest to treat among the terms in our decomposition \eqref{eq:majo_d_nute0_nuhatte}, due to the fact that we wish to achieve a uniform bound in $\te$. We summarize our analysis in the following lemma.
\begin{lemma} Let $D^{(3)}_{N,\ga}$ be the random variable defined by \eqref{eq:d_term3}, and assume that $d$ is either $d_s$ or $d_{\ccf,p}$ with $p>(q+d)/2$. Then, we have
\begin{equation}\label{eq:final_term3}
\ES\left[\sup_{\te\in\tte}D^{(3)}_{N,\ga}(\te)^q\right]\leq C_q\left(\ga^{qH}+T^{-\tilde{\eta}}\right)
\end{equation}
with $\tilde{\eta}:=\frac{q^2}{2(q+d)}(2-(2H\vee1))$ and $T=N\ga$.
\end{lemma}

\begin{proof}
We will further decompose the term $D^{(3)}_{N,\ga}$ and then divide our analysis in several steps.
First, let us introduce some notations: denote by $T$ the quantity $N\ga$ and for all $t\in[0,T]$, set $\underline{t}:=\inf\{k\ga~|~k\ga\leqslant t<(k+1)\ga\}$ as we did in the proof of Proposition \ref{prop:unifdecreas}. With this notations in hand, we have $D^{(3)}_{N,\ga}(\te)=d(\nu_\te, \frac{1}{T}\int_0^T\delta_{Y^\te_{\underline{t}}}dt)$ from which we deduce the following decomposition:
\begin{equation}\label{eq:majo_term3}
\sup_{\te\in\tte}D^{(3)}_{N,\ga}(\te)\leq \sup_{\te\in\tte} D^{(31)}_{N,\ga}(\te)+\sup_{\te\in\tte} D^{(32)}_{N,\ga}(\te)
\end{equation}
where 
$$D^{(31)}_{N,\ga}(\te)=d\left(\nu_\te, \frac{1}{T}\int_0^T\delta_{Y^\te_{t}}dt\right)\quad\text{ and }\quad D^{(32)}_{N,\ga}(\te)=d\left(\frac{1}{T}\int_0^T\delta_{Y^\te_{t}}dt, \frac{1}{T}\int_0^T\delta_{Y^\te_{\underline{t}}}dt\right).$$
We will now handle those two terms separately:\\
\textbf{Step 1}: Bound on $D^{(32)}_{N,\ga}$.
As in Section \ref{subsection:second_term}, since $d\in\mathcal{D}_1$ whenever $d=d_s$ or $d=d_{\ccf,p}$, we have
\begin{equation}
\sup_{\te\in\tte}D^{(32)}_{N,\ga}(\te)\leq \frac{1}{T}\int_0^T\sup_{\te\in\tte}|Y^\te_{t}-Y^\te_{\underline{t}}|dt.
\end{equation}
We now proceed as in Section \ref{subsection:second_term} in order to get the equivalent  of \eqref{eq:majo_jensen_term2} thanks to Jensen's inequality. We get 
\begin{equation}\label{eq:majo_jensen_term3}
\ES\left[\sup_{\te\in\tte}|D^{(32)}_{N,\ga}(\te)|^q\right]\leq \frac{1}{T}\int_0^T\ES\left[\sup_{\te\in\tte}|Y^\te_{t}-Y^\te_{\underline{t}}|^q\right]dt.
\end{equation}
In order to bound the right hand side of \eqref{eq:majo_jensen_term3}, we start by recalling the bound \eqref{eq:c1} for $Y^\te_t-Y^\te_{\underline{t}}$:
$$|Y^\te_t-Y^\te_{\underline{t}}|\leq\int_{\underline{t}}^t|b_\te(Y^\te_s)|ds+\|\sigma\||B_t-B_{\underline{t}}|.$$
The drift term above is now bounded thanks to the sublinear growth of $b_\te$ given by \eqref{eq:inward-weak} and the uniform bound on the $L^q$ moments of $Y_t^\te$ given by Proposition \ref{prop:controlmoment}. As far as the term $|B_t-B_{\underline{t}}|$ is concerned , we obviously have thanks to \eqref{eq:var-increments-fbm} and the fact that $|t-\underline{t}|\leq\ga$:
$$\ES\left[|B_t-B_{\underline{t}}|^q\right]\leq C_H\ga^{qH}.$$
From here, it is readily checked that 
\begin{equation*}
\ES\left[\sup_{\te\in\tte}|Y^\te_t-Y^\te_{\underline{t}}|^q\right]\leq C_H~\ga^{qH}.
\end{equation*}
Plugging this information into \eqref{eq:majo_jensen_term3} we end up with
\begin{equation}\label{eq:term3_2}
\ES\left[\sup_{\te\in\tte}|D^{(32)}_{N,\ga}(\te)|^q\right]\leq C_q\ga^{qH}.
\end{equation}
\textbf{Step 2}: Bound on $D^{(31)}_{N,\ga}(\te)$ for a fixed $\te$.
For a fixed value of $\te\in\tte$, the term $D^{(31)}_{N,\ga}(\te)$ will be handled similarly to Section \ref{subsection:first_term}. Namely, along the same lines as for relation \eqref{eq:proof_term1} we write
\begin{equation}\label{eq:proof_term2}
D^{(31)}_{N,\ga}(\te)\leq d\left(\nu_\te, \frac{1}{T}\int_0^T\ES[\delta_{Y^\te_{t}}]dt\right)+d\left(\frac{1}{T}\int_0^T\ES[\delta_{Y^\te_{t}}]dt, \frac{1}{T}\int_0^T\delta_{Y^\te_{t}}dt\right).
\end{equation}
Then the first term in the right hand side of \eqref{eq:proof_term2} is handled exactly as \eqref{eq:term1_traj} in Section \ref{subsection:first_term}, which yields
\begin{equation}\label{eq:term3_traj}
\left|\frac{1}{T}\int_0^T\ES[f(Y^\te_{t})]dt-\nu_{\te}(f)\right|\leqslant\frac{C}{T}\|f\|_{\rm Lip}.
\end{equation}
The second term in the right hand side of \eqref{eq:proof_term2} can be upper bounded thanks to a continuous time version of Lemma \ref{lem:moment_concentration} (also based on \cite[Theorem 2.3]{Var} and left to the reader for the sake of conciseness). We get 
\begin{equation}\label{eq:moment_concentration_T}
\ES\left[\left| \frac{1}{T}\int_{0}^{T}\left(f(Y^\te_{t})-\ES[f(Y^\te_{t})]\right)dt\right|^q~\right]\leqslant C_q\|f\|_{\rm Lip}^q T^{-\frac{q}{2}(2-(2H\vee1))}.
\end{equation}
Therefore putting together \eqref{eq:term3_traj} and \eqref{eq:moment_concentration_T} and arguing as in Section \ref{subsection:first_term}, we get that for $d=d_{\ccf,p}$ with $p>(q+d)/2$ or $d=d_s$ and for any $\te\in\tte$,
\begin{equation}\label{eq:term3_1}
\ES\left[|D^{(31)}_{N,\ga}(\te)|^q\right]=\ES\left[d\left(\nu_\te, \frac{1}{T}\int_0^T\delta_{Y^\te_{t}}dt\right)^q~\right]\leqslant C_q\left(T^{-q}+T^{-\frac{q}{2}(2-(2H\vee1))}\right)
\end{equation}
where $C_q$ is a positive constant which does not depend on $\te$.\\
\textbf{Step 3}: Bound on $\sup_{\te\in\tte}D^{(31)}_{N,\ga}(\te)$. In order to gor from \eqref{eq:term3_1} to a bound for the supremum over $\tte$, we proceed to a discretization of the parameter space $\tte$ as in Section \ref{subsection:main_consistency_results}. Towards this aim, we will use the following notation: for any $\te\in\tte$, we set
\begin{equation}\label{eq:def_phi}
\varphi(\te):=d\left(\nu_\te, \frac{1}{T}\int_0^T\delta_{Y^\te_{t}}dt\right).
\end{equation}
Let $\varepsilon>0$ and recall that $\tte^{(\varepsilon)}:=\{\te_{i}^{(\varepsilon)}~|~1\leqslant i \leqslant M_\varepsilon\}$ is defined at the beginning of Subsection \ref{subsection:main_consistency_results} in such a way that $\tte\subset \bigcup_{i=1}^{M_\varepsilon}B(\te_{i}^{(\varepsilon)},\varepsilon)$.
Then, for any $\te\in\tte$,
\begin{equation*}
\varphi(\te)\leqslant |\varphi(\te)-\varphi(\te^{(\varepsilon)})|+|\varphi(\te^{(\varepsilon)})|
\end{equation*}
where $\te^{(\varepsilon)}$ is defined by \eqref{eq:discretize-parameter-space}. Therefore
\begin{equation*}
\varphi(\te)\leq|\varphi(\te)-\varphi(\te^{(\varepsilon)})|+\max\limits_{1\leq i\leq M_\varepsilon}|\varphi(\te_i^{(\varepsilon)})|
\end{equation*}
and finally
\begin{align}\label{eq:disc_te}
\ES\left[\sup_{\te\in\tte}\varphi(\te)^q\right]&\leq c_q\ES\left[\sup_{\te\in\tte}|\varphi(\te)-\varphi(\te^{(\varepsilon)})|^q\right]+c_q\ES\left[\max\limits_{1\leq i\leq M_\varepsilon}|\varphi(\te_i^{(\varepsilon)})|^q\right]\nonumber\\
&\leq c_q\ES\left[\sup_{\te\in\tte}|\varphi(\te)-\varphi(\te^{(\varepsilon)})|^q\right]+c_q\sum_{i=1}^{M_\varepsilon}\ES\left[|\varphi(\te_i^{(\varepsilon)})|^q\right].
\end{align}
Owing to inequality \eqref{eq:term3_1} for a fixed $\te\in\tte$, we can deduce from \eqref{eq:disc_te} that
\begin{equation}\label{eq:disc_te2}
\ES\left[\sup_{\te\in\tte}\varphi(\te)^q\right]\leq c_q\ES\left[\sup_{\te\in\tte}|\varphi(\te)-\varphi(\te^{(\varepsilon)})|^q\right]+c'_q M_\varepsilon\left(T^{-q}+T^{-\frac{q}{2}(2-(2H\vee1))}\right).
\end{equation}
In the remainder of the step, we thus focus on the first right hand term in \eqref{eq:disc_te2}. Namely we will show the existence of an integrable random variable $\zeta^{}_T>0$ such that for all $\te_1,\te_2\in\tte$ we have
\begin{equation}\label{eq:varphi_lip}
|\varphi(\te_1)-\varphi(\te_2)|\leq \zeta^{}_T|\te_1-\te_2|\quad a.s.
\end{equation}
For this purpose, let us split the quantity $|\varphi(\te_1)-\varphi(\te_2)|$ in two terms
\begin{equation}\label{eq:varphi_lip2}
|\varphi(\te_1)-\varphi(\te_2)|\leq d\left(\nu_{\te_1},\nu_{\te_2}\right)+d\left( \frac{1}{T}\int_0^T\delta_{Y^{\te_1}_{t}}dt, \frac{1}{T}\int_0^T\delta_{Y^{\te_2}_{t}}dt\right).
\end{equation}
Then, one can show the following inequalities for any $d\in\mathcal{D}_1$
\begin{eqnarray}
&d\left(\nu_{\te_1},\nu_{\te_2}\right)\leq C|\te_1-\te_2|\sup_{\te\in\tte}\nu_\te\left(|\cdot|^{r}\right)\label{eq:nu_te_lip}\\
&d\left( \frac{1}{T}\int_0^T\delta_{Y^{\te_1}_{t}}dt,\frac{1}{T}\int_0^T\delta_{Y^{\te_2}_{t}}dt\right)\leq C|\te_1-\te_2|\left(\frac{1}{T}\int_0^T\sup_{\te\in\tte}|Y^\te_s|^{r}ds\right) \label{eq:mes_oc_lip}
\end{eqnarray}
where $r$ is given in assumption $\HUNS$ and where equation \eqref{eq:nu_te_lip} is obtained by following the proof of Proposition \ref{prop:contracttt} (see appendix \ref{append:B}) for the stationary solutions $\bar{Y}^{\te_1}$ and $\bar{Y}^{\te_2}$. Plugging \eqref{eq:nu_te_lip} and \eqref{eq:mes_oc_lip} into \eqref{eq:varphi_lip2}, we obtain that \eqref{eq:varphi_lip} holds true with $$\zeta^{}_T:=C\max\left\{\sup_{\te\in\tte}\nu_\te\left(|\cdot|^{r}\right),~\frac{1}{T}\int_0^T\sup_{\te\in\tte}|Y^\te_s|^{r}ds\right\}.$$
By Proposition \ref{prop:controlmoment} $(i)$, we get $\sup_{T>0}\ES[\zeta^{~q}_T]<+\infty$. Then plugging \eqref{eq:varphi_lip} into \eqref{eq:disc_te2} we end up with
\begin{equation}\label{eq:majo_eps}
\ES\left[\sup_{\te\in\tte}\varphi(\te)^q\right]\leq c_q\sup_{T>0}\ES[\zeta^{~q}_T]\varepsilon^q+c'_q M_\varepsilon\left(T^{-q}+T^{-\frac{q}{2}(2-(2H\vee1))}\right).
\end{equation} 
For all $\varepsilon>0$, since $\tte$ is a compact of $\R^d$, we can choose $M_\varepsilon\leq \frac{C_{\tte}}{\varepsilon^{d}}$ and then if we choose $\varepsilon:=T^{-\eta}$ for some $\eta>0$, we finally get
\begin{align*}
\ES\left[\sup_{\te\in\tte}\varphi(\te)^q\right]&\leq c_q\sup_{T>0}\ES[\zeta^{~q}_T]T^{-q\eta}+c'_q C_{\tte}\left(T^{-q+d\eta}+T^{-\frac{q}{2}(2-(2H\vee1))+d\eta}\right)\nonumber\\
&\leq c_q\sup_{T>0}\ES[\zeta^{~q}_T]T^{-q\eta}+c'_q C_{\tte}T^{-\frac{q}{2}(2-(2H\vee1))+d\eta}.
\end{align*}
It just remains to optimize in $\eta$ to conclude that:
\begin{equation}\label{eq:disc_te_final}
\ES\left[\sup_{\te\in\tte}\varphi(\te)^q\right]\leq C_qT^{-\tilde{\eta}}
\end{equation}
with $\tilde{\eta}:=\frac{q^2}{2(q+d)}(2-(2H\vee1))$.
By putting together \eqref{eq:majo_term3} with both \eqref{eq:term3_2} and \eqref{eq:disc_te_final}, this concludes the proof of \eqref{eq:final_term3}.
\end{proof}

Let us now conclude the section. Through inequality \eqref{eq:majo_d_nute0_nuhatte} and the control of the three right hand side terms, namely \eqref{eq:final_term1}, \eqref{eq:final_term2} and \eqref{eq:final_term3}, we are in position to conclude that Theorem \ref{thm:rate_of_convergence} holds true.

\section{Identifiability assumption}\label{sec:identif-inv-measure}
In this section we will provide some examples of equations of the form \eqref{eq:sde} for which the crucial assumptions \eqref{eq:I_w} and \eqref{eq:I_s} are satisfied. We first review briefly the diffusion case in Section \ref{subsection:diffusion_case}, and then give a particular example in the fractional Brownian motion case in Section \ref{subsection:fBm_case}.

\subsection{Case of a diffusion process}\label{subsection:diffusion_case}

In this section we consider equation \eqref{eq:sde} in the case $H=\frac{1}{2}$, that is when the equation is driven by a $d$-dimensional Wiener process. Our considerations are summarized in the following proposition.
\begin{proposition} Consider equation \eqref{eq:sde} in the case $H=\frac{1}{2}$. We assume Hypothesis $\HZERO$ and $\HUNW$ to be met and call $\nu_\te$ the invariant measure corresponding to the coefficient $b_\te$. We pick $\te_1,\te_2\in\tte$ and set 
$$F={\rm Span}\{\nabla f, \textnormal{ for all }f\in {\cal C}^2(\ER^d,\ER)\;\textnormal{with compact support}\}$$
where $F$ is considered as a subspace of $L^2(\nu_{\te_1})$. We assume that $b_{\te_1}-b_{\te_2}$ is not an element of $F^{\perp}$ in $L^2(\nu_{\te_1})$. Then
$$d(\nu_{\te_1},\nu_{\te_2})>0.$$
If this condition is satisfied for all couples $(\te_1,\te_2)\in\tte^2$, then $\HDEUX$ holds true.
\end{proposition}

\begin{proof} For $\te\in\tte$, let $L_\te$ denote the linear operator defined on ${\cal C}^2(\ER^d,\ER)$ by:
$$L_\te f(x)=\langle \nabla f, b_\te\rangle (x)+\frac{1}{2}( \sigma^*D^2 f(x) \sigma), $$
where $D^2 f$ denote the Hessian matrix of $f$. By a classical criterion, $\nu_\te$ is invariant for \eqref{eq:sde} when $H=\frac{1}{2}$ if and only if 
$\nu_\te(L_\te f)=0$ for any $f\in {\cal C}^2(\ER^d,\ER)$ with compact support. As a consequence, 
$\nu_{\te_1}=\nu_{\te_2}$ if and only if $\nu_{\te_1}((L_{\te_1}-L_{\te_2}) f)=0$ for any compactly supported ${\cal C}^2$-function $f:\ER^d\rightarrow\ER$. Now observe that
 $$\nu_{\te_1}((L_{\te_1}-L_{\te_2}) f)=\int \langle\nabla f(x),b_{\te_1}(x)-b_{\te_2}(x)\rangle \nu_{\te_1}(dx).$$
 The result follows.
\end{proof}
In other words, this result says that in the diffusion setting, the identifiability assumption is true if ${\rm Proj}_F(b_\te-b_{\te_0})$ is not the null function for any $\te\neq\te_0$. Notice that this is always true  in the one-dimensional case or if $b_\te$ is a gradient. Unfortunately, the generalization of this simple characterization to SDEs driven by fBm is far from being straightforward.

\subsection{Fractional Brownian motion case}\label{subsection:fBm_case}
In this section we wish to check $\HTROIS$ for some specific examples of equation \eqref{eq:sdetheta} and for the distance $\dcf$. Specifically, we shall consider a family $Y^{\te,\lambda}$ of real valued processes defined by
\begin{equation}\label{eq:def-Y-theta-alpha}
d Y^{\lambda,\te}_{t}
=
\lc  -\te \, Y^{\lambda,\te}_{t} + \lambda b_{\te}(Y^{\lambda,\te}_{t}) \rc \, dt + \si \, d B_{t}
\end{equation}
where $B$ is a $1$-dimensional fractional Brownian motion.
In equation \eqref{eq:def-Y-theta-alpha} the quantity $\lambda$ is a small enough parameter, which is assumed to be known. The estimation procedure is still for $\te$ only. The coefficient $b_{\te}$ is bounded together with its derivatives with respect to $y$ and $\te$. The process $Y^{\te,\lambda}$ has to be seen as a small perturbation of a fractional Ornstein-Uhlenbeck process with parameter $\te$. We also assume that $\te$ is a $1$-dimensional parameter and:
\begin{equation}\label{hyp:compact-theta}
\te\in[m,M],
\quad\text{with}\quad
0<m<M<\infty.
\end{equation}

Let us start our analysis by the case $X^{\te}\equiv Y^{0,\te}$, that is the fractional Ornstein-Uhlenbeck process itself, solution of the following equation:
\begin{equation}\label{eq:def-fou}
d X^{\te}_{t}
=
  -\te \, X^{\te}_{t}  \, dt + \si \, d B_{t}.
\end{equation}
It is easily seen that $X^\te$ is a centered Gaussian process whose variance is given (see e.g. \cite[p.724]{Ha}) by
\begin{equation}\label{eq:variance_fou}
\ES[(X^\te_t)^2]=2\sigma^2e^{-\te t}\int_0^ts^{2H-1}\cosh(\te(t-s))ds.
\end{equation}
In this case our assumption $\HTROIS$ is easily satisfied, as shown in the following lemma.

\begin{lemma}\label{lem:dcf-fou} Let $\te\in[m,M]$ as in \eqref{hyp:compact-theta}, and consider the fractional Ornstein-Uhlenbeck process $X^\te$ defined by \eqref{eq:def-fou}. We call $\mu_\te$ its invariant measure. Then for all $\te_1,\te_2\in[m,M]$, we have
\begin{equation}\label{eq:low-bnd-dcf-fou}
\dcf(\mu_{\te_{1}}, \mu_{\te_{2}})
\ge
c_{m,M,H} \, |\te_{1}-\te_{2}|.
\end{equation}
\end{lemma}

\begin{proof}
It is well-known (see e.g \cite{BK}) that for the fractional Ornstein-Uhlenbeck process we have $\mu_{\te}=\cn(0,\si_{\te}^{2})$, with
\begin{equation}\label{eq:variance_mes-inv_fou}
\si_{\te}^{2} = \frac{c_{H}}{\te^{2H}}.
\end{equation}
Taking expression \eqref{eq:def_CFp} into account, this yields
\begin{equation*}
\dcf^{2}(\mu_{\te_{1}}, \mu_{\te_{2}})
=
\int_{\R}\left[\exp\lp -\frac{c_{H}}{2\te_{1}^{2H}}\xi^{2}\rp - \exp\lp-\frac{c_{H}}{2\te_{2}^{2H}}\xi^{2}\rp \right]^2g_p(\xi)d\xi,
\end{equation*}
from which our claim \eqref{eq:low-bnd-dcf-fou} is easily proved. Notice that the fact that $\te$ is bounded away from $0$ is crucial here in order to ensure the continuity of $\te\mapsto\sigma_\te$ in \eqref{eq:variance_mes-inv_fou} on the interval $[m,M]$.
\end{proof}

Let us also state an elementary bound on ordinary differential equations for further use.

\begin{lemma}\label{lem:easy-ode}
Let $f,g:\R_{+}\to\R$ be two functions such that there exist some constants $\ka,M>0$ satisfying
\begin{equation}\label{a1}
f_{r} \ge \ka,
\quad\text{and}\quad
|g_{r}| \le M,
\quad\text{for all } r\in\R_{+}.
\end{equation}
Let $y$ be the solution of the following differential equation:
\begin{equation}\label{eq:linear-ode}
\dot{y}_{t} + f_{t} \, y_{t} = g_{t}.
\end{equation}
Then $y$ is uniformly bounded in $t$ and verifies
\begin{equation*}
|y_{t}| \le \frac{M}{\ka}.
\end{equation*}
\end{lemma}

\begin{proof}
Equation \eqref{eq:linear-ode} admits an explicit solution under the form
\begin{equation*}
y_{t} 
=
\iot 
\exp\lp -\ist f_{r} \, dr  \rp 
g_{s} \, d s.
\end{equation*}
Plugging the bounds \eqref{a1} into the above expression, we easily get
\begin{equation*}
|y_{t}| \le
M \iot e^{-\ka(t-s)} \, d s 
\le
\frac{M}{\ka},
\end{equation*}
which is our claim.
\end{proof}

We now wish to extend Lemma \ref{lem:dcf-fou} to the model given by equation \eqref{eq:def-Y-theta-alpha}. Namely we wish to prove the following proposition.
\begin{proposition}\label{prop:dcf-perturbed-fou} 
Let $Y^{\lambda,\te}$ be the process defined by \eqref{eq:def-Y-theta-alpha} and consider $p>3/2$. We assume $\te\in[m,M]$ and $\lambda\in(0,\lambda_0)$ with a small enough $\lambda_0=\lambda_0(m,M,p)$. Also assume (without loss of generality) that $b_\te$, $\partial_yb_\te$, $\partial_\te b_\te$, $\partial^2_{\te y}b_\te$ are all bounded by $1$. Then the following lower bound holds true for any $\te_1,\te_2\in[m,M]$:
\begin{equation}\label{eq:low-bnd-dcf-perturbed-fou}
\dcf(\nu_{\te_{1}}, \nu_{\te_{2}})
\ge
c_{m,M,H} \, |\te_{1}-\te_{2}|.
\end{equation}
\end{proposition}

\begin{proof} Owing to the definition \eqref{eq:def_CFp} of the distance $d_{\ccf,p}$, we have
\begin{equation*}
\dcf(\nu_{\te_{1}}, \nu_{\te_{2}})
=
\left\{\int_{\R}\left(\be[e^{i \xi \by^{\lambda,\te_{1}}}]-\be[e^{i \xi \by^{\lambda,\te_{2}}}]\right)^2 \, g_p(\xi) \, d\xi\right\}^{1/2}
\end{equation*}
We will decompose this quantity as follows:
\begin{equation}\label{eq:fou1}
\dcf(\nu_{\te_{1}}, \nu_{\te_{2}})
\ge
I_{3}^{1/2} - \lp I_{2}^{1/2} +  I_{11}^{1/2} + I_{12}^{1/2} \rp,
\end{equation}
where the quantities $I_{1j}$ and $I_3$ are defined by
\begin{eqnarray*}
I_{1j}
&=&
\int_{\R}\left(\be\lc \exp\lp i \xi \by^{\lambda,\te_{j}} \rp\rc-\be\lc \exp\lp i \xi Y_{t}^{\lambda,\te_{j}} \rp\rc\right)^2 \, g_p(\xi) \, d\xi  \\
I_{3}
&=&
\int_{\R}\left(\be\lc \exp\lp i \xi X_{t}^{\te_{1}} \rp\rc-\be\lc \exp\lp i \xi X_{t}^{\te_{2}} \rp\rc\right)^2 \, g_p(\xi) \, d\xi ,
\end{eqnarray*}
and where we recall that the Ornstein-Uhlenbeck process $X^\te$ is given by \eqref{eq:def-fou}. In equation \eqref{eq:fou1}, we also have
\begin{multline*}
I_{2}
=
\int_{\R}\Big(\be\lc \exp\lp i \xi Y_{t}^{\lambda,\te_{1}} \rp\rc 
- \be\lc \exp\lp i \xi X_{t}^{\te_{1}} \rp\rc \\
-\be\lc \exp\lp i \xi Y_{t}^{\lambda,\te_{2}} \rp\rc 
+ \be\lc \exp\lp i \xi X_{t}^{\te_{2}} \rp\rc \Big)^2
\, g_p(\xi) \, d\xi .
\end{multline*}
In the definitions above, $t$ is an arbitrarily large time, to be determined later on. Our goal is now to lower bound $I_{3}$ and upper bound $I_{1,j}$ and $I_{2}$.

\paragraph{\emph{Lower bound for $I_{3}$.}} In order to lower bound $I_3$, we proceed as in Lemma \ref{lem:dcf-fou}. Indeed, Lemma \ref{lem:dcf-fou} stems from a lower bound on
$$|\ES[e^{i\xi X^{\te_1}_{\infty}}]-\ES[e^{i\xi X^{\te_2}_{\infty}}]|,$$ while we are interested here in a lower bound on $$|\ES[e^{i\xi X^{\te_1}_{t}}]-\ES[e^{i\xi X^{\te_2}_{t}}]|,$$ for a fixed $t$. However it is readily checked from \eqref{eq:variance_fou} that there exists $t_0>0$ such that for all $t\geq t_0$ we have
\begin{equation}\label{eq:I3}
I_{3} \geq c_{1} \, |\te_{1}-\te_{2}|^{2}
\end{equation}
with a given constant $c_1>0$ depending on $m,M$.

\paragraph{\emph{Upper bound for $I_{1j}$.}} Recall that both $b_\te$ and $\partial_y b_\te$ are bounded by $1$. We also assume that $\lambda$ is small enough so that $\lambda\leq m(1-\varepsilon)$ with $\varepsilon>0$. Then it is readily checked that $x\mapsto -\te x+\lambda b_\te(x)$ satisfies the condition $\HUNS$. Hence one can see as in \eqref{eq:conv_exp} that $$I_{1j}\leq Ce^{-\frac{m\varepsilon}{2}t}.$$
If we wish to have $I_{1j} \leq c_{3} \, |\te_{1}-\te_{2}|^{2},$ with $c_3$ arbitrarily small, it is thus sufficient to pick $t\geq t_1$ with $t_1=C\log\left(\frac{1}{|\te_1-\te_2|^2}\right)$. In the sequel we choose this time $t_1$ such that 
\begin{equation}\label{eq:I11_I12}
I_{11}^{1/2}+I_{12}^{1/2}\leqslant \frac{\sqrt{c_1}}{4}|\te_1-\te_2|.
\end{equation}

\paragraph{\emph{Upper bound for $I_{2}$.}}
We start by recalling that $X^{\te}=Y^{0,\te}$. Next we set $R=[\te_{1},\te_{2}]\times[0,\lambda]$ and for $r,\tau\in[0,1]$ we define
\begin{equation}\label{a2}
a(r,\tau)
=
Y_{t}^{0,\te_{2}}
+ r \lp Y_{t}^{0,\te_{1}}-Y_{t}^{0,\te_{2}}  \rp
+\tau \lp Y_{t}^{\lambda,\te_{2}}-Y_{t}^{0,\te_{2}}  \rp
+r \tau \Delta_{R} Y_{t},
\end{equation}
where the rectangular increment $\Delta_{R} Y_{t}$ is given by
\begin{equation}\label{a3}
\Delta_{R} Y_{t}
=
Y_{t}^{\lambda,\te_{1}} - Y_{t}^{0,\te_{1}} - Y_{t}^{\lambda,\te_{2}} + Y_{t}^{0,\te_{2}}.
\end{equation}
Notice that $I_{2}$ can be expressed as
\begin{equation}\label{eq:I2}
I_{2}
=
\int_{\R} \Big(\be\lc \Delta_{R} \psi_{\xi}(Y_{t}) \rc \Big)^{2} \, g_p(\xi) \, d\xi,
\end{equation}
where $\psi_{\xi}$ is the oscillating function $e^{i \xi x}$ and $\Delta_{R}\psi_{\xi}(Y_{t})$ still denotes a rectangular increment as in \eqref{a3}. Moreover, resorting to the path $a$ introduced in \eqref{a2} we get
\begin{equation*}
\Delta_{R} \psi_{\xi}(Y_{t})
=
\int_{[0,1]^{2}} 
\partial^2_{r\tau}\left[\psi_{\xi}(a(r,\tau))\right]
d r d \tau,
\end{equation*}
and computing the differential $\partial^2_{r\tau}\left[\psi_{\xi}(a(r,\tau))\right]$ explicitly we get
\begin{equation}\label{eq:Delta_R}
\Delta_{R} \psi_{\xi}(Y_{t})
=
\int_{[0,1]^{2}} 
\lp \psi_{\xi}^{\prime}(a(r,\tau)) \partial_{r\tau}^{2} a(r,\tau) 
+  \psi_{\xi}^{\prime\prime}(a(r,\tau))  \partial_{r} a(r,\tau) \partial_{\tau} a(r,\tau) \rp 
d r d \tau .
\end{equation}
Taking into account \eqref{eq:Delta_R} and \eqref{eq:I2}, plus the fact that $\psi$ and its derivatives $\psi^\prime,~\psi^{\prime\prime}$ are bounded by  $|\xi|^2$ and $|\xi|^2g_{p}$ is integrable if $p>3/2$, we get that 
\begin{align}\label{eq:bound_I2}
I_2&\leq C_p\left\{\be\lc\left(|Y_t^{0,\te_2}-Y_t^{0,\te_1}|+|\Delta_{R}(Y_{t})|\right)\left(|Y_t^{\lambda,\te_2}-Y_t^{0,\te_2}|+|\Delta_{R}(Y_{t})|\right)\rc+\be\lc|\Delta_{R}(Y_{t})|\rc\right\}^2\nonumber\\
&\leq C_{p,m,M,\varepsilon}\be\left[\|\partial_{\lambda}Y^\te\|_{\infty}^2+\|\partial_{\te}Y^\te\|_{\infty}^2+\|\partial_{\lambda\te}^2Y^\te\|_{\infty}^2\right]^2\lambda^2|\te_1-\te_2|^2.
\end{align}
In order to bound the right hand side of \eqref{eq:bound_I2}, we are now reduced to the estimation of $\partial_{\lambda}Y^{\lambda,\te}$, $\partial_{\te}Y^{\lambda,\te}$ and $\partial_{\lambda\te}^{2}Y^{\lambda,\te}$. Let us thus show how to establish a bound for $\partial_{\lambda}Y^{\lambda,\te}$. To this aim, differentiating formally equation  \eqref{eq:def-Y-theta-alpha}, the process $\partial_{\lambda}Y^{\lambda,\te}$ solves a system of the form
\begin{equation}\label{eq:dif-Y-wrt-alpha}
\frac{d[ \partial_{\lambda}Y^{\lambda,\te}]_{t}}{dt}
=
\lc  -\te  + \lambda \partial_{y}b_{\te}(Y^{\lambda,\te}_{t}) \rc \partial_{\lambda}Y^{\lambda,\te}_{t}+b_\te(Y^{\lambda,\te}_t) .
\end{equation}
Whenever $\te$ satisfies \eqref{hyp:compact-theta} and $\lambda\leq m(1-\varepsilon)$ with $\varepsilon>0$, equation \eqref{eq:dif-Y-wrt-alpha} above fulfills the hypothesis of Lemma \ref{lem:easy-ode}. Hence we get
\begin{equation*}
|\partial_{\lambda}Y^{\te,\lambda}_{t}| \le c_{m,M,\varepsilon},
\end{equation*}
uniformly in $t\ge 0$. We let the reader check that the same kind of inequality holds true for $\partial_{\te}Y^{\lambda,\te}$ and $\partial_{\lambda\te}^{2}Y^{\lambda,\te}$ as well, and thus we get
$$\|\partial_{\lambda}Y^\te\|_{\infty}+\|\partial_{\te}Y^\te\|_{\infty}+\|\partial_{\lambda\te}^2Y^\te\|_{\infty}\leq c_{m,M,\varepsilon}.$$
Plugging this inequality into \eqref{eq:bound_I2}, we end up with 
\begin{equation*}
|I_{2}| \le C_{p,m,M,\varepsilon} \lambda^{2} \, |\te_{1}-\te_{2}|^{2}.
\end{equation*}
We now choose $\lambda$ such that $\lambda\leq\frac{1}{4}\left(\frac{c_1}{C_{p,m,M,\varepsilon}}\right)^{1/2}$, where $c_1$ is defined by \eqref{eq:I3} and $\varepsilon=1/2$. This yields
\begin{equation}\label{eq:bound_I2_bis}
I_2^{1/2}\leq \frac{\sqrt{c_1}}{4}|\te_1-\te_2|.
\end{equation}
We now gather \eqref{eq:bound_I2_bis} and \eqref{eq:I11_I12} into \eqref{eq:fou1}, which proves our claim with $\lambda_0=\frac{1}{4}\left(\frac{c_1}{C_{p,m,M,\varepsilon}}\right)^{1/2}$.
\end{proof}
\section{Numerical Discussions and Illustrations}\label{sec:simu}

In this section, we provide several numerical examples in order to illustrate our main results. To this end, we first investigate several numerical questions which are related to our theoretical results.

\paragraph{\textbf{Simulated data and Euler scheme}:} In order to test our results, we have chosen to simulate our observations. Nevertheless, the fractional SDE \eqref{eq:sde} cannot be simulated exactly, except in some particular cases. Therefore we have opted for a discretization procedure thanks to a simple first order Euler scheme with very small step $\underline{\gamma}$ (namely $\underline{\gamma}=10^{-3}$) in order to get a sharp approximation of the true process. 

Let us recall that in the additive setting of equation \eqref{eq:sde} the simple Euler scheme converges strongly to the true SDE, while this is not true in general in the multiplicative case (see $e.g.$ \cite{nourdineuler}). The convergence of the scheme can be checked for instance through  Proposition \ref{prop:rapdisccont}$(i)$, applied with constant step $\underline{\gamma}$. Furthermore, taking the expectation in Proposition \ref{prop:rapdisccont}$(i)$  leads to a marginal control of the $L^2$-distance between the Euler scheme and the true SDE (with same fBm) of order $\underline{\gamma}^{H}$ (independently of the horizon). This confirms that our approximation of the observations is reasonable when $H$ is not too small (getting a control for the uniform distance is more involved).  

Let us also recall that the increments of the fBm can be simulated  through the  Wood-Chan method  (see \cite{wood}), which is based on the embedding of the covariance matrix of the fractional increments in a symmetric circulant matrix (whose eigenvalues can be computed using the Fast Fourier Transform).  Therefore up to the approximation of the true SDE detailed above, we now assume that we are given a sequence $(Y_{k\underline{\gamma}})_{k\ge0}$, where $(Y_t)_{t\ge0}$ is a solution to  \eqref{eq:sdetheta} with a given $\theta_0$. Then we select from this path a subsequence of observations $(Y_{t_{k}})_{k=1}^n$ where $t_{k}=k\ga$, which means in particular that we assume $\ga$ to be of the form $k_0\underline{\gamma}$ with $k_0\in\mathbb{N}^*$.

\paragraph{\textbf{Computation of the distance between empirical measures}:} The theoretical construction of an estimator like \eqref{eq:firstdefesti} involves in practice the computation of the distance $d$ between the empirical measures of the observed process and of the Euler scheme, for a distance  $d\in{\cal D}_p$ as defined in \eqref{set_distances}. We briefly describe how to compute this kind of distance.

Whenever $d$ is the $p$-Wasserstein distance, an explicit computation of the distance $d$ in \eqref{eq:firstdefesti} is possible if the observation $Y$ is  1-dimensional. To this aim, one can use the following representation (see \cite{vallander}): if $\mu$ and $\nu$ are two one-dimensional probabilities with $c.d.f$ $F$ and $G$ respectively, then for all $p>0$ we have
$$
{\cal W}_p^p(\mu,\nu)=\int_0^1 |F^{-}(t) -G^{-}(t)|^p dt,
$$
where $F^{-}$ and $G^{-}$ denote the (left or right) pseudo-inverse of $F$ and $G$. Moreover, when $\mu=\sum_{i=1}^{n_1} p_i \delta_{x_i}$ and $\nu=\sum_{j=1}^{n_2} q_j\delta_{y_j}$, the computation of the right hand side above can be made explicit through a reordering (and using the fact that $F^{-1}$ and $G^{-1}$ are stepwise constant). In particular, when $n_1=n_2=n$ and $p_i=q_j=1/n$, the Wasserstein distance between $\mu$ and $\nu$ simply reads
\begin{equation}\label{exp:wp}
 {\cal W}_p^p(\mu,\nu)=\frac{1}{n}\sum_{i=1}^n |x_{(i)}-y_{(i)}|^p.
 \end{equation}
 We will use this representation in our simulations.
 
In higher dimension, the computation of the Wasserstein distance   generally requires approximation/optimization methods which are out of the scope of this paper. In this context it seems to be numerically simpler to work with an approximation of the distance $d_{\ccf,p}$ (defined in  \eqref{eq:def_CFp}), which is also used for our analysis of the rate of convergence. Such an approximation can be obtained by a standard discretization of the integral which appears in the definition  \eqref{eq:def_CFp}.

\paragraph{\textbf{Minimization of the distance with respect to $\te$}:} 
Eventually the implementation of our estimation procedure relies on an optimization problem in order to compute the argmin in~\eqref{eq:firstdefesti}. 
More specifically, in case the estimator is built with a constant step Euler scheme, this consists in minimizing the function 
\begin{equation}\label{calfd}
{\cal F}_d:\te \mapsto d\left(\frac{1}{n}\sum_{k=0}^{n-1} \delta_{Y_{t_k}}, \frac{1}{N}\sum_{k=0}^{N-1} \delta_{Z_{k\gamma}^{\te,\gamma}}\right).
\end{equation}
In this paper, we first use the naive approach which consists in evaluating the function on a (finite) grid and then computing the minimum on this finite set. This minimization algorithm is clearly restricted to a low dimensional setting.
%
Secondly, we use a stochastic optimization algorithm to explore an example in dimension two (in space) for the distance $d_{CF,p}$ which is easier to manage. More precisions are given after the one-dimensional case.

\paragraph{\textbf{Numerical illustrations}:} Let us now turn to some numerical tests, for which we consider two one-dimensional examples and one two-dimensional example. We begin with the classical

\noindent $\rhd$  Ornstein-Uhlenbeck (OU) process: We consider the process $X^{\te}$ defined by \eqref{eq:def-fou},
where $\te$ is assumed to sit in a compact interval of $(0,+\infty)$ like in \eqref{hyp:compact-theta}. Let us recall that this case is a toy example since the Gaussian linear structure of the OU-process allows to develop specific estimation methods (on this topic, see \cite{KL} or more recently \cite{xiao} and \cite{brouste}). The assumptions $\HZERO$ and $\HUNS$ are clearly satisfied, whereas $\HTROIS$ follows from Lemma \ref{lem:dcf-fou}. Using the strategy described in the first part of this section, we get a discretely observed path of $Y$ with the following parameters:
$$\te_0=2, \; \underline{\gamma}=10^{-3}, \; \ga= 10^{-2},\; n=3.10^4,$$
and different values of $H$. In Figure \ref{figOU1}, we depict the function ${\cal F}_d$ defined in \eqref{calfd} with $d=d_{CF,2}$ for $H=0.3$ and $H=0.7$ respectively. As in the next examples, the Euler scheme is computed with $N=n$ and $\gamma=10^{-2}$ as specified above. We remark that the function attains its minimum very close to the true value of $\te$.
\begin{figure}[htbp]
   \begin{minipage}[c]{.46\linewidth}
      \includegraphics[width=7cm, height=5cm]{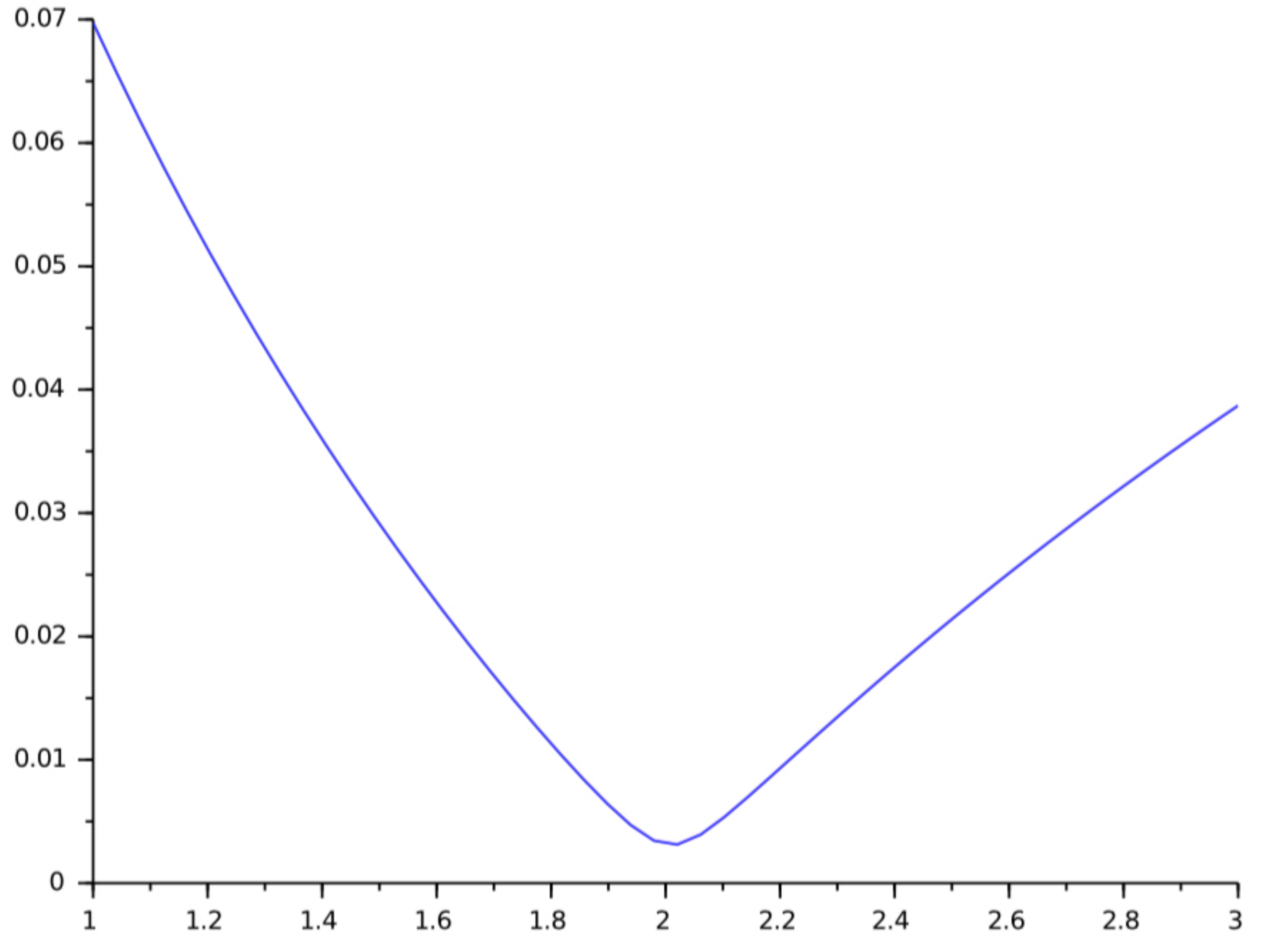}
  \end{minipage} \hfill
  \begin{minipage}[c]{.46\linewidth}
      \includegraphics[width=7cm,height=5cm]{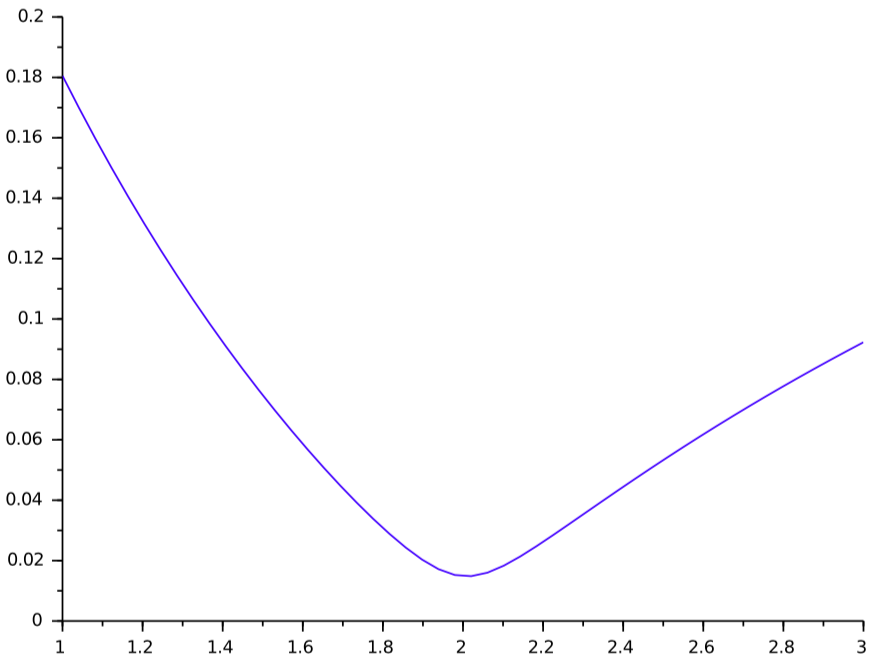}
  \end{minipage}
	\caption{$\te \mapsto {\cal F}_{d_{CF,2}}(\te)$ for  $H=0.3$ (left) and $H=0.7$ (right). }
	\label{figOU1} 
	\end{figure}
We also observe that the function ${\cal F}_d$ is more flat when $H$ is small, which is consistent with the fact that $H\mapsto \si_{\te}^{2}$ in \eqref{eq:variance_mes-inv_fou} is an increasing function. 

Figure \ref{figOU2} below is devoted to a comparison between the different $p$-Wasserstein distances as $p$ varies. Namely we fix $H=0.3$ and we compute  the function  ${\cal F}_d$ defined by \eqref{calfd} with $d={\cal W}_p$ for different values of $p$. Notice that in the 1-dimensional case we are considering we can resort to formula \eqref{exp:wp}, since we have chosen $N=n$. The true parameter is still $\te=2$. Our distances all perform correctly, although $p=4$ seems to yields a slightly sharper contrast.
\begin{figure}[htbp]
      \includegraphics[width=10cm, height=5cm]{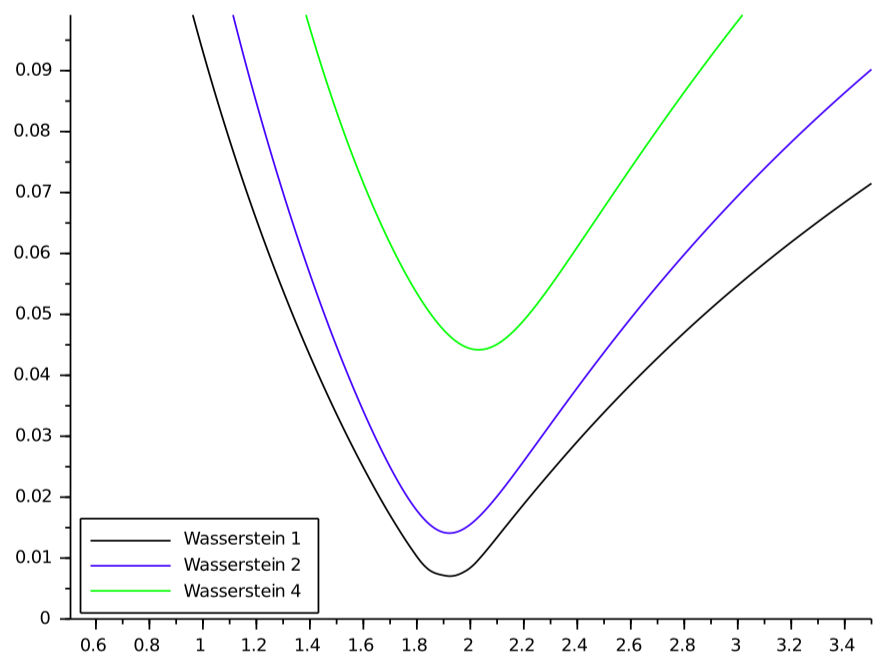}
	\caption{$\te \mapsto {\cal F}_{{\cal W}_p}(\te)$, $H=0.3$.  }
\label{figOU2}
	\end{figure}
	
\noindent $\rhd$ We now consider a second example with a non linear dependence in $\te$, namely an equation of the form:
 $$ d Y^{\te}_{t}
= - Y^{\te}_{t}(1+\cos(\te Y^\te)) dt +dB_t. $$ 
In this case, we only compute the Wasserstein distance for different values of $p$ with the same choices of parameters. Once again, the minimum of the function ${\cal F}_{{\cal W}_p}$
is attained close to $\te_0=2$. One also observes that the local behavior in the neighborhood of $\te_0$ is similar to the linear case.
\begin{figure}[htbp]
   \begin{minipage}[c]{.46\linewidth}
      \includegraphics[width=7cm, height=5cm]{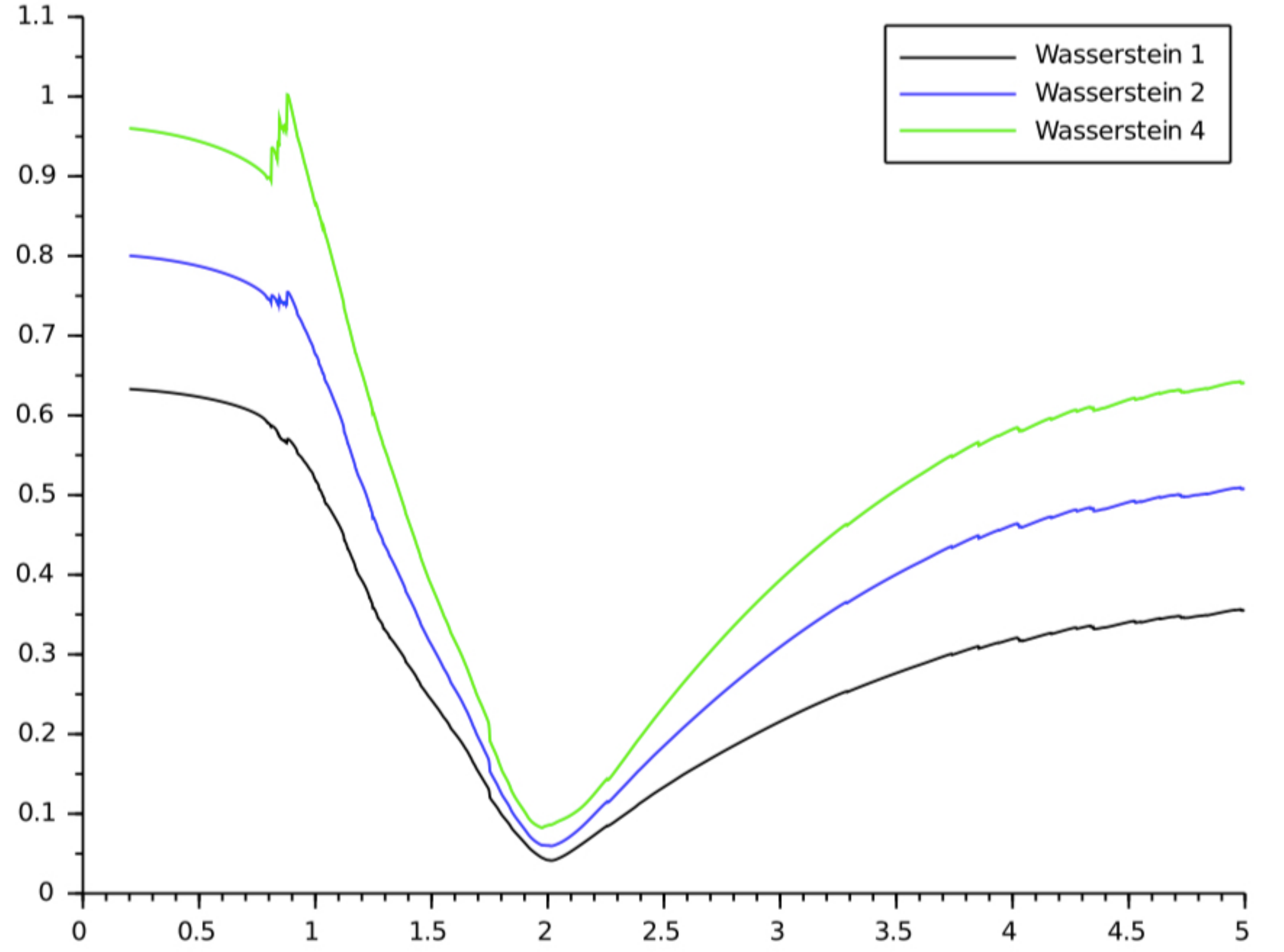}
  \end{minipage} \hfill
  \begin{minipage}[c]{.46\linewidth}
      \includegraphics[width=7cm,height=5cm]{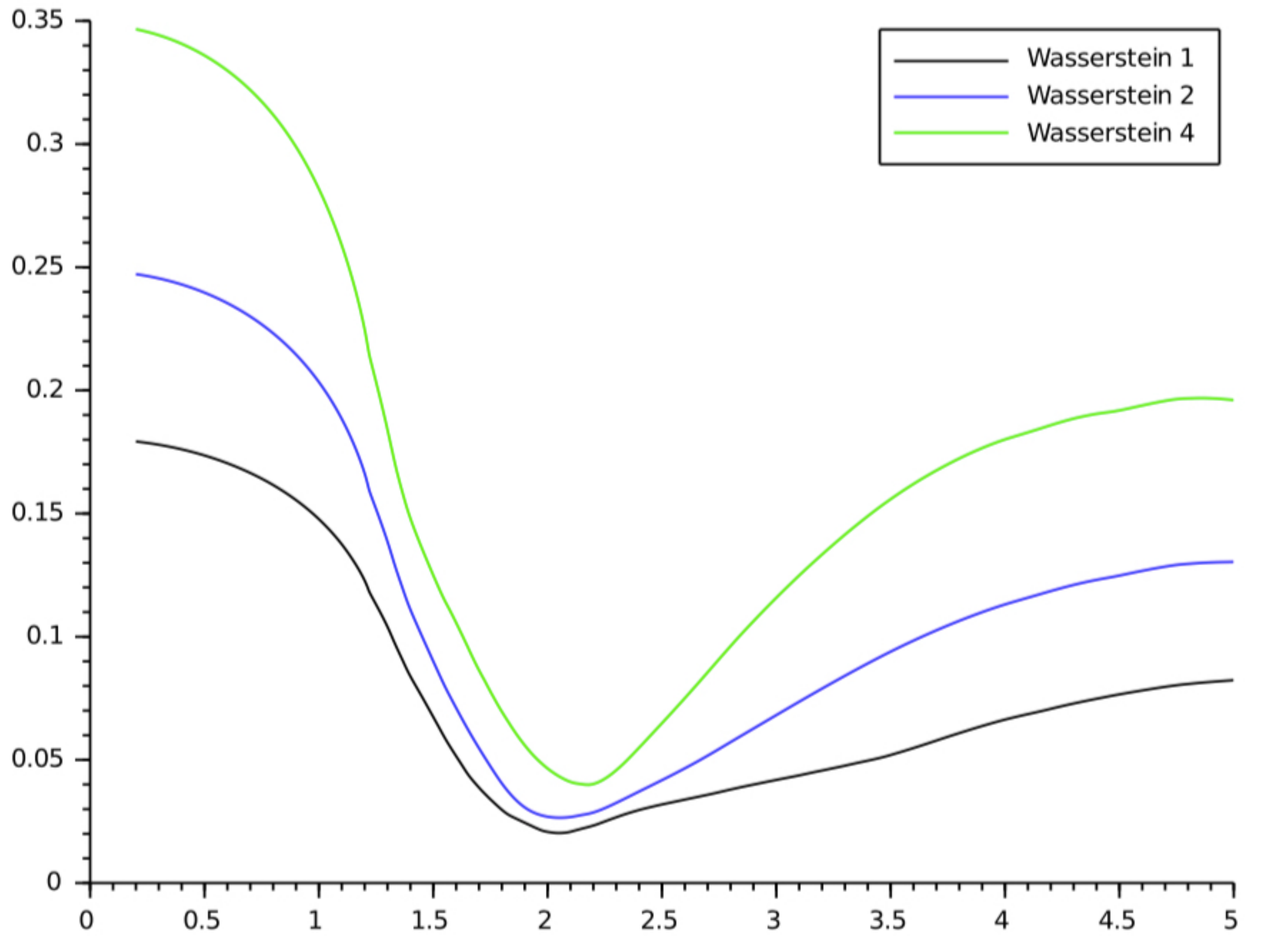}
  \end{minipage}
	\caption{$\te \mapsto {\cal F}_{{\cal W}_p}(\te)$ for  $H=0.3$ (left) and $H=0.7$ (right).  }
\label{figOU3}
	\end{figure}

\noindent $\rhd$ A two-dimensional example: we finally propose to focus on  a family of two-dimensional fractional processes $(Z^{\te}_{t})_{t\ge0}$ indexed by a real parameter $\te\in\Theta=[m,M]\subset(0,+\infty)$, where $m,M$ are two finite constants. Specifically,  $Z^{\te}$ is the solution to the stochastic system
\begin{equation}\label{eq:example-2d}
dZ^{\te}_t=b_\te(Z^\te_t) dt + dB_t^{H},
\end{equation}
where $(B_t^H)_{t\ge0}$ is a two-dimensional standard fBm with Hurst parameter $H$. In equation \eqref{eq:example-2d}, we also have that for all $\te\in\Theta$ the function 
$b_\te:\ER^2\rightarrow\ER^2$ is defined by:
\begin{equation}\label{eq:def-drift-2d-example}
\text{for all } z=(z_1,z_2)\in\ER^2,\quad 
b_\te(z)= -\te z\psi\lp ( 1+|z|^2)^{1/2}\rp+ \varepsilon z^\perp,
\end{equation}
where $z^\perp=(-z_2,z_1)$ and $\psi(x)=\frac{e^x}{1+e^x}$ (namely $\psi$ is the \textit{sigmoid} function). Let us first remark that the coefficient $b_{\te}$ above satisfies the main standing assumptions of the current paper and generalizes the framework of \cite{NT}. Indeed, the following holds true: (a) Due to the orthogonal component $z^\perp$ in equation \eqref{eq:def-drift-2d-example}, it is readily checked that $\partial_{z_2}b_{\te}^{1}\ne\partial_{z_1}b_{\te}^{2}$ (where we recall the notation of Section~\ref{subsection:notation} for the partial derivatives of $b$). Therefore our family of  drift functions $b_{\te}$ cannot be of gradient type.
(b) The functions $b_{\te}$ are obviously smooth. In addition, we have $\langle x_{1}^{\perp}-x_{2}^{\perp}, x_{1}-x_{2}\rangle=0$. 
Thus 
in order to see that $b_{\te}$ satisfies our Hypothesis~$\HUNS$, it is enough to check that the eigenvalues of the Jacobian matrix of $z\mapsto z\vp(|z|^2)$ are uniformly lower bounded by a positive constant, where we have set
\begin{equation*}
\vp(x) = \psi\lp (1+x)^{1/2}\rp, 
\quad\text{for}\quad x \ge 0,
\end{equation*}
and where $\psi$ is the sigmoid function featuring in our definition~\eqref{eq:def-drift-2d-example}.
We leave to the patient reader the elementary task of computing those eigenvalues. Let us just mention that a lower bound is given by $\vp(|z|^2)$, which is uniformly lower bounded by $\frac12$ for all $z\in\R^{2}$.  This implies that Assumption $\HUNS$ is fulfilled for the coefficient $b_{\te}$.

The multi-dimensional setting leads to new numerical difficulties. The first of those problems is due to the fact that the exact computation of the Wasserstein distance in $\R^{2}$ is still possible but costly (about $n^3$ computations for an empirical measure based on $n$ points by the \textit{Hungarian algorithm}), while approximations  lead to numerical issues which are out of the scope of this paper (see $e.g.$ \cite{schmitz}). We have thus decided to work with the distance $d_{\ccf,p}$ introduced in Section~\ref{sec:rateofconvergence}, which is simpler to compute. The second numerical problem we are facing comes from the underlying spatial dimension. 
Indeed, when $d=2$ one can still optimize over the parameter $\te$ by computing the contrast given by \eqref{eq:firstdefesti} over a discrete grid of parameters $\te$. However, it is clear that in larger dimensions such contrast estimators require to implement associated convex optimization algorithms. 

\noindent With the above preliminary considerations in mind, we now briefly introduce  a stochastic optimization algorithm for the approximation of 
$\hat{\te}_{N,n,\ga}$ related to $d_{\ccf,p}$.  We will then implement this method on the coefficient $b_{\te}$ given by \eqref{eq:def-drift-2d-example}. Here, the parameter $\te$ is still one-dimensional but the method which is proposed can be adapted to the higher dimension. In order to describe our algorithm, let us first 
recast our expression \eqref{eq:def_CFp} for the $d_{\ccf,p}$ distance. That is, observe that if $f_\xi(x):=e^{i\langle \xi,x\rangle}$ and if
$\Xi$ is a random variable with density $g_p$ then for two probability measures $\mu_1$ and $\mu_2$ we have
$$
d_{\ccf,p}(\mu_1,\mu_2)
=\ES\lc |\mu_1(f_\Xi)-\mu_2(f_\Xi)|^2 \rc.
$$
Now recall from relation~\eqref{eq:firstdefesti} that we are dealing with the following empirical measures:
$$
\mu=\frac{1}{n}\sum_{k=0}^{n-1}\delta_{Y_{t_k}}\quad \textnormal{and}\quad \mu_\te=\frac{1}{N}\sum_{k=0}^{N-1}\delta_{Z_{k\ga}^{\te,\ga}}.
$$
Then our aim in~\eqref{eq:firstdefesti}  is to minimize a functional $F$ defined on $\tte=[m,M]$ by
$$ 
F(\te)=d_{\ccf,p}(\mu,\mu_\te)
=\ES\lc |\mu(f_\Xi)-\mu_\te(f_\Xi)|^2 \rc.
$$
Our gradient descent algorithm will thus be applied to the functional $F$, and we now provide some details about its computation. Indeed, the gradient $\nabla F$ is formally obtained, similarly to what we did in the proof of Proposition~\ref{prop:dcf-perturbed-fou}, as
\begin{equation}\label{f1}
\nabla F(\te)=\ES[\Lambda(\te,\Xi)],
\quad\text{where}\quad
\Lambda(\te,\xi)=\partial_\te \left(|\mu(f_\xi)-\mu_\te(f_\xi)|^2\right).
\end{equation}
Furthermore, it is easily seen that $\Lambda(\te,\xi)$ can be expressed as
\begin{align}\label{f2}
\Lambda(\te,\xi)=2(\mu_\te-\mu)(\cos(\langle \xi,\cdot\rangle))\rho_\te(-\sin(\langle \xi,\cdot\rangle))+2(\mu_\te-\mu)(\sin(\langle \xi,\cdot\rangle))\rho_\te(\cos(\langle \xi,\cdot\rangle)) ,
\end{align}
where we resort to the following notation for a given function $g:\ER\rightarrow\ER$, 
\begin{equation}\label{f3}
\rho_\te(g(\langle \xi,.\rangle))=\frac{1}{N}\sum_{k=0}^{N-1} g(\langle \xi, Z_{k\ga}^{\te,\ga}\rangle)\, \langle \xi, \partial_\te Z_{k\ga}^{\te,\ga}\rangle. 
\end{equation}
Plugging \eqref{f3} into \eqref{f2} and then \eqref{f1}, one observes that the computation of  $\nabla F(\te)$ mostly relies on our ability to evaluate $(\partial_\te Z_{k\ga}^{\te,\ga})_{k\ge0}$. Now it should be highlighted that this quantity can be computed \emph{on the fly}, since $\partial_\te Z_{0}^{\te,\ga}=0$ and
for all $k\ge 0$ one can evaluate $\partial_\te Z_{(k+1)\ga}^{\te,\ga}$ recursively as
$$
\partial_\te Z_{(k+1)\ga}^{\te,\ga}=\partial_\te Z_{k\ga}^{\te,\ga}+\gamma \left(\partial_\te b_{\te}(Z_{k\ga}^{\te,\ga})+\nabla_z b_\te(Z_{k\ga}^{\te,\ga}) \partial_\te Z_{k\ga}^{\te,\ga}\right).
$$

The convergence of the gradient descent algorithm to a minimum of $F$ lies outside the scope of the current paper (see $e.g.$ \cite{duflo} for background on this topic). Let us just mention that 
the algorithm is recursively defined by
\begin{equation}\label{eq:def-step-gradient-descent}
\theta_{\ell+1}=\theta_\ell-\eta_\ell \, \Lambda(\theta_\ell,\Xi_{\ell+1}),
\end{equation}
where $(\eta_\ell)_{\ell}$ is a sequence of (positive) steps and $(\Xi_\ell)_{\ell\ge1}$ is a sequence \emph{i.i.d} random variables with common distribution $g_p$. 
Notice that in dimension $2$ the random variables $\Xi_{\ell}$ can be simulated through a change of variable. More precisely, if $R$ and $\Theta$ are two independent variables such that $R$ has density $q_p(r)=\frac{2(p-1) r}{(1+r^2)^p}$ and $\Theta$ has uniform distribution on $[0,2\pi]$, then $(R\cos (\Theta), R\sin(\Theta))$ has density $g_p$. Furthermore, thanks to the inversion method we get that  $R$ can be simulated as $R=(U^{1/(1-p)}-1)^{1/2}$ where $U$ has uniform distribution on $[0,1]$.

The illustration represented in Figure \ref{ex:dim2} is obtained with $\varepsilon=1/10$ in equation \eqref{eq:def-drift-2d-example} and $\eta_\ell= \gamma_0 (1+\ell)^{-1/2}$ in relation \eqref{eq:def-step-gradient-descent}.  We have also considered two values of $\gamma_0$ and two different starting points $y_{0}$ for the process $Z^{\te}$ in \eqref{eq:example-2d}. 
Furthermore, in these simulations we used a \textit{mini-batch} approach. This means that in equation \eqref{eq:def-step-gradient-descent} we replace the random simulated term $\Lambda(\theta_\ell,\Xi_{\ell+1})$ by an average over $m=30$ simulations. The mini-batch procedure is known  to reduce the randomness of the algorithm. The results displayed in Figure~\ref{ex:dim2} show that our approach yields a good convergence of the estimate to the true parameter $\te$. However, our simulations also reveal that the gradient gets quite flat near the true parameter $\te_0=1$. Therefore the algorithm moves very slowly after a large number of iterations. This is the reason why in the illustration, the algorithm is stopped after $n=100$ simulations. 
Also notice that we have provided several examples of initial values $y_{0}$ and step sizes $\gamma_0$, which exemplifies the sensibility of the current version of the algorithm with respect to the parameters. Those first attempts to combine our statistical procedure with gradient descent algorithms would certainly need to be enriched with deeper considerations, although this task is deferred to a subsequent publication for sake of conciseness.

\begin{figure}[htbp]
\includegraphics[height=6cm]{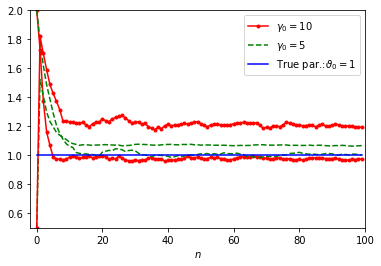}
\caption{Example in dimension $2$.  }
\label{ex:dim2}
\end{figure}

\newpage

\appendix
\section{Tightness and convergence of the occupation measures}
In this Appendix, we first show some uniform estimates for $Y^\te$ and $Z^{\te,\ga}$ and their respective occupation measures, as well as some convergence result for the occupation measure of $Y^\te$. Those results are all used in the proof of Proposition \ref{prop:ergodicSDE} and \ref{prop:ergodicEuler}.

\subsection{Moment estimates}
We start by bounding the moments of $Y^\te$ and its occupation measure.
\begin{proposition}\label{prop:controlmoment} Let $Y^\te$ be the unique solution of \eqref{eq:sdetheta}. Assume $\HZERO$ and $\HUNW$. Then the following inequalities hold true for $p\geq1$.
\begin{itemize}
\item[$(i)$]$\underset{t\geqslant0}{\sup}~\be[\underset{\te\in\tte}{\sup}~|Y^\te_t|^p]<+\infty.$
\smallskip
\item[$(ii)$]$\underset{\te\in\tte}{\sup}~\underset{n\geqslant1}{\sup}~\frac{1}{n}\sum_{k=0}^{n-1}|Y^\te_{k\kappa}|^{p}<+\infty$ a.s.
\item[$(iii)$]$\underset{\te\in\tte}{\sup}~\underset{t>0}{\sup}~\frac{1}{t}\int_0^t|Y^\te_{t}|^{p}dt<+\infty$ a.s.
\end{itemize}
\end{proposition}

\begin{proof} We treat our three items separately. 

$(i)$ The proof is done in \cite{Ha}. It is based on a comparison with the moments of the fractional Ornstein-Uhlenbeck process, similarly to what is done in step $(ii)$ below. Since the constants in $\HUNW$ are independent from $\te$, we get the uniformity with respect to $\te$.

$(ii)$ Let $p>0$. By Cauchy-Schwarz inequality it is enough to prove the result with $2p$ instead of $p$. Let us denote by $(X_t)_{t\geqslant0}$ the fractional Ornstein Uhlenbeck process defined by \eqref{eq:def-fou} with $\te=1$ and starting from $y_0$, i.e. 
\begin{equation}\label{eq:FOU1}
\left\{\begin{array}{l}
dX_t=-X_t~dt +  \si ~dB_t\\
X_0=y_0.
\end{array}\right.
\end{equation}
In this proof, we set $\underline{t}:=\inf\{k\kappa~|~k\kappa\leqslant t<(k+1)\kappa\}$ and $T:=n\kappa$. With this notations, similarly to what is done in the proof of Proposition \ref{prop:unifdecreas}, we have $$\frac{1}{n}\sum_{k=0}^{n-1}|Y^\te_{k\kappa}|^{2p}=\frac{1}{T}\int_{0}^T|Y^\te_{\underline{t}}|^{2p}dt.$$ 
Let us also denote by $\tilde{X}$ the stationary Ornstein-Uhlenbeck process related to \eqref{eq:FOU1}. Then the proof of item $(ii)$ is based on the following inequality:
\begin{equation}\label{eq:proof_moments_Y_emp}
\frac{1}{T}\int_{0}^T|Y^\te_{\underline{t}}|^{2p}dt\leqslant 3^{2p-1}\left(\frac{1}{T}\int_{0}^T|Y^\te_{\underline{t}}-X_{\underline{t}}|^{2p}dt+\frac{1}{T}\int_{0}^T|X_{\underline{t}}-\tilde{X}_{\underline{t}}|^{2p}dt+\frac{1}{T}\int_{0}^T|\tilde{X}_{\underline{t}}|^{2p}dt\right)
\end{equation}
Let us start by bounding the two last term in the right hand side of \eqref{eq:proof_moments_Y_emp}. First, since $(\tilde{X}_{k\kappa})_{k\geqslant0}$ is a stationary Gaussian sequence and since $\be[\tilde{X}_{\kappa}\tilde{X}_{(n+1)\kappa}]\underset{n\to+\infty}{\longrightarrow}0$ (according to \cite[Theorem 2.3]{CKM}), a criterion for Gaussian processes (see \cite{maruyama}) gives that $(\tilde{X}_{k\kappa})_{k\geqslant0}$ is ergodic. Therefore, we get
\begin{equation}\label{eq:ergodic_Xbar}
\lim\limits_{T\to+\infty}\frac{1}{T}\int_{0}^T|\tilde{X}_{\underline{t}}|^{2p}dt=\lim\limits_{n\to+\infty}\frac{1}{n}\sum_{k=0}^{n-1}|\tilde{X}_{k\kappa}|^{2p}=\be[|\tilde{X}_{\kappa}|^{2p}]<+\infty,
\end{equation}
where the limit holds in the almost sure sense.\\
In order to treat the second term in the right hand side of \eqref{eq:proof_moments_Y_emp}, we resort to equation \eqref{eq:FOU1}. From there it is readily checked that $\frac{d}{dt}|X_t-\tilde{X}_t|^2=-2|X_t-\tilde{X}_t|^2$, and thus $|X_t-\tilde{X}_t|^2=e^{-2t}|y_0-\tilde{X}_0|^2$.
In particular, we have $\lim\limits_{k\to+\infty}|X_{k\kappa}-\tilde{X}_{k\kappa}|^{2p}=\lim\limits_{k\to+\infty}e^{-2pk\kappa}|y_0-\tilde{X}_0|^{2p}=0$ and the corresponding Cesaro summation tends also to $0$, i.e.
\begin{equation}\label{eq:conv_X-Xbar}
\lim\limits_{T\to+\infty}\frac{1}{T}\int_{0}^T|X_{\underline{t}}-\tilde{X}_{\underline{t}}|^{2p}dt=\lim\limits_{n\to+\infty}\frac{1}{n}\sum_{k=0}^{n-1}|X_{k\kappa}-\tilde{X}_{k\kappa}|^{2p}=0.
\end{equation}
Now it remains to treat the first term in the right hand side of \eqref{eq:proof_moments_Y_emp}. To this end, we will invoke our hypothesis $\HUNW$ and a classical argument based on Gronwall's lemma. Namely, note that for all $\varepsilon>0$, 
\begin{align*}
\frac{d}{dt}|Y^\te_t-X_t|^2&=2\langle Y^\te_t-X_t, b_\te(Y^\te_t)+X_t\rangle\\
&=2\langle Y^\te_t-X_t, b_\te(Y^\te_t)-b_\te(X_t)\rangle+2\langle Y^\te_t-X_t, b_\te(X_t)+X_t\rangle\\
&\leqslant 2\beta-2\alpha|Y^\te_t-X_t|^2+\frac{1}{\varepsilon}|Y^\te_t-X_t|^2+\varepsilon\tilde{C}(1+|X_t|^{2r}),
\end{align*}
where the last inequality stems from $\HUNW$, Young's inequality and the linear growth of $b_\te$ entailed by \eqref{eq:inward-weak}.
We then set $\varepsilon=1/\alpha$, which gives
$$\frac{d}{dt}|Y^\te_t-X_t|^2\leqslant 2\beta-\alpha|Y^\te_t-X_t|^2+\frac{\tilde{C}}{\alpha}(1+|X_t|^{2r}).$$
Hence from Gronwall's lemma, we deduce that
\begin{equation}\label{eq:OU-SDE_comparison}
|Y^\te_t-X_t|^2\leqslant\int_{0}^te^{-\alpha(t-s)}\lp 2\beta+\frac{\tilde{C}}{\alpha}(1+|X_s|^{2r})\rp ds.
\end{equation}
Starting from \eqref{eq:OU-SDE_comparison}, one can easily get a bound on $|Y^\te_t-X_t|^{2p}$. Namely, due to the fact that $\int_{0}^t e^{-\alpha(t-s)}ds=\alpha^{-1}(1-e^{-\alpha t})$, a direct application of Jensen's inequality yields
\begin{align*}
|Y^\te_t-X_t|^{2p}&\leqslant\lp\frac{1-e^{-\alpha t}}{\alpha}\rp^{p-1}\int_{0}^te^{-\alpha(t-s)}\lp 2\beta+\frac{\tilde{C}}{\alpha}(1+|X_s|^{2r})\rp^p ds\\
&\leqslant C_p\int_{0}^te^{-\alpha(t-s)}(1+|X_s|^{2pr}) ds.
\end{align*}
Then, by using Fubini theorem, it comes
\begin{align}\label{eq:Y-X}
\frac{1}{T}\int_0^T|Y^\te_{\underline{t}}-X_{\underline{t}}|^{2p}dt&\leqslant \frac{C_p}{T}\int_{0}^T\int_{0}^{\underline{t}}e^{-\alpha(\underline{t}-s)}(1+|X_s|^{2pr}) ds~dt\nonumber\\
&\leqslant \frac{C_p}{T}\int_{0}^T(1+|X_s|^{2pr})\int_{s}^{T}{\bf{1}}_{[0,\underline{t}]}(s)e^{-\alpha(\underline{t}-s)} dt~ds.
\end{align}
In addition, for $t\in[0,T]$ we have ${\bf{1}}_{[0,\underline{t}]}(s)e^{-\alpha(\underline{t}-s)}\leqslant e^{\alpha\kappa}e^{-\alpha(t-s)}$. Hence integrating the right hand side of \eqref{eq:Y-X} we get
\begin{align*}
\frac{1}{T}\int_0^T|Y^\te_{\underline{t}}-X_{\underline{t}}|^{2p}dt\leqslant \frac{C_pe^{\alpha\kappa}}{\alpha}\frac{1}{T}\int_{0}^T(1+|X_s|^{2pr})\left(1-e^{-\alpha(T-s)} \right) ds\leqslant\frac{C_pe^{\alpha\kappa}}{\alpha}\lp1+\frac{1}{T}\int_{0}^T|X_s|^{2pr}ds\rp.
\end{align*}
Finally, by the same type of arguments used in the first part of the proof (see \eqref{eq:conv_X-Xbar}), we can show that $\lim\limits_{T\to+\infty}\frac{1}{T}\int_{0}^T|X_s|^{2pr}ds=\be[|\tilde{X}_0|^{2pr}]<+\infty$. Thus we have obtained
\begin{equation}\label{eq:final_Y-X}
\sup_{T>0}\sup_{\te\in\tte}\frac{1}{T}\int_0^T|Y^\te_{\underline{t}}-X_{\underline{t}}|^{2p}dt<+\infty.
\end{equation}
Plugging \eqref{eq:ergodic_Xbar}, \eqref{eq:conv_X-Xbar} and \eqref{eq:final_Y-X} into \eqref{eq:proof_moments_Y_emp}, our claim $(ii)$ is now proved.

$(iii)$ We let the patient reader check the details for item $(iii)$. It follows the same lines as item $(ii)$, except for the fact that the discretization procedure is avoided.
\end{proof}

Our next result is an analog of Proposition \ref{prop:controlmoment} for the Euler scheme $Z^{\te,\ga}$.

\begin{proposition}\label{prop:control_moment_Z} Let $Z^{\te,\ga}$ be the Euler approximation scheme defined by \eqref{eq:Euler_scheme}. Assume $\HZERO$ and $\HUNW$. Then, there exists $\ga_0>0$ such that for all $p>0$,
\begin{itemize}
\item[$(i)$] $\underset{\te\in\tte}{\sup}~\underset{\gamma\in(0,\gamma_0]}{\sup}~\limsup\limits_{N\to+\infty}~\be[|Z^{\te,\ga}_{N\gamma}|^{p}]<+\infty.$
\smallskip
\item[$(ii)$] $\underset{\gamma\in(0,\gamma_0]}{\sup}~\underset{N\geqslant0}{\sup}~\ga^{p(1-H)}\be[\underset{\te\in\tte}{\sup}|Z^{\te,\ga}_{N\gamma}|^{p}]<+\infty.$
\smallskip
\item[$(iii)$] Fix $\ga\in(0,1]$. Then $\underset{\te\in\tte}{\sup}~\underset{N\geqslant1}{\sup}~\frac{1}{N}\sum_{k=0}^{N-1}|Z^{\te,\ga}_{k\gamma}|^{p}<+\infty$ $a.s.$
\end{itemize}
\end{proposition}

The strategy for the proof of Proposition \ref{prop:control_moment_Z} is based on a comparison between $Z^{\te,\ga}$ and the Euler scheme $\Sigma$ related to the Ornstein-Uhlenbeck process $X$ given by \eqref{eq:FOU1}. Namely define $\Sigma$ recursively by $\Sigma_0=y_0$ and:
\begin{equation}\label{eq:rec_Euler_scheme_O-U}
\Sigma_{(k+1)\ga}=(1-\ga)\Sigma_{k\ga}+\sigma((B_{(k+1)\ga}-B_{k\ga})),\quad\forall k\geqslant0.
\end{equation}
We first prove some bounds on $\Sigma$ itself, which are labeled in the following lemma

\begin{lemma}\label{prop:control_moment_O-U} Let $\Sigma$ be the Euler scheme defined by \eqref{eq:rec_Euler_scheme_O-U}. There exists $\ga_0>0$ such that for all $p>0$ we have
\begin{itemize}
\item[$(i)$] $\underset{\gamma\in(0,\gamma_0]}{\sup}~\limsup\limits_{N\to+\infty}~\be[|\Sigma_{N\gamma}|^{p}]<+\infty$.
\smallskip
\item[$(ii)$] $\underset{\gamma\in(0,\gamma_0]}{\sup}~\underset{N\geqslant0}{\sup}~\ga^{p(1-H)}\be[|\Sigma_{N\gamma}|^{p}]<+\infty$.
\smallskip
\item[$(iii)$] Fix $\ga\in(0,\ga_0]$. Then $\underset{N\geqslant1}{\sup}~\frac{1}{N}\sum_{k=0}^{N-1}|\Sigma_{k\gamma}|^{p}<+\infty$ $a.s.$
\end{itemize}
\end{lemma}

\begin{proof} By Cauchy-Schwarz inequality, it is enough to show the three results for $2p$ instead of~$p$. We treat again the three items separately.

$(i)$ Let $\ga_0\in(0,1)$. By Lemma 2 in \cite{CP}, we know that
\begin{equation}\label{eq:moment_2_O-U}
\sup\limits_{\ga\in(0,\ga_0]}~\limsup\limits_{N\to+\infty}~\be[|\Sigma_{N\ga}|^2]<+\infty.
\end{equation}
Since $\Sigma$ is a Gaussian process, the extension from a second order moment to a moment of order $2p$ is trivial. We thus omit details for sake of conciseness.

$(ii)$ Let $\ga_0\in(0,1)$ and $\ga\in(0,\ga_0)$. First note that by induction, we have for all $k\geq1$: 
\begin{equation}\label{eq:Euler_scheme_O-U}
\Sigma_{k\ga}=(1-\ga)^ky_0+\sigma\sum_{j=0}^{k-1}(1-\ga)^j\Delta_{k-j},
\end{equation}
where $\Delta_{k-j}:=(B_{(k-j)\ga}-B_{(k-1-j)\ga})$.
In equation \eqref{eq:Euler_scheme_O-U}, we apply the triangular inequality for the norm in $L^2(\Omega)$ and we invoke the fact that $\be[(\Delta^i_{k})^{2}]^{1/2}=\ga^{H}$ for all $i\in\{1,\dots,d\}$. This yields
\begin{align*}
\be[|\Sigma_{k\ga}|^{2}]^{1/2}&\leqslant(1-\ga)^k|y_0|+|\sigma|\sum_{j=0}^{k-1}(1-\ga)^j\be[|\Delta_{k-j}|^{2}]^{1/2}\\
&\leqslant (1-\ga)^k|y_0|+|\sigma|\sum_{j=0}^{k-1}(1-\ga)^j \sum_{i=1}^d\be[(\Delta^i_{k-j})^{2}]^{1/2}\\
&\leqslant (1-\ga)^k|y_0|+d|\sigma|\ga^{H}\sum_{j=0}^{k-1}(1-\ga)^j \\
&\leqslant |y_0|+d|\sigma|c\ga^{H-1}.
\end{align*}
Therefore, we easily get that
\begin{equation*}
\gamma^{2-2H}\be[|\Sigma_{k\ga}|^{2}]\leqslant 2\ga^{2-2H}\lp|y_0|^{2}+d^2|\sigma|^2\ga^{2H-2}\rp\leqslant2\ga_0^{2-2H}|y_0|^{2}+2d^2|\sigma|^2
\end{equation*}
which gives the following bound
\begin{equation*}
\underset{\gamma\in(0,\gamma_0]}{\sup}~\underset{N\geqslant0}{\sup}~\ga^{2-2H}\be[|\Sigma_{N\gamma}|^{2}]<+\infty.
\end{equation*}
Exactly as in item $(i)$, we now resort to the Gaussian nature of $\Sigma$ in order to conclude that
\begin{equation}\label{eq:result_for_O-U}
\underset{\gamma\in(0,\gamma_0]}{\sup}~\underset{N\geqslant0}{\sup}~\ga^{2p(1-H)}\be[|\Sigma_{N\gamma}|^{2p}]<+\infty,\quad\text{ for all }p\geq1.
\end{equation}

$(iii)$ We start from identity \eqref{eq:Euler_scheme_O-U}. Applying Jensen's inequality we deduce that 
\begin{align}\label{eq:Jensen_OU}
|\Sigma_{k\ga}|^{2p}&\leqslant 2^{2p-1}(1-\ga)^{2kp}|y_0|^{2p}+\lp\frac{2}{\ga}\rp^{2p-1}|\sigma|^{2p}\sum_{j=0}^{k-1}(1-\ga)^j|\Delta_{k-j}|^{2p}\nonumber\\
&\leqslant 2^{2p-1}|y_0|^{2p}+\lp\frac{2}{\ga}\rp^{2p-1}|\sigma|^{2p}\sum_{j=0}^{k-1}(1-\ga)^j|\Delta_{k-j}|^{2p}.
\end{align}
We have thus obtained that
\begin{align*}
\frac{1}{N}\sum_{k=0}^{N-1}|\Sigma_{k\ga}|^{2p}\leqslant (2^{2p-1}+1)|y_0|^{2p}+\lp\frac{2}{\ga}\rp^{2p-1}|\sigma|^{2p}\frac{1}{N}\sum_{k=1}^{N-1}\sum_{j=0}^{k-1}(1-\ga)^j|\Delta_{k-j}|^{2p}.
\end{align*}
Therefore our claim $(iii)$ is reduced to show that
\begin{equation}\label{eq:moy_emp_Delta}
\sup_{N\geqslant1}~\frac{1}{N}\sum_{k=1}^{N-1}\sum_{j=0}^{k-1}(1-\ga)^{j}|\Delta_{k-j}|^{2p}<+\infty.
\end{equation}
In order to prove \eqref{eq:moy_emp_Delta} we make a change of variable $\ell=k-1-j$ and apply Fubini's theorem. This gives
\begin{align*}
\frac{1}{N}\sum_{k=1}^{N-1}\sum_{j=0}^{k-1}(1-\ga)^{j}|\Delta_{k-j}|^{2p}&=\frac{1}{N}\sum_{\ell=0}^{N-2}|\Delta_{\ell+1}|^{2p}\sum_{k=\ell+1}^{N-1}(1-\ga)^{k-1-\ell}\\
&\leqslant\frac{1}{\ga}\frac{1}{N}\sum_{\ell=0}^{N-2}|\Delta_{\ell+1}|^{2p}\leqslant\frac{d^{p'-1}}{\ga}\sum_{i=1}^d\frac{1}{N}\sum_{\ell=0}^{N-2}(\Delta^i_{\ell+1})^{2p}.
\end{align*}
Since the sequence $(\Delta^i_n)_{n\geqslant1}$ is ergodic for every $i\in\{1,\dots,d\}$, we have 
$$\frac{1}{N}\sum_{\ell=0}^{N-2}(\Delta^i_{\ell+1})^{2p}\underset{N\to+\infty}{\longrightarrow}\be[(\Delta^i_{1})^{2p}]=c_p\ga^{2pH}$$
and our claim \eqref{eq:moy_emp_Delta} follows. This finishes the proof.
\end{proof}

With those preliminary considerations on $\Sigma$ in hand, we can now prove our Proposition \ref{prop:control_moment_Z}.

\begin{proof}[Proof of Proposition \ref{prop:control_moment_Z}] As previously, we will prove the result for $2p$ instead of $p$. Moreover, for sake of simplicity, we write $Z^\te$ instead of $Z^{\te,\ga}$. According to the dynamics \eqref{eq:rec_Euler_scheme_O-U} for $\Sigma$ and \eqref{eq:Euler_scheme} for $Z^{\te}$, we have for all $k\geqslant1$,
\begin{multline}\label{eq:rec_norm2_euler-OU}
|Z^\te_{k\gamma}-\Sigma_{k\gamma}|^2
=|Z^\te_{(k-1)\gamma}-\Sigma_{(k-1)\gamma}|^2\\
+2\ga\langle Z^\te_{(k-1)\gamma}-\Sigma_{(k-1)\gamma},\Sigma_{(k-1)\ga}+b_\te(Z^\te_{(k-1)\ga})\rangle +\ga^2|\Sigma_{(k-1)\ga}+b_\te(Z^\te_{(k-1)\ga})|^2.
\end{multline}
In order to treat the second term in \eqref{eq:rec_norm2_euler-OU}, we recast it as 
\begin{align*}
\langle &Z^\te_{(k-1)\ga}-\Sigma_{(k-1)\gamma},\Sigma_{(k-1)\ga}+b_\te(Z^\te_{(k-1)\ga})\rangle\\
&=\langle Z^\te_{(k-1)\ga}-\Sigma_{(k-1)\gamma},b_\te(Z^\te_{(k-1)\ga})-b_\te(\Sigma_{(k-1)\ga})\rangle+\langle Z^\te_{(k-1)\ga}-\Sigma_{(k-1)\gamma},\Sigma_{(k-1)\ga}+b_\te(\Sigma_{(k-1)\ga})\rangle
\end{align*}
Then, we invoke our condition $\HUNW$ to bound the first term in the right hand side above and Young's inequality for the second term. Similar manipulations can be performed for the third term in \eqref{eq:rec_norm2_euler-OU}. We let the patient reader check that for some arbitrary parameters $\varepsilon,\varepsilon'$ we get
\begin{align*}
|Z^\te_{k\gamma}&-\Sigma_{k\gamma}|^2\\
&\leqslant 2\beta\ga+(1-2\ga\alpha)|Z^\te_{(k-1)\gamma}-\Sigma_{(k-1)\gamma}|^2+\frac{\ga}{\varepsilon}|Z^\te_{(k-1)\gamma}-\Sigma_{(k-1)\gamma}|^2\\
&+\ga\varepsilon|\Sigma_{(k-1)\ga}+b_\te(\Sigma_{(k-1)\ga})|^2
+\frac{\ga^2}{2\varepsilon'}|\Sigma_{(k-1)\ga}+b_\te(\Sigma^\te_{(k-1)\ga})|^2+\frac{\ga^2L^2\varepsilon'}{2}|Z^\te_{(k-1)\gamma}-\Sigma_{(k-1)\gamma}|^2.
\end{align*}
We now choose $\varepsilon=\frac{1}{\alpha}$ and $\varepsilon'=\frac{\alpha}{\ga L^2}$, which yields
\begin{align*}
|Z^\te_{k\gamma}-\Sigma_{k\gamma}|^2\leqslant2\beta\ga+\lp1-\frac{\ga\alpha}{2}\rp|Z^\te_{(k-1)\gamma}-\Sigma_{(k-1)\gamma}|^2+\lp\frac{\gamma}{\alpha}+\frac{\ga^3 L^2}{2\alpha}\rp|\Sigma_{(k-1)\ga}+b_\te(\Sigma_{(k-1)\ga})|^2.
\end{align*}
Observe that under $\HUNW$, $b_\te$ is sublinear. Hence there exists $C>0$ depending only on $\alpha,\beta,L$ such that
\begin{equation}\label{eq:induction_distance2_Z_O-U}
|Z^\te_{k\gamma}-\Sigma_{k\gamma}|^2\leqslant\lp1-\ga\tilde{\alpha}\rp|Z^\te_{(k-1)\gamma}-\Sigma_{(k-1)\gamma}|^2+C\ga\lp 1+\ga^2\rp\lp 1+|\Sigma_{(k-1)\ga}|^2\rp
\end{equation}
where we have set $\tilde{\alpha}:=\alpha/2$ in order to ease our next computations. We now choose $0<\ga_0<1/\tilde{\alpha}$. By a direct induction, it comes
\begin{equation*}
|Z^\te_{k\gamma}-\Sigma_{k\gamma}|^2\leqslant C\ga\lp 1+\ga^2\rp \sum_{j=0}^{k-1}(1-\ga\tilde{\alpha})^{k-1-j} \lp 1+|\Sigma_{j\ga}|^2\rp.
\end{equation*}
Finally, applying Jensen's inequality similarly to what is done in \eqref{eq:Jensen_OU}, we end up with
\begin{align}\label{eq:control_between_Z_O-U_scheme}
|Z^\te_{k\gamma}-\Sigma_{k\gamma}|^{2p}&\leqslant C^p\frac{\ga}{\tilde{\alpha}^{p'-1}}\lp 1+\ga^2\rp^p \sum_{j=0}^{k-1}(1-\ga\tilde{\alpha})^{k-1-j} \lp 1+|\Sigma_{j\ga}|^2\rp^p\nonumber\\
&\leqslant C^p\lp\frac{2}{\tilde{\alpha}}\rp^{p-1}\ga\lp 1+\ga^2\rp^p \sum_{j=0}^{k-1}(1-\ga\tilde{\alpha})^{k-1-j} \lp 1+|\Sigma_{j\ga}|^{2p}\rp.
\end{align}
With those preliminary considerations in hand, we now prove the three items in Proposition \ref{prop:control_moment_Z} separately.

$(i)$ We start from relation \eqref{eq:induction_distance2_Z_O-U}. Since $\ga\in(0,\ga_0]$, $C(1+\ga^2)\leq \tilde{C}$ for a constant $\tilde{C}>0$
$$|Z^\te_{k\gamma}-\Sigma_{k\gamma}|^2\leqslant\lp1-\ga\tilde{\alpha}\rp|Z^\te_{(k-1)\gamma}-\Sigma_{(k-1)\gamma}|^2+\ga\tilde{\alpha}\frac{\tilde{C}}{\tilde{\alpha}}\lp 1+|\Sigma_{(k-1)\ga}|^2\rp.$$
We now invoke the convexity of $x\mapsto|x|^p$ in order to get
\begin{align*}
|Z^\te_{k\gamma}-\Sigma_{k\gamma}|^{2p}&\leqslant\lp1-\ga\tilde{\alpha}\rp|Z^\te_{(k-1)\gamma}-\Sigma_{(k-1)\gamma}|^{2p}+\ga\tilde{\alpha}\frac{\tilde{C}^p}{\tilde{\alpha}^p}\lp 1+|\Sigma_{(k-1)\ga}|^2\rp^{p}\\
&\leqslant\lp1-\ga\tilde{\alpha}\rp|Z^\te_{(k-1)\gamma}-\Sigma_{(k-1)\gamma}|^{2p}+\tilde{C}_p\ga\lp 1+|\Sigma_{(k-1)\ga}|^{2p}\rp.
\end{align*}
Then we take expectations and upper limits in $k$, which gives
$$\limsup\limits_{k\to +\infty}~\be[|Z^\te_{k\gamma}-\Sigma_{k\gamma}|^{2p}]\leqslant\frac{\tilde{C}_p}{\tilde{\alpha}}\lp 1+\limsup\limits_{k\to +\infty}~\be[|\Sigma_{k\ga}|^{2p}]\rp.$$
With this inequality in hand and Proposition \ref{prop:control_moment_O-U} $(i)$, inequality $(i)$ is proved.

$(ii)$ First, we take successively the supremum over $\tte$ and the expectation in \eqref{eq:control_between_Z_O-U_scheme}. We then multiply by $\ga^{2p(1-H)}$ and perform the same kind of manipulations as for $(i)$. With the help of Proposition \ref{prop:control_moment_O-U} $(ii)$, we get that 
\begin{equation}\label{eq:cc}
\underset{\gamma\in(0,\gamma_0]}{\sup}~\underset{k\geqslant0}{\sup}~\ga^{2p(1-H)}\be[\underset{\te\in\tte}{\sup}|Z^\te_{k\gamma}-\Sigma_{k\ga}|^{2p}]<+\infty.
\end{equation}
Having achieved the upper bound \eqref{eq:cc}, our claim $(ii)$ now stems from another direct application of Proposition \ref{prop:control_moment_O-U} $(ii)$.

$(iii)$ Here, we just need to sum \eqref{eq:control_between_Z_O-U_scheme} for $k$ from $0$ to $N-1$ and divide by $N$. Then, we use Fubini theorem on the right hand side combined with Proposition \ref{prop:control_moment_O-U} $(iii)$ to conclude that $\underset{\te\in\tte}{\sup}~\underset{N\geqslant1}{\sup}~\frac{1}{N}\sum_{k=0}^{N-1}|Z^\te_{k\gamma}-\Sigma_{k\ga}|^{2p}<+\infty$. Finally, by using again Proposition \ref{prop:control_moment_O-U} $(iii)$, the result follows.

\end{proof}
\subsection{Proof of Proposition \ref{prop:ergodicSDE}}\label{proof:weakconvergenceqsjkl} We will focus on the proof of \eqref{eq:cvgce-occp-measure} only, the discrete counterpart of Proposition \ref{prop:ergodicSDE} being obtained by similar argument. We also note that \eqref{eq:cvgce-occp-measure} cannot be obtained as a consequence of Birkhoff's theorem. Indeed, Birkhoff's theorem would yield a $\bar{\nu}_\te$-a.s. convergence in \eqref{eq:cvgce-occp-measure}, where $\bar{\nu}_\te$ is the invariant measure alluded to in Remark \ref{rem:invariant_dist}. In order to avoid any reference to the support of the invariant measure $\bar{\nu}_\te$, our relation \eqref{eq:cvgce-occp-measure} gives a weak convergence result which does not include any condition on the initial value of $Y^\te_0$. 

In order to prove \eqref{eq:cvgce-occp-measure}, we first observe that according to \eqref{eq:contmomentocccont}, the family of measures $\{\frac{1}{t}\int_0^t \delta_{Y^\te_{s}} ds; t\geq0\}$ is $a.s.$ tight. Therefore it only remains to prove that any limit is $\nu_\te$. We now focus on the convergence of $\frac{1}{t}\int_0^t \delta_{Y^\te_{s}} ds$. To this aim, we consider an additional family of probability measures on ${\cal C}([0,\infty),\ER^d)$ in the following way:
$$\pi^\te_t=\frac{1}{t}\int_0^t \delta_{Y^\te_{s+.}} ds$$
where $Y^\te_{s+.}=(Y^\te_{s+u})_{u\geq0}$.
We will first prove that $\{\pi^\te_t; t\geq0\}$ is an $a.s.$ tight family in the set ${\cal M}_1({\cal C}([0,\infty),\ER^d))$, where we recall that the notation ${\cal M}_1$ is introduced in Section \ref{subsection:notation}. The tightness of $\{\pi^\te_t; t\geq0\}$ can be handled in the following way: we have seen that $\{\frac{1}{t}\int_0^t \delta_{Y^\te_{s}} ds; t\geq0\}$ is $a.s.$ tight in $\mathcal{M}_1(\R^d)$. Therefore a classical criterion (see $e.g.$ \cite[Theorem 8.3]{billingsley}) ensures that $\{\pi^\te_t; t\geq0\}$ is $a.s.$ tight if for every positive $T$, $\eta$ and $\varepsilon$, there exists $\delta>0$ such that for all $t_0\in[0,T]$,
\begin{equation}\label{eq:criterion_tightness}
\limsup_{t\rightarrow+\infty}\frac{1}{t}\int_0^t \frac{1}{\delta}{\bf 1}_{\{\sup_{u\in[t_0,t_0+\delta]} |Y^\te_{s+u}-Y^\te_{s+t_0}|\ge \varepsilon\}} ds\le \eta\quad a.s.
\end{equation}
Moreover, inequality \eqref{eq:criterion_tightness} holds true as long as there exist some positive $r$ and $\rho$ such that 
\begin{equation}\label{tighttoprove}
\limsup_{t\rightarrow+\infty}\frac{1}{t}\int_0^t \sup_{u\in[t_0,t_0+\delta]} |Y^\te_{s+u}-Y^\te_{s+t_0}|^r ds\le C_{r,T}\delta^{1+\rho}\quad a.s.
\end{equation}
Let us now prove \eqref{tighttoprove}. On the interval $[t_0,t_0+\delta]$, we have
\begin{equation}\label{eq:bbb}
|Y^\te_{s+u}-Y^\te_{s+t_0}| \le \int_{t_0}^{t_0+\delta} |b_\te(Y^\te_{s+u})|du+\sup_{u\in [t_0,t_0+\delta]} |B_{s+u}-B_{s+t_0}|.
\end{equation}
Using Jensen's inequality and Fubini's Theorem, we get that for any $r>0$,
$$\frac{1}{t}\int_0^t \left(\int_{t_0}^{t_0+\delta} |b_\te(Y^\te_{s+u})|du\right)^r ds\le \delta^{r-1}\left(\frac{1}{t}\int_{0}^{t+t_0+\delta} |b(Y^\te_s)|^r ds\right).$$
Furthermore, since $b_\te$ is sublinear according to $\HUNW$, it follows from \eqref{eq:contmomentocc2}  that
\begin{equation}\label{eq:ccc}
\limsup_{t\rightarrow+\infty}\frac{1}{t}\int_0^t \left(\int_{t_0}^{t_0+\delta} |b_\te(Y^\te_{s+u})|du\right)^r ds\leq C\delta^{r-1}
\end{equation}
for a constant $C>0$.
On the other hand, by the ergodicity of the increments of the 
fBm and its self-similarity, 
\begin{equation}\label{eq:ddd}
\frac{1}{t}\int_0^t\sup_{u\in [t_0,t_0+\delta]} |B_{s+u}-B_{s+t_0}|^r ds\xrightarrow{t\rightarrow+\infty} \ES[\sup_{u\in[0,\delta]} |B_u|^r]=C_r \delta^{Hr},
\end{equation}
where the limit in \eqref{eq:ddd} holds in the $a.s.$ sense.
Choosing $r>\sup\{2,1/H\}$ in \eqref{eq:ccc} and \eqref{eq:ddd}, then plugging \eqref{eq:ccc} and \eqref{eq:ddd} into \eqref{eq:bbb}, we get that \eqref{tighttoprove} is satisfied. We have thus proved the $a.s.$ tightness of $\{\pi^\te_t; t\geq0\}$.\\
The second step is then to show that any limiting distribution  of $\left(\frac{1}{t}\int_0^t \delta_{Y^\te_{s+.}} ds\right)$ is necessarily the law of a stationary solution $\bar{Y}^\te$ to SDE  \eqref{eq:sdetheta}.  This step in turn implies the result by uniqueness of the stationary solutions.\\

Let $\left(\frac{1}{t_n}\int_0^{t_n} \delta_{Y^\te_{s+.}} ds\right)_n$ be a (pathwise) convergent sequence with limiting distribution $\mu$ where $\{t_n; n\geq1\}$ is an increasing sequence converging to $+\infty$.
We first prove that $\mu$ is the law of a stationary process. Namely, for any bounded functional $F:{\cal C}([0,\infty),\ER^d)\rightarrow\ER$, we have 
 $$\mu(F\circ \theta_T)-\mu(F)=\lim_{n\rightarrow+\infty}\frac{1}{t_n}\int_0^{t_n} [F( Y^\te_{s+T+.}) -F(Y^\te_{s+.})] ds.$$
However, a simple change of variables reveals that $a.s.$ we have
$$\lim\limits_{n\to+\infty}\frac{1}{t_n}\int_0^{t_n} [F( Y^\te_{s+T+.})-F(Y^\te_{s+.})] ds=\lim\limits_{n\to+\infty}\frac{1}{t_n}\left(\int_{t_n}^{t_n+T} F(Y^\te_{s+.}) ds-\int_{0}^{T} F(Y^\te_{s+.}) ds\right)=0.$$
We thus easily get that $\mu$ is stationary.
Now, let us prove that $\mu$ is the law of a solution to \eqref{eq:sdetheta}. Without loss of generality, we can say that a process $(x_t)_{t\ge0}$ is a solution to \eqref{eq:sdetheta} if $x_.-x_0-\int_0^. b_\te(x_u) du$ is a fBm. In other words, we have to prove that 
$\mu\circ G^{-1}$ is the law of a fBm where $G(x)=x_.-x_0-\int_0^. b_\te(x_u) du.$ Since $G$ is continuous for the $u.s.c$ topology, it is readily checked that
$$\mu\circ G^{-1}=\lim_{n\rightarrow+\infty}\frac{1}{t_n}\int_0^{t_n} \delta_{G(Y^\te_{s+.})} ds.$$
In addition, by construction $G(Y^\te_{s+.})=B_{s+.}-B_s$. Hence the fact that $\mu\circ G^{-1}$ is the law of a fBM follows again from  the ergodicity of the increments of the fBM. \\
Summarizing our considerations, we have proved that $\mu$ is a stationary measure related to the system \eqref{eq:sdetheta}. Therefore we have $\mu=\mathcal{L}((\bar{Y}^\te_{t})_{t\geq0})$, which concludes the proof.

\section{Proof of Proposition \ref{prop:contracttt}}\label{append:B}
For sake of conciseness, we will focus on the proof of Proposition \ref{prop:contracttt} $(i)$. The proof of item $(ii)$ relies on the same kind of tools, plus the discrete computations invoked in the proof of Proposition \ref{prop:control_moment_Z}.
In order to prove item $(i)$, let us consider $t\geq1$ and a parameter $\rho>0$ to be chosen later on. An easy elaboration of \eqref{c1} shows that
\begin{equation}\label{eq:Y_te1-Y_te2}
e^{\rho t}|Y_t^{\te_1}-Y_t^{\te_2}|^2=\int_0^t e^{\rho s}\left(\rho|Y_s^{\te_1}-Y_s^{\te_2}|^2+2\langle b_{\te_1}(Y_s^{\te_1})-b_{\te_2}(Y_s^{\te_2}),Y_s^{\te_1}-Y_s^{\te_2}\rangle\right) ds.
\end{equation}
In addition, one can write
$$ b_{\te_1}(Y_s^{\te_1})-b_{\te_2}(Y_s^{\te_2})=b_{\te_1}(Y_s^{\te_1})-b_{\te_1}(Y_s^{\te_2})+b_{\te_1}(Y_s^{\te_2})-b_{\te_2}(Y_s^{\te_2}).$$
We now combine the assumption $\HUNS$ (including the contraction property, the fact $x\mapsto b_\te(x)$ is uniformly Lipschitz in $\te$ and the fact that $\partial_\te b_\te(x)$ has polynomial growth) and Young's inequality 
$|ab|\le \frac{1}{2\varepsilon}|a|^2+\frac{\varepsilon}{2}|b|^2$ for an arbitrary $\varepsilon>0$. This yields the existence of a constant $L>0$ such that
$$\langle b_{\te_1}(Y_s^{\te_1})-b_{\te_2}(Y_s^{\te_2}),Y_s^{\te_1}-Y_s^{\te_2}\rangle\le \left(-\alpha+\frac{L^2\varepsilon}{2}\right) |Y_s^{\te_1}-Y_s^{\te_2}|^2+\frac{C|\te_1-\te_2|^2(1+ |Y_s|^r)^2}{2\varepsilon}.$$
Plugging this inequality into \eqref{eq:Y_te1-Y_te2} and setting $\varepsilon=L^2/\alpha$ and $\rho=\alpha/2$, we have thus obtained
\begin{equation}\label{eq:carre}
e^{\rho t}|Y_t^{\te_1}-Y_t^{\te_2}|^2\le C|\te_1-\te_2|^2\int_0^t e^{\rho s} (1+|Y_s^{\te_2}|^{2r} ds. 
\end{equation}
Thus, using Fubini's theorem, one deduces that
\begin{align*}
\frac{1}{t}\int_0^t |Y_s^{\te_1}-Y_s^{\te_2}|^2 ds&\le \frac{C|\te_1-\te_2|^2}{t}\int_0^{t}(1+|Y_s^{\te_2}|^{2r}\int_u^{t} e^{\rho (s-u)} ds du\\
&\le {C_\rho |\te_1-\te_2|^2}\left(1+\frac{1}{t}\int_0^{t}(1+|Y_s^{\te_2}|^{2r} du\right).
\end{align*}
Hence a direct application of inequality \eqref{eq:contmomentocccont} yields the existence of a random variable $C=C(\omega)$ such that for all $(\te_1,\te_2)\in\tte^2$ we have
$$\sup_{t\ge1}\frac{1}{t}\int_0^t |Y_s^{\te_1}-Y_s^{\te_2}|^2 ds\le C|\te_1-\te_2|^2.$$
This concludes the proof of Proposition \ref{prop:contracttt} $(i)$ for a distance $d\in {\cal D}_2$. To extend the result to any $p\ge2$, one can apply Jensen's inequality to \eqref{eq:carre} and follow the same lines as for $d\in{\cal D}_2$.

\section*{Acknowledgements}
The authors would like to thank the anonymous referee for his careful reading and his
suggestions, which substantially improved the quality of the paper.

\end{document}